\documentclass[a4paper]{article}

\usepackage{microtype}
\usepackage{amscd}
\usepackage{amsmath,amssymb,amsfonts, accents}
\usepackage{xstring}
\usepackage{xcolor}
\usepackage[mathscr]{euscript}
\usepackage{xypic}
\xyoption{rotate}
\usepackage[cmtip,all,2cell]{xy}
\entrymodifiers={+!!<0pt,\fontdimen22\textfont2>}

\usepackage{tikz-cd}
 \usetikzlibrary{decorations.pathmorphing}

\usepackage{url}
\usepackage{latexsym}
\usepackage{amsthm}
\usepackage[latin1]{inputenc}
\usepackage{multicol}
\usepackage{enumitem}
\usepackage[colorlinks=true, linkcolor={blue!50!black}, pdfhighlight=/O, ocgcolorlinks=true]{hyperref}
\usepackage{graphicx}
\usepackage{color}
\usepackage{mathtools}
\usepackage[colorinlistoftodos]{todonotes}

\usepackage[affil-it]{authblk}

\usepackage[a4paper, total={5.5in,8.5in}]{geometry}

\numberwithin{equation}{section}
\theoremstyle{plain}
\newtheorem{theorem}[equation]{Theorem}
\newtheorem{lemma}[equation]{Lemma}
\newtheorem{proposition}[equation]{Proposition}
\newtheorem{corollary}[equation]{Corollary}

\theoremstyle{definition}
\newtheorem{assumption}[equation]{Assumption}
\newtheorem{definition}[equation]{Definition}
\newtheorem{construction}[equation]{Construction}
\newtheorem{example}[equation]{Example}
\newtheorem{remark}[equation]{Remark}
\newtheorem{variant}[equation]{Variant}
\newtheorem{observation}[equation]{Observation}

\setlist[enumerate]{label=(\arabic*), leftmargin=*}
\setenumerate{label=(\arabic*), leftmargin=*}
\setlist[itemize]{label=$\vcenter{\hbox{\footnotesize$\bullet$}}$, leftmargin=*}
\setitemize{label=$\vcenter{\hbox{\footnotesize$\bullet$}}$, leftmargin=*}

\newcommand{\mb}[1]{\mathbf{#1}}
\newcommand{\mm}[1]{\mathrm{#1}}
\newcommand{\mf}[1]{\mathfrak{#1}}
\newcommand{\mc}[1]{\mathcal{#1}}
\newcommand{\ul}[1]{\underline{#1}}
\newcommand{\ol}[1]{\overline{#1}}

\newcommand{\cat}[1]{
\StrLen{#1}[\mystrlen]
\ifnum\mystrlen=1 \mathscr{#1}
\else \mathrm{#1}
\fi}
\newcommand{\scat}[1]{\mb{#1}}

\newcommand*\cocolon{% colon for maps appearing on the right
        \nobreak
        \mskip6mu plus1mu
        \mathpunct{}%
        \nonscript
        \mkern-\thinmuskip
        {:}%
        \mskip2mu
        \relax
}

\newcommand{\sS}[0]{\cat{S}}
\newcommand{\colim}{\operatornamewithlimits{\mathrm{colim}}}
\newcommand{\hocolim}{\operatornamewithlimits{\mathrm{hocolim}}}
\newcommand{\Hom}[0]{\mm{Hom}}
\newcommand{\Map}[0]{\mm{Map}}
\renewcommand{\rto}[1]{\stackrel{#1}{\rt}}
\renewcommand{\lto}[2][\;]{\xleftarrow[#1]{#2}}
\newcommand{\rt}[0]{\longrightarrow}

\newcommand{\drt}[0]{\mathrel{
  \rt\mkern-19mu{\color{white}\boldsymbol{\cdot}}\mkern-15mu{\color{white}\boldsymbol{\cdot}}\mkern17mu
}}

\def\dar{\ar@{*{\hspace{-1pt}}*{\hspace{1pt}\text{-}}>}}
\def\ddar{\ar@{*[left]{\hspace{-1pt}}*[left]+<-4.5pt, 0pt>{\hspace{1pt}\text{-}}>}}

\newcommand{\Fun}[0]{\cat{Fun}}

\newcommand{\op}[0]{\mm{op}}
\newcommand{\dg}[0]{\mm{dg}}

\renewcommand{\Bar}{\mathrm{B}}

\DeclareMathOperator{\End}{End}
\DeclareMathOperator{\Conv}{Conv}
\DeclareMathOperator{\Def}{Def}
\DeclareMathOperator{\Tot}{Tot}
\DeclareMathOperator{\Tw}{Tw}

\newcommand{\base}[0]{\Bbbk} %Ricardo: I cannot compile this :(

\newcommand{\Alg}[0]{\cat{Alg}}
\newcommand{\Mod}[0]{\cat{Mod}}
\newcommand{\MC}[0]{\mm{MC}}
\newcommand{\Koszul}[0]{\cat{Koszul}}
\newcommand{\CC}[0]{\mathscr{C}}
\newcommand{\DD}[0]{\mathscr{D}}
\newcommand{\PP}[0]{\mathscr{P}}
\newcommand{\QQ}[0]{\mathscr{Q}}
\newcommand{\ee}[0]{\mathsf{e}}
\newcommand{\triv}[2]{\base_{#1}[{#2}]}
\newcommand{\trivc}[0]{\base_c}

\newcommand{\trivop}[2]{\base^{\op}_{#1}[{#2}]}
\newcommand{\FMP}[0]{\cat{FMP}}
\newcommand{\Der}[0]{\mm{Der}}
\newcommand{\Free}[0]{\mm{Free}}
\newcommand{\antishriek}{{\scriptstyle \text{\rm !`}}}

\newcommand{\wt}[0]{\mm{wt}}

\newcommand{\nonsym}{\base^\mm{ns}}
\newcommand{\Art}[0]{\cat{Art}}
\newcommand{\ArtOp}[0]{\cat{Art}_{\mm{Op}}}

\newcommand{\Com}[0]{\mathsf{Com}}
\newcommand{\Perm}[0]{\mathsf{Perm}}
\newcommand{\As}[0]{\mathsf{As}}
\newcommand{\Lie}[0]{\mathsf{Lie}}
\newcommand{\preLie}[0]{\mathsf{preLie}}

\newcommand{\Pois}[0]{\mathsf{Pois}}
\newcommand{\BD}[0]{\mathsf{BD}}

\title{Moduli problems for operadic algebras}

\author[]{Damien Calaque\thanks{\href{mailto:damien.calaque@umontpellier.fr}{damien.calaque@umontpellier.fr}}}
\author[]{Ricardo Campos\thanks{\href{mailto:ricardo.campos@umontpellier.fr}{ricardo.campos@umontpellier.fr}}}
\author[]{Joost Nuiten \thanks{\href{mailto:joost.nuiten@math.univ-toulouse.fr}{joost.nuiten@math.univ-toulouse.fr}}}
\affil[]{IMAG, Univ. Montpellier, CNRS, Montpellier, France.}

\begin{document}

\maketitle
\begin{abstract}
A theorem of Pridham and Lurie provides an equivalence between formal moduli problems and Lie algebras in characteristic zero. We prove a generalization of this correspondence, relating formal moduli problems parametrized by algebras over a Koszul operad to algebras over its Koszul dual operad. In particular, when the Lie algebra associated to a deformation problem is induced from a pre-Lie structure it corresponds to a permutative formal moduli problem. As another example we obtain a correspondence between operadic formal moduli problems and augmented operads.
\end{abstract}

\tableofcontents

\section{Introduction}
A classical heuristic in deformation theory asserts that the infinitesimal deformations of an algebro-geometric object over a field $k$ of characteristic zero are controlled by a differential graded Lie algebra. A first instance of this can already be found in the work of Kodaira--Spencer on deformations of complex manifolds \cite{KodSpe}; its recognition as a key principle in deformation theory traces back to ideas of Deligne and Drinfeld. These ideas have been further developed in the work of various authors \cite{GM, hinichformal, Kontsevich2003, Manetti}, leading to a precise mathematical formulation of the above heuristic as an equivalence of categories between deformation problems and dg-Lie algebras \cite{Pridham, DAGX}.

More precisely, following work of Schlessinger \cite{Schlessinger}, one can describe the infinitesimal deformations of an algebro-geometric object $X$ over $k$ by a functor
$$\begin{tikzcd}
\mm{def}_X\colon \cat{Ring}^{\mm{art}}_k\arrow[r] & \cat{Set}
\end{tikzcd}$$
from the category of (commutative) Artin local $k$-algebras with residue field $k$. This functor sends each Artin local $k$-algebra $A$ to the set of deformations of $X$ over $A$.
The aforementioned works have led to two modifications of this idea. 

First, the deformations of $X$ typically have automorphisms and homotopies between them, leading to the study of deformation functors with values in \emph{spaces} or simplicial sets. Second, it has been observed that the deformation theory of an object $X$ usually comes with an additional obstruction theory, which is not encoded by the functor $\mm{def}_X$. A key idea, tracing back to Drinfeld, is to incorporate such an obstruction theory by extending $\mm{def}_X$ to the category of dg-Artin local $k$-algebras. One is therefore led to contemplate deformation functors
$$\begin{tikzcd}
\mm{Def}_X\colon \cat{CAlg}_k^\mm{art}\arrow[r] & \sS
\end{tikzcd}$$
from the $\infty$-category of (connective) dg-Artin local $k$-algebras to the $\infty$-category of spaces. Such deformation problems satisfy a variant of the Schlessinger conditions: their value on $k$ is contractible and they preserve fiber products along maps inducing a surjection on $H^0$ (see Section \ref{sec:fmps} for more details). Following Lurie \cite{DAGX}, we will refer to such functors as \emph{formal moduli problems}. The work of Lurie \cite{DAGX} and Pridham \cite{Pridham} now provides an equivalence of $\infty$-categories
$$\begin{tikzcd}
\cat{FMP}_k\arrow[r, "\sim"] & \cat{Lie}_k
\end{tikzcd}$$
between formal moduli problems and differential graded Lie algebras over $k$.

The equivalence between Lie algebras and formal moduli problems indexed by commutative algebras can be viewed as a manifestation of the \emph{Koszul duality} between the commutative operad and the Lie operad. In fact, there is a similar equivalence between associative algebras and formal moduli problems indexed by associative algebras \cite{DAGX}, which can be thought of as an incarnation of the Koszul self-duality of the associative operad. These two equivalences are related in a natural way: if a Lie algebra arises from an associative algebra by taking the commutator bracket, then the corresponding commutative formal moduli problem is the restriction of an associative formal moduli problem.

\subsection*{Statement of results}
The purpose of this paper is to generalize the above results to more general pairs of Koszul dual operads over a field of characteristic zero. More precisely, for any augmented operad $\PP$ one can define an $\infty$-category $\Art_{\PP}$ of \emph{Artin} $\PP$-algebras. A $\PP$-algebraic formal moduli problem is then given by a functor
$$\begin{tikzcd}
F\colon \Art_{\PP}\arrow[r] & \sS
\end{tikzcd}$$
satisfying a natural analogue of the Schlessinger conditions (see Section \ref{sec:fmps} for more details). We denote the $\infty$-category of such functors by $\cat{FMP}_{\PP}$. When $\PP$ is a Koszul binary quadratic operad, we prove that such $\PP$-algebraic formal moduli problems can be classified by algebras over its Koszul dual operad:
\begin{theorem}\label{thm:mainthmkoszul}
Let $k$ be a field of characteristic zero and consider:
\begin{itemize}
\item a Koszul binary quadratic operad $\mathscr{P}$ in nonpositive cohomological degrees.
\item its Koszul dual operad $\mathscr{P}^!$.
\end{itemize}
Then there is an equivalence of $\infty$-categories
$$\begin{tikzcd}
\cat{FMP}_{\mathscr{P}}\arrow[r] & \cat{Alg}_{\mathscr{P}^!}; \hspace{4pt} F\arrow[r, mapsto] & T(F)[-1].
\end{tikzcd}$$
Here $T(F)$ denotes the \emph{tangent complex} of the formal moduli problem, as defined by Lurie \cite{DAGX} (see also Definition \ref{def:tangent}).
\end{theorem}
This recovers the aforementioned results of Lurie and Pridham, taking $\PP$ to be the commutative operad, whose Koszul dual is the Lie operad, or the associative operad. It also applies to many other Koszul dual pairs of algebraic operads (see Section \ref{sec:examples}). 
In fact, allusions to the role of Koszul duality in such a correspondence have appeared before, notably in \cite{KharkovtoMoscow}, \cite[Lecture 15]{kontsevichlecturenotes} and in \cite{pridham2013constructing}.
For example, taking $\PP$ to be the permutative operad, whose Koszul dual is the pre-Lie operad \cite{ChaLiv}, we obtain a classification of permutative formal moduli problems in terms of    pre-Lie algebras. Such    pre-Lie algebras indeed appear naturally in the deformation theory of operadic algebras, see the work of Dotsenko, Shadrin and Vallette \cite{dotsenko2015pre} (in fact, this was the original motivation for the present paper). From the point of view of deformation theory, a Lie algebra underlies a    pre-Lie algebra structure whenever the corresponding commutative formal moduli problem is the restriction of a permutative formal moduli problem.

\medskip

Our proof of Theorem \ref{thm:mainthmkoszul} will make little use of the Koszul property of $\mathscr{P}$: the Koszul property mainly serves to guarantee that the operad $\mathscr{P}$ admits a resolution with good properties. More precisely, we will deduce Theorem \ref{thm:mainthmkoszul} from a statement about algebras over the dual of the bar construction of an augmented dg-operad. In fact, it will be convenient to work in a slightly more general setting:
\begin{enumerate}[label=(\alph*)]
\item We work with \emph{coloured} dg-operads.
\item Instead of taking dg-operads over the base field, we will consider operads which are augmented over a connective \emph{dg-algebra} or, in the coloured case, over a \emph{connective dg-category} $\base$ (here connective means that the cohomology groups are concentrated in nonpositive degrees). More precisely, we will consider coloured dg-operads $\mathscr{P}$ which fit into a retract diagram of operads
$$\begin{tikzcd}
\base\arrow[r] & \mathscr{P}\arrow[r] & \base.
\end{tikzcd}$$
Given a dg-category $\base$, we will refer to such objects as (augmented) \emph{$\base$-operads}.
\end{enumerate}

For example, one can take $\base$ to be a discrete dg-category with finitely many objects and only zero maps between them, or (Morita equivalently, cf.\ Remark \ref{rem:compact gen}) the semisimple associative algebra $\base=k^{\times n}$. In this case, deformations parametrized by Artin associative algebras \emph{relative to} $\base$ correspond to \emph{multi-pointed deformations}, as frequently considered in noncommutative geometry \cite{laudal2002noncommutativedeformations}.

The usual operadic homological algebra (as in \cite{LodayVallette2012}, for example) has an analogue for augmented $\base$-operads; Appendix \ref{sec:operads} provides all the results and definitions that we will need. In particular, every (augmented) $\base$-operad $\PP$ has a dual $\base^\op$-operad, given somewhat informally as
$$
\mf{D}(\PP)=\mf{D}_{\base}(\PP)=\mm{Ext}_{\PP}(\base, \base).
$$
More precisely, we can make the following definition:
\begin{definition}\label{def:dualoperad}
Let $\base$ be a dg-category and let $\mathscr{P}$ be a (augmented) $\base$-operad, which we assume to be cofibrant as a left $\base$-module throughout this introduction.
We define the dual operad to be the $\base$-linear dual of the bar construction of $\PP$ \emph{over} $\base$
$$ 
\mf{D}_{\base}(\mathscr{P})\coloneqq \Bar_{\base}(\mathscr{P})^\vee.
$$
The $\base^{\op}$-operad structure arises from the $\base$-cooperad structure on the bar construction. 
Up to a degree shift, this corresponds to the dual operad introduced in \cite{ginzburg1994koszul}, see Section \ref{sec:barcobar} for more details.
\end{definition}
With these definitions, our main result is then the following:
\begin{theorem}\label{thm:mainthm}
Let $\base$ be a dg-category over $\mathbb{Q}$ and let $\mathscr{P}$ be an augmented $\base$-operad. Suppose that the following conditions are satisfied:
\begin{enumerate}
\item $\base$ and $\mathscr{P}$ are both connective, i.e.\ their cohomology is concentrated in nonpositive degrees.
\item $\base$ is cohomologically bounded, i.e.\ there exists an $N\in\mathbb{N}$ such that all $H^*(\base)(c, d)$ are concentrated in degrees $[-N, 0]$.
\item\label{it:splendidness} The derived relative composition product
$$
\mathscr{P}(1)\circ^h_{\mathscr{P}^{\geq 1}} \mathscr{P}(1)
$$
is concentrated in increasingly negative cohomological degrees as the arity increases (cf.\ Definition \ref{def:splendid}).
\end{enumerate}
Then there is an equivalence of $\infty$-categories
\begin{equation}\label{diag:mainequivalence}
\begin{tikzcd}
\cat{FMP}_{\mathscr{P}}\arrow[r] & \cat{Alg}_{\mf{D}_{\base}(\mathscr{P})}; \hspace{4pt} F\arrow[r, mapsto] & T(F)
\end{tikzcd}\end{equation}
where $T(F)$ denotes the tangent complex (Definition \ref{def:tangent}). Furthermore, this equivalence is natural in $\PP$.
\end{theorem}
We will denote the inverse equivalence \eqref{diag:mainequivalence} by
$$\begin{tikzcd}
\MC\colon \Alg_{\mf{D}_{\base}(\mathscr{P})}\arrow[r] & \cat{FMP}_{\mathscr{P}}
\end{tikzcd}$$
and think of it as sending a $\mf{D}(\PP)$-algebra to the `formal $\mathscr{P}$-algebraic stack of solutions to the Maurer--Cartan equation'. Some justification for this terminology is provided in Section \ref{sec:Dprime}, where we show that for various operads $\PP$, this inverse functor does indeed admit a description in terms of Maurer--Cartan elements of dg-Lie algebras, resembling the construction by Hinich \cite{hinichformal} in the commutative case.
This is a by-product of our proof of Theorem \ref{thm:mainthm}, which relies on a careful analysis of the adjoint pair
\begin{equation}\label{diag:mainadjunction}
\begin{tikzcd}
\mf{D}\colon \cat{Alg}_{\PP}\arrow[r, yshift=0.8ex] & \cat{Alg}^{\op}_{\mf{D}(\PP)}\colon \mf{D}'.\arrow[l, yshift=-0.8ex]
\end{tikzcd}\end{equation}
Here $\mf{D}$ sends a $\PP$-algebra to the $\base$-linear dual of its operadic bar construction. We point out a slight difference from the arguments of Lurie: when $\PP$ is the commutative operad, we study the behaviour of the functor $\mf{D}$ (the Harrison complex) instead of the functor $\mf{D}'$ (the Chevalley--Eilenberg complex). An adjunction between the Harrison and Chevalley--Eilenberg complex also appears in the arguments from \cite{Pridham, guan2020reviewdeformations}; here the $\infty$-category of commutative algebras is replaced by a certain model category of pro-artinian algebras (with the effect of making the Harrison complex a right adjoint detecting equivalences, in contrast to \eqref{diag:mainadjunction}).

The conditions of Theorem \ref{thm:mainthm} hold for Koszul binary quadratic operads, leading to the following proof of Theorem \ref{thm:mainthmkoszul}:
\begin{proof}[Proof of Theorem \ref{thm:mainthmkoszul} (from Theorem \ref{thm:mainthm})]
The Koszul property of $\mathscr{P}$ asserts that there are weak equivalences of operads (using curly brackets to denote degree shift)
\[\begin{tikzcd}
\Omega\mathscr{P}^{\text{!`}}\arrow[r, "\sim"] & \mathscr{P} & & \mf{D}(\mathscr{P})\arrow[r, "\sim"] & \mathscr{P}^!\{-1\}.
\end{tikzcd}\]
Since $\mathscr{P}$ is generated by binary operations in degrees $\leq 0$, the quadratic dual cooperad $\mathscr{P}^{\text{!`}}$ is generated by binary operations in degrees $\leq -1$. It follows that the generators of $\Omega\mathscr{P}^{\text{!`}}$ are concentrated in increasingly negative degrees as the arity increases. The operad $\mathscr{P}$ then satisfies the conditions of Theorem \ref{thm:mainthm} (condition \emph{(3)} follows from Corollary \ref{cor:bar of operad as derived comp}), and the sequence of equivalences
$$\begin{tikzcd}[column sep=3pc]
\cat{FMP}_{\mathscr{P}}\arrow[r, "F\mapsto T(F)"] & \cat{Alg}_{\mf{D}(\mathscr{P})} & \cat{Alg}_{\mathscr{P}^!\{-1\}}\arrow[l, "\sim"{swap}] \arrow[r, "{V\mapsto V[-1]}"] & \cat{Alg}_{\mathscr{P}^!}
\end{tikzcd}$$
provides the desired result.
\end{proof}

Suppose that $\mathscr{P}$ is an augmented $\base$-operad arising as the bar dual of a (sufficiently nice) $\base^{\op}$-operad. Theorem \ref{thm:mainthm} then gives an interpretation of the $\infty$-category $\cat{Alg}_{\mathscr{P}}$ in terms of formal moduli problems. 
One may wonder if there is a similar interpretation of the $\infty$-category of algebras over an \emph{arbitrary} augmented $\base$-operad $\mathscr{P}$. Somewhat informally, one expects a $\mathscr{P}$-algebra to correspond to some homotopy-theoretic, or geometric, analogue of a `conilpotent coalgebra over a conilpotent cooperad'. Theorem \ref{thm:mainthm} precisely provides us with a geometric way to think about this conilpotent cooperad, as a formal moduli problem
$$\begin{tikzcd}
\mm{MC}_{\mathscr{P}}\colon \ArtOp\arrow[r] & \sS
\end{tikzcd}$$ 
on the category of Artin (i.e.\ nilpotent, finite-dimensional) operads. This is discussed in more detail in Section \ref{sec:relativekoszul} and relies on the fact that the operad for nonunital symmetric operads is Koszul self-dual, \emph{relative} to the dg-category of finite sets and bijections ($k$-linearizing all sets of maps). One may informally think of the functor $\mm{MC}_{\mathscr{P}}$ as encoding a family of finite-dimensional nilpotent operads, corresponding to the family of linear duals of the finite-dimensional conilpotent sub-cooperads of the conilpotent cooperad $\Bar\PP$.

Given a formal moduli problem $X\colon \ArtOp\rt \sS$, there is a natural notion of formal moduli problem over $X$. Indeed, in a similar way as one usually defines quasi-coherent sheaves on moduli functors in algebraic geometry, one can define a formal moduli problem over $X$ to consist of the following data:
\begin{enumerate}
\item a $\QQ$-algebraic formal moduli problem $\mc{F}_x\in \cat{FMP}_{\QQ}$ for every $x\in X(\QQ)$.
\item for every map $f\colon \QQ\rt \QQ'$ and every $x\in X(\QQ)$, an equivalence
$$\begin{tikzcd}
\mc{F}_{f_*(x)}\arrow[r, "\sim"] & f_*\mc{F}_x
\end{tikzcd}$$
together with coherence data between them (see Definition \ref{def:fmpoverfmp}). Here $f_*\mc{F}_x$ denotes the restriction of $\mc{F}_x$ along the forgetful functor $\cat{Art}_{\QQ'}\rt \cat{Art}_{\QQ}$.
\end{enumerate}
Informally, these formal moduli problems can be thought of as geometric analogues of conilpotent coalgebras over conilpotent cooperads. Indeed, a formal moduli problem over $X$ describes a coherent collection of finite-dimensional nilpotent algebras over finite-dimensional nilpotent operads. This roughly corresponds to the collection of linear duals of the finite-dimensional conilpotent sub-coalgebras of a conilpotent coalgebra. 

We then have the following result:
\begin{theorem}\label{thm:metaduality}
Let $\mathscr{P}$ be a 1-coloured augmented operad. Then there is an equivalence of $\infty$-categories
$$\begin{tikzcd}
\cat{FMP}_{\mm{MC}_{\mathscr{P}}}\arrow[r, "\sim"] & \cat{Alg}_{\mathscr{P}}; \hspace{4pt} F\arrow[r, mapsto] & T(F).
\end{tikzcd}$$
\end{theorem}
One may consider this as a geometric, or $\infty$-categorical, version of the relation between algebras over $\PP$ and conilpotent coalgebras over the conilpotent cooperad $\Bar\PP$ \cite{vallette2014homotopy}.

Finally, let us point out that Brantner and Mathew \cite{brantner2019deformation} have recently established that in positive characteristic, formal moduli problems do not correspond to dg Lie algebras but rather to \emph{partition} Lie algebras. Their result can also be interpreted as a refinement of Koszul duality following \cite[Example 1.6]{brantner2021pd}.

\subsection{Outline and how to read the paper}

Let us briefly describe the structure of the rest of the paper. 
In Section \ref{sec:fmps} we introduce the main definitions concerning formal moduli problems parametrized by algebras over operads.

\medskip

In Section \ref{sec:examples}, we will discuss various (non-)examples and special cases of our main theorem; these include many of the well-known operads. This section essentially only makes reference to the \emph{statement} of the main theorem relating operadic moduli problems with algebras over the dual operad (Theorem \ref{thm:mainthm}), or rather to a slightly more precise formulation thereof (Theorem \ref{thm:biduals}). Taking this for granted, Section \ref{sec:examples} essentially only assumes some familiarity with operadic homological algebra. For operads over a field $k$ of characteristic $0$, all operadic results we use are classical and contained for instance in \cite{ginzburg1994koszul, LodayVallette2012}. 

Some additional techniques are required in Section \ref{sec:operadicdeftheory}, where we treat deformations parametrized by operads. Since these are themselves algebras over a coloured operad, this involves a version of the usual operadic homological algebra relative to a dg-category $\base$. This is developed in Appendix \ref{sec:operads}, with a few more specific results used in Section \ref{sec:operadicdeftheory} appearing in Section \ref{sec:relativekoszul}. For readability, we have tried to put the tools from Appendix \ref{sec:operads} in the background; as a rule of thumb, the reader may think of $\base$ as a dg-algebra and suppose that all results that hold for classical $k$-operads will also hold for $\base$-operads and their algebras, as long as the corresponding modules are $\base$-cofibrant. 
In particular, the reader only interested in moduli problems for algebras over $k$-operads can safely ignore Appendix \ref{sec:operads}. 

\medskip

In Section \ref{sec:operadicfmps} we explain how to associate to an operadic formal moduli problem an algebra which requires establishing the adjoint pair \eqref{diag:mainadjunction} (see Theorem \ref{thm:DprimeisD}) and proceed to describe the general framework that that goes into our proof of Theorem \ref{thm:mainthm}.

\medskip

Section \ref{sec:cohsmall} is the technical heart of the proof (Theorem \ref{thm:biduals}); in this section, we verify the technical hypotheses that allow us to apply the axiomatic argument (Theorem \ref{thm:axiomatic}) described in Section \ref{sec:outline}. In Section \ref{sec:naturality}, we discuss the naturality of the equivalence \eqref{diag:mainequivalence} in the operad $\PP$ and apply this to deduce Theorem \ref{thm:metaduality} in Section \ref{sec:fmpsoverfmps}. Under certain conditions on the operad $\PP$, we describe this equivalence more concretely in terms of Maurer--Cartan elements in Section \ref{sec:Dprime} (see Theorem \ref{thm:MC functor}).

\medskip

Finally, Section \ref{sec:relativekoszul} contains some further remarks about Koszul duality relative to a base $\base$. In particular, we use this to spell out some leftover proofs from Section \ref{sec:operadicdeftheory}; in particular, we show that the operad for operads is, in a relative sense, Koszul self-dual.

\subsubsection*{Acknowledgements}
We are grateful to Ben Ward for discussions about the    pre-Lie structure on the totalisation of an operad and to Joan Bellier-Mill\`es, Lukas Brantner, Benjamin Hennion and Bertrand To\"en for their comments on a preliminary version of this paper. We are also grateful to the anonymous referees for multiple comments, references and suggestions for further examples, and for pointing out mistakes.

The third author thanks the Max Planck Institute for Mathematics for their hospitality.
The first and third author have received funding from the European Research Council (ERC) under the
European Union's Horizon 2020 research and innovation programme (Grant Agreement No.\ 768679).

\subsection{Conventions}\label{sec:conventions}

Throughout, we work over a field $k$ of characteristic zero\footnote{In fact, everything we do also works over an arbitrary ring, instead of a field of characteristic zero, 
	if one restricts to \emph{nonsymmetric} operads.} and all objects involved are differential graded (with differentials of degree +1), 
even if this is not said explicitly.

Given $k\in \mathbb Z$ and a graded vector space $V$ we denote by $V[k]$ its degree shift satisfying $(V[k])^d = V^{k+d}$.
\medskip

\noindent\textbf{Model categories and $\infty$-categories.} Since certain functors are only defined at the level 
of $\infty$-categories, we will need to distinguish between model categories or relative categories, and the 
$\infty$-categories obtained from them by inverting the weak equivalences. We will employ the following basic 
convention: we will denote by $\cat{C}^{\dg}$ a certain category of dg-objects, e.g.\ operads or algebras over 
them, and by $\cat{C}$ the underlying $\infty$-category. For example:
$$
\cat{Alg}_{\mathscr{P}}^{\dg}=\big\{\text{dg-algebras over }\mathscr{P}\big\} 
\qquad \qquad 
\cat{Alg}_{\mathscr{P}}=\cat{Alg}_{\mathscr{P}}^{\dg}[\text{quasi-iso}^{-1}].
$$
We will typically refer to the objects of each of these two categories as $\PP$-algebras, leaving the differential graded structure implicit (except for dg-categories, in order not to confuse these with ordinary categories or $\infty$-categories).

\medskip

\noindent \textbf{Linear algebra.} Let $\cat{A}$ be a dg-category (over our base field of characteristic zero).
Recall that a left $\cat{A}$-module is a dg-functor $\cat{A}\rt \cat{Ch}_k$ to the category of cochain complexes 
and a right $\cat{A}$-module is a functor $\cat{A}^{\op}\rt \cat{Ch}_k$, i.e.\ a left $\cat{A}^{\op}$-module. 
By default, modules are left modules. For an object $c\in\cat{A}$, the free $\cat{A}$-module at $c$ is the 
corepresentable functor 
$$
\begin{tikzcd}
\cat{A}_c\colon\cat{A}\arrow[r] & \cat{Ch}_k;\hspace{4pt} d\arrow[r, mapsto] & \cat{A}(c, d).
\end{tikzcd}
$$
An $\cat{A}$-$\cat{B}$-bimodule is a $\cat{B}^{\op}\otimes \cat{A}$-module (so that an $\cat{A}$-$k$-bimodule is just an $\cat{A}$-module). We will denote the canonical 
$\cat{A}$-bimodule by
$$
\begin{tikzcd}
\cat{A}\colon \cat{A}^{op}\otimes \cat{A}\arrow[r] & \cat{Ch}_k; \hspace{4pt} (c, d)\arrow[r, mapsto] & \cat{A}(c, d).
\end{tikzcd}
$$
Given an $\cat{A}$-$\cat{B}$-bimodule $E$ and a $\cat{B}$-$\cat{C}$-bimodule $F$, we can form the tensor product (Day convolution) $E\otimes_{\cat{B}} M$; for $a\in \cat{A}$ and $c\in \cat{C}$, one explicitly computes $(E\otimes_{\cat{B}} F)(c, a)$ as the coequalizer
$$\begin{tikzcd}
\bigoplus\limits_{b, b'\in \cat B} E(b', a)\otimes \cat{B}(b, b')\otimes M(c, b)\arrow[rr, yshift=1ex, "\text{act on }E"]\arrow[rr, yshift=-1ex, "\text{acts on }F"{swap}] & & \bigoplus\limits_{b\in \cat B} E(b, a)\otimes M(c, b)\arrow[r] & (E\otimes_{\cat{B}} F)(c, a).
\end{tikzcd}$$
In particular, when $\cat{C}=k$ this defines a functor $E\otimes_{\cat{B}}-\colon\cat{LMod}_{\cat{B}}^{\dg}\longrightarrow\cat{LMod}_{\cat{A}}^{\dg}$, and similarly for right modules. This satisfies the obvious identities, e.g.\ $\cat{A}\otimes_{\cat{A}} M=M$ and $E\otimes_{\cat{B}} \cat{B}_c=E_c=E(c, -)$.
As usual, composing two such functors coincides with tensoring bimodules: if $E$ is a 
$\cat{A}$-$\cat{B}$-bimodule, and $F$ a $\cat{B}$-$\cat{C}$-bimodule, then 
$$
(E\otimes_{\cat{B}}-)\circ(F\otimes_{\cat{C}}-)\cong (E\otimes_{\cat{B}}F)\otimes_{\cat{C}}-\,.
$$
The functor $E\otimes_{\cat{B}}-$ has a right adjoint, given by $\Hom_{\cat{A}}(E,-)$, where the left 
$\cat{B}$-module structure on $\Hom_{\cat{A}}(E,N)$ comes from the right $\cat{B}$-action on $E$. 
Similarly, the right adjoint to $-\otimes_{\cat{A}}E$ is $\Hom_{\cat{B}}(E,-)$. 

Finally, for $M$ a left $\cat{A}$-module, we will denote
$$
M^\vee\coloneqq\Hom_{\cat{A}}(M, \cat{A}).
$$
This is a right $\cat{A}$-module via the canonical right action of $\cat{A}$ on itself. 

\begin{remark}\label{rem:compact gen}
Let $\cat{A}$ be a dg-category with finitely many objects and consider its (ordinary) category of left modules. 
This category has a single compact projective generator, given by the direct sum of all free modules
$$
P=\bigoplus_{c\in\cat{A}} \cat{A}_c.
$$
Likewise, the category of right $\cat{A}$-modules has a single compact generator $Q=\bigoplus_{c\in\cat{A}} ({}_c\cat{A}).$
It follows from Morita theory that the category of left $\cat{A}$-modules is equivalent to the category of modules 
over the dg-algebra $B=\mm{End}_{\mc{A}}(P)^{\mm{op}}\cong \mm{End}_{\mc{A}^{\mm{op}}}(Q)$. 
Unraveling the definition, this algebra is given by the cochain complex
$$
B\cong \prod_{c,d}\cat{A}(c,d)
$$
and the product of two morphisms is their composition whenever they are composable, and $0$ otherwise.
The equivalence from left $\cat{A}$-modules to $B$-modules simply sends a left $\cat{A}$-module $M$ 
to $Q\otimes_{\cat{A}}M\cong \bigoplus_c M(c)$, with the obvious action maps arising from 
$\cat{A}(c, d)\otimes M(c)\rt M(d)$.
\end{remark}
The following is a special case of \cite[Proposition A.3.3.2]{HTT} (where we enrich over $\mm{Ch}_k$ with the projective model structure):
 \begin{proposition}\label{prop:kk-modules are model cat}
 	For any dg-category $\base$, the category of $\base$-modules admits a cofibrantly generated model structure for which the fibrations are pointwise surjections and weak equivalences are pointwise quasi-isomorphisms. Furthermore, this model structure is enriched over $\mm{Ch}_k$.
 \end{proposition}
\begin{remark}[Cofibrant objects]\label{rem:filtrations}
Note that this model structure arises from transfer along the free-forgetful adjunction $\prod_{c\in \base} \mm{Ch}_k\leftrightarrows \cat{LMod}^{\dg}_\base$ with right adjoint evaluating at each object $c\in\base$. Every free $\base$-module (in the image of the left adjoint) is therefore cofibrant and the generating cofibrations are the cone inclusions $\base_c[n]\rt \mm{Cone}(\base_c[n])$ for $c\in \base$. Conversely, a cofibrant object is (in particular) the retract of quasi-free $\base$-module.

It is not hard to verify that a map in $\cat{LMod}^{\dg}_{\base}$ is a cofibration if and only if it is a monomorphism whose cokernel is cofibrant. As a consequence, if $M$ is a $\base$-module equipped with an exhaustive increasing filtration whose associated graded is cofibrant, then $M$ is itself cofibrant.
\end{remark}
If $\base$ is of the form $k[G]$, where $G$ is any locally finite groupoid, then Maschke's theorem \cite[Chapter 4.2]{weibel} implies that every $k[G]$-module $M$ is the retract of the free $k[G]$-module generated by $M$ (by averaging over $G$). By the above remark, this implies that all objects are cofibrant and that the cofibrations are the monomorphisms.

\medskip

\noindent \textbf{Operads.} In this paper we make extensive use of the machinery of algebraic operads, namely bar (denoted $\Bar$) and cobar (denoted $\Omega$) constructions for both (co)operads and for (co)algebras, relative to a Koszul twisting morphism (denoted $\drt$), as developed in \cite{LodayVallette2012}.

In fact, throughout we will employ the theory of coloured operads \emph{relative} to a base dg-category, which we will denote by $\base$. More precisely, suppose that $\base$ has a set of objects $S$. By a \emph{$\base$-operad} $\PP$ we will mean an $S$-coloured (symmetric, differential graded) operad together with maps of $S$-coloured operads $\base\rt \PP\rt\base$ that compose to the identity (cf.\ Proposition \ref{prop:modelstructures} for a slightly different perspective). In particular, we will always assume that $\PP$ is \emph{augmented} over $\base$, unless explicitly stated otherwise.

An algebra over a $\base$-operad $\PP$ is simply an algebra over the underlying coloured operad. In particular, each $\PP$-algebra has an underlying left $\base$-module and the usual constructions with $\PP$-algebras, such as the free $\PP$-algebra or the bar construction, can be performed at the level of $\base$-modules as well. We refer to Appendix \ref{sec:operads} for an extensive discussion of the usual operadic homological algebra relative to a dg-category $\base$.

For any operad $\PP$ we denote by $\PP\{k\}$ its degree shift by $k$, such that $A$ is a $\PP\{k\}$ algebra if and only if $A[k]$ is a $\PP$ algebra. In particular, if $\PP$ is concentrated in degree $0$, $\PP\{1\}(n)$ is concentrated in degree $-n+1$.

All operads (resp. cooperads) are assumed to be unital and augmented (resp. counital and coaugmented) and have no other constraints in arities 0 and 1, unless otherwise explicitly written.

We will say that a $\base$-operad $\PP$ is \emph{$n$-reduced} if the map $\base\rt \PP$ is an isomorphism in arities $\leq n$ (in particular, it is trivial in arities $\neq 1$ and $\leq n$).

\begin{assumption}[Cofibrancy assumptions]\label{ass:cofibrancy}
Since we are not working over a field, various point-set level constructions involving tensor products and $\base$-linear duals are only well-behaved when applied to left $\base$-modules that are cofibrant (for the model structure of Proposition \ref{prop:kk-modules are model cat}). For this reason, we will typically (tacitly) assume throughout the text that our $\base$-operads are cofibrant as left $\base$-modules and that $\base$-cooperads are filtered-cofibrant left $\base$-modules (Definition \ref{def:cooperad}). Since our main results are formulated in homotopy-invariant terms, one can always replace $\mathscr{P}$ by a $\base$-operad for which this assumption holds.
\end{assumption}

%%%%%%%%%%%%% Section 2 starts here %%%%%%%%%%%%
\section{Moduli problems for algebras over operads}\label{sec:fmps}

In this section we introduce the main notion of a formal moduli problem for algebras over a 
(augmented) $\base$-operad $\PP$ and describe the associated tangent complex.
Recall that a classical (commutative) formal deformation functor is a functor 
$$\begin{tikzcd}
	\cat{Ring}^\mm{art}_k\arrow[r] & \cat{Set}
\end{tikzcd}$$
from the category of Artin local commutative $k$-algebras satisfying (some version of) the Schlessinger conditions. To describe the notion of a formal moduli problem for algebras over a $\base$-operad $\PP$, we will replace the category of Artin local rings by the following category of Artin $\PP$-algebras.
\begin{definition}\label{def:trivialalgebra}
	Let $\base$ be a dg-category and let $\PP$ be a $\base$-operad. A \emph{trivial algebra} is a $\PP$-algebra obtained from a $\base$-module by restriction along the augmentation map $\PP\rt \base$. We will denote by
	$$
	\triv{c}{n}\coloneqq\base(c, -)[n]
	$$
	the trivial algebra whose underlying $\base$-module is free on a generator at the object $c\in \base$, of cohomological degree $-n$. We denote its cone by $\triv{c}{n, n+1}$.
\end{definition}
\begin{definition}[{cf.\ \cite[Definition 1.1.8]{DAGX}}]\label{def:smallalgebra}
	The $\infty$-category $\Art_{\PP}$ of \emph{Artin $\PP$-algebras} is the smallest full subcategory of the $\infty$-category of $\PP$-algebras such that:
	\begin{enumerate}
		\item the trivial algebra $\triv{c}{n}$ is Artin for every object $c\in\base$ and every $n\geq 0$.
		
		\item for any Artin $\PP$-algebra $A$ and any map $A\rt \triv{c}{n}$ with $n\geq 1$, the homotopy pullback $A\times^h_{\triv{c}{n}} 0$ is also Artin.
	\end{enumerate}
\end{definition}
By definition, being Artin is a homotopy-invariant condition: any algebra quasi-isomorphic to an Artin algebra is itself Artin. As we see in Example \ref{ex:com yields artin}, in the case $\mathscr P$ is the commutative operad, we recover the usual notion of an Artin local algebra (or rather, their augmentation ideals, which is equivalent data).
\begin{example}\label{ex:squarezeroext}
	Suppose that $A$ is an Artin $\PP$-algebra and that 
	$$\begin{tikzcd}
		\triv{c}{n}\arrow[r] & B\arrow[r, two heads] & A
	\end{tikzcd}$$
	is a strict square zero extension of $A$ by the trivial $A$-module $\triv{c}{n}$, for $n\geq 0$. Then $B$ is Artin as well. Indeed, pulling back to a quasi-free resolution of $A$ if necessary, we may assume that $A$ is a quasi-free $\PP$-algebra. In this case, we can write $B=A\oplus \triv{c}{n}$ as a split square zero extension, with differential of the form
	$$
	d\big(a, v\big) = \big(da, dv+\chi(a)\big) \qquad\qquad\qquad a\in A\qquad v\in \triv{c}{n}.
	$$
	The map $\chi$ defines a $\PP$-algebra map $\chi\colon A\rt \triv{c}{n+1}$, and one can easily verify that $A'$ fits into a strict pullback square (a homotopy pullback since the vertical maps are fibrations)
	$$\begin{tikzcd}
		\hspace{-10pt}B=A\oplus \triv{c}{n}\arrow[r]\arrow[d] & \triv{c}{n, n+1}\arrow[d]\\
		A\arrow[r, "\chi"{swap}] &  \triv{c}{n+1}.
	\end{tikzcd}$$
	Since $\triv{c}{n, n+1}$ is contractible, we find that $B\simeq A\times^h_{\triv{c}{n+1}} 0$ is Artin.
\end{example}
\begin{remark}\label{rem:smallalgebras}
	In fact, the argument from Example \ref{ex:squarezeroext} can be used to give the following chain-level description of the Artin $\PP$-algebras: they form the smallest class of $\PP$-algebras that is closed under quasi-isomorphisms and (strict) square zero extensions by the trivial modules $\triv{c}{n}$ with $n\geq 0$.
\end{remark}
\begin{definition}\label{def:verysmall}
	A $\PP$-algebra $A$ is \emph{strictly Artin} if it admits a filtration 
	$$
	A=A^{(n)}\rt \dots \rt A^{(0)}=0
	$$
	with the property that each $A^{(i)}\rt A^{(i-1)}$ is a square zero extension with kernel $\triv{c_i}{p_i}$, for some $c_i\in\base$ and $p_i\geq 0$.
\end{definition}
An iterated application of Example \ref{ex:squarezeroext} shows that a strictly Artin $\PP$-algebra is Artin. Conversely, if $\PP$ is a \emph{cofibrant} $\base$-operad, then every Artin $\PP$-algebra is quasi-isomorphic to a strictly Artin $\PP$-algebra (see Lemma \ref{lem:verysmall}).

\begin{remark}\label{rem:smallisperfect}
	The $\base$-module underlying a strictly Artin $\PP$-algebra is cofibrant, and quasi-freely generated (i.e. disregarding differentials) by finitely many generators of degree $\leq 0$. In particular, it is a \emph{perfect} left $\base$-module.
	
\end{remark}

Let us remark that in favourable cases, being Artin reduces to a condition at the level of the cohomology groups of a $\mathscr{P}$-algebra:
\begin{definition}\label{def:highlyconn}
	Let $\mathscr{X}$ be a coloured symmetric sequence of chain complexes (e.g.\ a $\base$-operad). 
	We will say that $\mathscr{X}$ is \emph{connective} if for all tuples of colours,
	$$
	H^*\big(\mathscr{X}(c_1, \dots, c_p; c_0)\big)=0 \qquad \text{for all } *>0.
	$$ 
	Furthermore, $\mathscr{X}$ is \emph{eventually highly connective} if for every $n\in\mathbb{Z}$, there exists an $p(n)\in\mathbb{N}$ such that $H^*(\mathscr{X})$ vanishes in degrees $*\geq n$ in arities $\geq p(n)$.
\end{definition}

\begin{lemma}\label{lem:smallhomotopygroups}
	Suppose that $\base=k$ is a field and that $\PP$ is a connective operad. Then a $\mathscr{P}$-algebra $A$ is Artin if and only if it satisfies the following conditions:
	\begin{itemize}
		\item $H^i(A)=0$ for $i>0$ and $i\ll 0$.
		\item each $H^i(A)$ is a finite-dimensional vector space.
		\item each $H^i(A)$ is a nilpotent module over the $H^0(\PP)$-algebra $H^0(A)$, in the following sense: consider the action maps
		\begin{equation}\label{diag:actiononhomotopy}\begin{tikzcd}
				\mu(a_1, \dots, a_{q-1}, -)\colon H^i(A)\arrow[r] & H^i(A)
		\end{tikzcd}\end{equation}
		for $\mu\in H^0(\PP)(q)$ and $a_i\in H^0(A)$. Then there exists an $n$ such that any $n$-fold composition of such (possibly different) action maps is zero.
	\end{itemize}
\end{lemma}
\begin{proof}
	Consider a homotopy pullback of $\PP$-algebras of the form $B\simeq A\times^h_{k[n]} 0$. Then the map $H^*(B)\rt H^*(A)$ on cohomology is a square zero extension of $H^0(\PP)$-algebras with kernel $k[n-1]$. Using this inductively, one verifies the above conditions for every Artin $\PP$-algebra $A$.
	
	For the converse, we may assume $\PP$ is a cofibrant operad and by homotopy transfer \cite[Section 10.3]{LodayVallette2012} that $A$ is minimal, so that $H^i(A)=A_i$. Let $i\leq 0$ be the minimal number such that $A_i\neq 0$. We claim that there exists a nonzero $v\in A_i$ such that $\mu(a_1, \dots, a_p, v)=0$ for any operation $\mu$. Assuming this, we find that $\big<v\big>\rt A\rt A/\big<v\big>$ is a square zero extension by $k[i]$. Example \ref{ex:squarezeroext} then shows that $A$ is Artin if $A/\big<v\big>$ is Artin, and the result follows by induction.
	
	Since we assumed $\mathscr{P}$ to be connective, degree reasons dictate that the claim is equivalent to the following: there exists a $v\in A_i$ on which the $H^0(\mathscr{P})$-algebra $A_0$ acts trivially.
	Let $n$ be the minimal number such that any $n$-fold composition of action maps \eqref{diag:actiononhomotopy} is zero. If $n=0$ then $A_0$ acts trivially on $A_i$ and we are done. For $n\geq 1$, there exists by assumption an $(n-1)$-fold composite of action maps which is nonzero. Any nontrivial element $v$ in its image is then annihilated by all of $A_0$.
\end{proof}
\begin{example}\label{ex:com yields artin}
	Let $\PP=\Com$ be the 1-reduced commutative operad. An Artin $\Com$-algebra is exactly a nonunital cdga $\mf{m}$ with finite dimensional cohomology groups which are zero in degrees $>0$ and $\ll 0$, and with $H^0(\mf{m})$ nilpotent. These are exactly the augmentation ideals of the unital Artin dg-$k$-algebras from \cite[Proposition 1.1.11]{DAGX}.
\end{example}
With the Artin $\PP$-algebras playing the role of local Artin dg-algebras, we now define a `$\PP$-algebraic formal moduli problem' to be a functor $\Art_{\PP}\rt \sS$ satisfying the Schlessinger conditions.
\begin{definition}\label{def:formalmoduli}
	Let $\PP$ be a $\base$-operad. A \emph{formal moduli problem over $\PP$} is a functor
	$$\begin{tikzcd}
		F\colon \Art_{\PP}\arrow[r] & \sS
	\end{tikzcd}$$
	to the $\infty$-category of spaces, satisfying the following two conditions:
	\begin{enumerate}
		\item\label{it:schlessinger1} $F(0)\simeq \ast$, where $0$ is the zero algebra.
		\item\label{it:schlessinger2} $F$ sends a pullback diagram in $\Art_{\PP}$ of the form
		\begin{equation}\label{diag:squarezeroFMP}\begin{tikzcd}
					A'\arrow[d]\arrow[r] & 0\arrow[d]\\
					A\arrow[r] & \triv{c}{n}
		\end{tikzcd}\end{equation}
		to a pullback square of spaces, for every colour $c\in\base$ and $n\geq 1$.
	\end{enumerate}
	We will denote the $\infty$-category of formal moduli problems over $\PP$ by $\FMP_{\PP}$.
\end{definition}
\begin{example}\label{ex:formal spectrum}
	To every $\PP$-algebra $B$ we can associate its \textit{formal spectrum}, a formal moduli problem $\mm{Spf}(B)\colon \Art_{\PP}\rt \sS$, given by $A\longmapsto \Map_{\PP}(B,A)$.
\end{example}
If one thinks of the functor $F$ as assigning to a $\PP$-algebra $A$ the space of deformations of a certain object $\mc{X}$, then the above conditions encode the usual obstruction theory for deformations along square zero extensions. Indeed, note that the pullback square \eqref{diag:squarezeroFMP} exhibits $A'$ as a square zero extension of $A$ by the trivial $A$-module $\triv{c}{n-1}$ (cf.\ Example \ref{ex:squarezeroext}). For every deformation $\mc{X}_A\in F(A)$, one obtains an `obstruction class'
$$
\mm{ob}(\mc{X}_A) \in \pi_0F(\triv{c}{n})
$$
by applying the functor $F$ to the map $A\rt \triv{c}{n}$. This obstruction class is zero if and only if $\mc{X}_A$ lifts to a deformation over the square zero extension $A'$.

Let us recall that there is a more cohomological way of interpreting these kinds of obstruction classes, as follows. Applying condition \ref{it:schlessinger2} in the case where $A=0$, one obtains a natural sequence of equivalences
$$\begin{tikzcd}
	F\big(\trivc\big)\arrow[r, "\sim"] & \Omega F\big(\triv{c}{1}\big)\arrow[r, "\sim"] & \Omega^2 F\big(\triv{c}{2}\big)\arrow[r, "\sim"] & \dots
\end{tikzcd}$$
In other words, the sequence of spaces $F\big(\triv{c}{n}\big)_{n\geq 0}$ forms an $\Omega$-spectrum $T(F)_c$.
\begin{definition}\label{def:tangent}
	We refer to $T(F)_c$ as the \emph{tangent complex of $F$ at $c$}, and to the spectra $T(F)_{c\in\base}$ collectively as the tangent complex of $F$.
\end{definition}

In fact, the tangent complex admits a canonical $\base^{\op}$-module structure, as we will see in the next lemma.
We denote the category of spectra by $\cat{Sp}$. Recall that there is a functor $\cat{Mod}_{\mathbb{Z}} \to \cat{Sp}$ sending $X$ to the spectrum formed by (the Dold--Kan image of) its iterated connective covers $\tau^{\leq 0}X, \tau^{\leq 1} X, \dots$, each of which is the looping of the next.
\begin{lemma}\label{lem:tangentcomplexmodule}
	For any formal moduli problem $F$ over $\PP$, the tangent complex has a unique inverse image under the forgetful functor
	\begin{equation}\label{diag:moduletospectra}\begin{tikzcd}
			\cat{Mod}_{\base^{\op}}\arrow[r] & \prod_{c\in \base^{\op}} \cat{Mod}_{\mathbb{Z}}\arrow[r] & \prod_{c\in \base^{\op}} \cat{Sp}
	\end{tikzcd}\end{equation}
	with the following property: for all free $\base^{\op}$-modules generated by $c\in \base^{\op}$ in degree $n\geq 0$, there is a natural equivalence
	$$
	\Map_{\cat{Mod}_{\base^{\op}}}\big(\trivop{c}{-n}, T(F)\big)\simeq F\big(\triv{c}{n}\big).
	$$
\end{lemma}
In other words, the obstructions to lifting deformations along square zero extensions are given by classes in the cohomology of the $\base^{\op}$-module $T(F)$.
\begin{remark}\label{rem:modulestospectra}
	The first functor in \eqref{diag:moduletospectra} takes a $\base^{\op}$-module $V$ to the collection of chain complexes $V(c)$. Equivalently, one can consider these as $H\mathbb{Z}$-module spectra (via the Dold-Kan correspondence \cite{SchwedeShipley}). The second functor then forgets the $H\mathbb{Z}$-module structure. The composite functor preserves both limits and colimits, since its left adjoint preserves compact generators: it sends $(0, \dots, 0, \mathbb{S}^n, 0, \dots, 0)$, with a sphere at place $c$, to the free $\base^\op$-module $\trivop{c}{n}$.
\end{remark}
\begin{proof}
	Uniqueness follows from the fact that the free modules $\trivop{c}{-n}$ with $n\geq 0$ generate the $\infty$-category $\cat{Mod}_{\base^{\op}}$ under colimits. Existence follows either from Theorem \ref{thm:axiomatic}, or from the following argument. Let $\cat{C}^{\leq n}\subseteq \Mod_{\base^{\op}}$ denote the subcategory generated by the free $\base^{\op}$-modules $\trivop{c}{n}$ under finite limits and let $\cat{C}=\colim_{n} \cat{C}^{\leq n}$ be their union. Consider the functors
	$$\begin{tikzcd}
		X_n\colon \big(\cat{C}^{\leq n}\big)^{\op}\arrow[r] & \sS;\hspace{4pt} V\arrow[r] & \Omega^n F\big(V[-n]^\vee\big).
	\end{tikzcd}$$
	These functors are well-defined because the trivial $\PP$-algebra $V[-n]^{\vee}$ is Artin for all $V\in\cat{C}^{\leq n}$. Because $F$ is a formal moduli problem, there are natural equivalences $X_{n}\simeq X_{n+1}\big|_{\cat{C}^{\leq n}}$, so that one obtains a functor
	$$\begin{tikzcd}
		X\colon \cat{C}^{\op}\arrow[r] & \sS ;\hspace{4pt} V\arrow[r, mapsto] & {\Omega^n F\big(V[-n]^\vee\big).}
	\end{tikzcd}$$
	This functor sends finite colimits in $\cat{C}$ to limits of spaces, since $F$ is a formal moduli problem. But $\cat{C}\subseteq \Mod_{\base^{\op}}$ contains all free $\base^\op$-modules and is closed under finite colimits, so it follows that $X$ is representable by a $\base^{\op}$-module \cite[Corollary 5.3.5.4, Proposition 5.3.5.11]{HTT}. Unravelling the definitions, this is exactly the desired $\base^{\op}$-module $T(F)$.
\end{proof}

%%%%%%%%%%%%% Section 3 starts here %%%%%%%%%%%%

\section{Examples and applications}\label{sec:examples}

In this section we discuss various examples and applications of our Theorem \ref{thm:mainthmkoszul}. We start in Section \ref{sec:comm def} by recalling the usual deformation theory along artinian algebras from this perspective, with emphasis on (commutative) deformations of algebraic structures.

In the deformations theory of operads and algebras over them, one also encounters deformation problems parametrized by \emph{permutative algebras}. Koszul dually, this roughly corresponds to the fact that deformation complexes are pre-Lie algebras. 
In Section \ref{sec:permutative} we more generally treat such permutative deformation problems, their pre-Lie tangent spaces and some concrete examples, while in Section \ref{sec:operadicdeftheory} we consider deformations along operads.

Finally, in Section \ref{sec:splendid}, we will give some further examples of operads satisfying the conditions of Theorem \ref{thm:mainthm}. Of these conditions, the most important one is condition \ref{it:splendidness}, which we discuss in some detail. This last section is mostly independent of the previous ones.

\subsection{Commutative and associative deformation theory}\label{sec:comm def}
Let us start by briefly reviewing how Theorem \ref{thm:mainthmkoszul} plays out in the classical cases of the $(k$-linear) commutative operad $\Com$ and associative operad $\As$. Note that, since we are always working with augmented operads, algebras over the operads $\Com$ and $\As$ are given by \emph{nonunital} commutative and associative algebras. Such algebras are equivalent to \emph{augmented} unital commutative algebras by adding a unit and taking the augmentation ideal:
$$\begin{tikzcd}[column sep=3pc]
\Alg_{\Com}\arrow[r, "\cong"{swap}, "A\leftrightarrow A^+"] & \cat{CAlg}^{\mm{aug}}_k & \Alg_{\As}\arrow[r, "\cong"{swap}, "A\leftrightarrow A^+"]& \Alg^{\mm{aug}}_{k}.
\end{tikzcd}$$
In particular, this identifies the subcategory of nonunital (commutative) Artin dg-algebras with the subcategory of augmented Artin dg-algebras and we can identify $F\in \FMP_{\Com}$ with a formal moduli problem $F\colon \cat{CAlg}^{\mm{art}}_k\rt \sS$ in the usual sense of \cite{DAGX} (and similarly in the associative case). 

The operads $\Com$ and $\As$ are binary Koszul with Koszul dual operads given by $\Lie$ and $\As$. Consequently, there are Koszul duality functors
$$\begin{tikzcd}
\mf{D}_{\Lie}\colon \Alg_{\Com}\arrow[r] & \Alg^{\op}_{\Lie} & & \mf{D}_{\As}\colon \Alg_{\As}\arrow[r] & \Alg^{\op}_{\As}
\end{tikzcd}$$
sending a commutative (associative) algebra $A$ to the linear dual of the cofree coLie (coassociative) coalgebra on $A[1]$, with differential induced by the multiplication on $A$. These constructions preserve quasi-isomorphisms, so that they indeed induce functors of $\infty$-categories. Theorem \ref{thm:mainthmkoszul}, or more precisely Theorem \ref{thm:biduals}, then yields equivalences
$$\begin{tikzcd}[column sep=3pc]
\FMP_{\Com}\arrow[r, yshift=1ex, "{T[-1]}"] & \Alg_{\Lie}\arrow[l, yshift=-1ex, "\MC", "\sim"{swap}] & \FMP_{\As}\arrow[r, yshift=1ex, "{T[-1]}"] & \Alg_{\As} \arrow[l, yshift=-1ex, "\MC", "\sim"{swap}]
\end{tikzcd}$$
where the functor $\MC$ is defined such that for a pair $(\mf{g}, A)$ of a Lie algebra and an Artin commutative algebra (resp.\ two associative algebras)
$$
\MC_{\mf{g}}(A)=\Map(\mf{D}(A), \mf{g}).
$$
Of course, these two cases of the theorem have already been established by Pridham \cite{Pridham} and Lurie \cite{DAGX}. 
To illustrate these equivalences (and motivate what will follow), let us recall how these equivalences can be used to study deformations of algebras and modules:
\begin{example}\label{ex:defmodule}
Suppose that $B$ is a connective associative algebra and $M$ a connective $B$-module. Then one can consider the associative formal moduli problem 
$$\begin{tikzcd}
\mm{Def}_M\colon \Art_{\As}\arrow[r] & \sS; & \mm{Def}_M(A)=\Mod_{A^+\otimes B}\times_{\Mod_{k\otimes B}} \{M\}
\end{tikzcd}$$
sending a (nonunital) associative algebra $A$ to the space of all $A^+\otimes B$-modules $M'$ equipped with a $B$-linear equivalence $k\otimes_{A^+} M'\simeq M$. This formal moduli problem is classified by the derived endomorphism algebra $\mm{RHom}_B(M, M)$ \cite[Corollary 5.2.15]{DAGX}, i.e.\ the endomorphism algebra of a cofibrant resolution of $M$ over $B$.
\end{example}
\begin{example}\label{ex:defalgebra}
Suppose that $\QQ$ is a connective $k$-operad and that $R$ is a connective $k$-linear $\QQ$-algebra. Then one can consider the deformation problem
$$\begin{tikzcd}
\mm{Def}_R\colon \Art_{\Com}\arrow[r] & \sS; & \mm{Def}_R(A)=\Alg_{A^+\otimes \QQ}\times_{\Alg_{\QQ}} \{R\}
\end{tikzcd}$$
sending a (nonunital) commutative algebra $A$ to the space of all $A^+$-linear $\QQ$-algebras $R'$ (equivalently, algebras over the tensor product $A^+\otimes \QQ$) equipped with a $\QQ$-algebra equivalence $k\otimes_{A^+} R'\simeq R$. The Lie algebra classifying this formal moduli problem is given by the derived derivations of $R$, i.e.\ the derivations of a cofibrant replacement of $R$ \cite{hinich2004deformations, nuiten2019koszul}.
\end{example}
\begin{remark}
When $\QQ$ is an operad concentrated in arity $1$ (i.e.\ an algebra), Example \ref{ex:defalgebra} is simply the restriction of Example \ref{ex:defmodule} to commutative Artin algebras. This is reflected in the fact that the Lie algebra classifying commutative deformations of a $B$-module $M$ is the Lie algebra underlying the associative algebra $\mm{RHom}_B(M, M)$. We will come back to this in Section \ref{sec:permutative} and (in more detail) in Section \ref{sec:naturality}.
\end{remark}
\begin{example}\label{ex:reduceddefalgebras}
For an operad $\QQ$ and a $\QQ$-algebra $(R, \mu)$, one can also consider the commutative formal moduli problem $\ol{\mm{Def}}_R$ sending $A$ to the space of $A^+$-linear $\QQ$-algebra structures on $A^+\otimes R$ with a $\QQ$-algebra equivalence $k\otimes_{A^+} (A^+\otimes R)\simeq R$. Note that this differs from Example \ref{ex:defalgebra}: we only consider deformations of $(R, \mu)$ whose underlying complex is the trivial deformation $A^+\otimes R$ of the complex underlying $R$ (i.e.\ we do not deform the differential).

Note that a $\QQ$-algebra structure on $R\otimes A^+$ is equivalent to the datum of a $k$-linear operad map to the ($A^+$-linear, \emph{non-augmented}) endomorphism operad of $A^+\otimes R$
$$
\QQ\rt \End_{A^+}(A^+\otimes R) \simeq \End(R)\otimes A^+.
$$
Here the equivalence follows from the fact that $A^+$ is Artin, so in particular perfect as a $k$-module. Using this, it follows that the space $\ol{\mm{Def}}_R(A^+)$ is equivalent to the space of dotted lifts in the following diagram in the $\infty$-category $\cat{Op}_k$ of $k$-linear operads:
\begin{equation}\label{diag:def endomorphism operad}\begin{tikzcd}
 & \End_k(R)\otimes A^+\arrow[d]\\
\QQ\arrow[r, "\mu"{swap}]\arrow[ru, dotted] & \End_k(R).
\end{tikzcd}\end{equation}
\end{example}
The commutative formal moduli problem from Example \ref{ex:reduceddefalgebras} makes sense more generally: instead of deforming an operad map into the endomorphism operad $\End_k(R)$, one can take any map of $k$-linear operads $\phi\colon \QQ\rt \PP$. To describe the associated Lie algebra, let us suppose that $\QQ=\Omega\CC$ arises as the cobar construction of a $k$-cooperad (which can always be arranged up to quasi-isomorphism, cf.\ Proposition \ref{prop:counit bar cobar}) and recall the following construction:
\begin{construction}[Deformation complex]\label{cons:def complex twisted}
Let $\PP$ be a (not necessarily augmented) $k$-linear operad, $\CC$ a $k$-cooperad with cokernel $\ol{\CC}$ of its coaugmentation, and $\QQ=\Omega\CC$. The complex
$$
\mf{g}=\Hom(\ol{\CC}, \PP)=\prod_p \Hom\big(\ol{\CC}(p), \PP(p)\big)^{\Sigma_p}
$$
of maps of symmetric sequences $\ol{\CC}\rt \PP$ comes equipped with a binary operation $\star$, such that $[\phi, \psi]=\phi\star \psi-(\pm)^{|\phi|\cdot |\psi|} \psi\star \phi$  endows $\mf{g}$ with the structure of a dg-Lie algebra  \cite[Proposition 6.4.7]{LodayVallette2012}. Informally, $\phi\star \psi(c)$ is obtained by taking the sum of all partial cocompositions of $c$ into two elements of $\ol{\CC}$, applying $\phi$ and $\psi$ to them and then applying the composition in $\PP$ (see also Construction \ref{constr:k-twisting morphisms} and Remark \ref{rem:convolutionoperad} for more details).

An operad map $\phi\colon \QQ=\Omega\CC\rt \PP$ then corresponds to a Maurer--Cartan element in $\mf{g}$ \cite[Theorem 6.5.7]{LodayVallette2012} (or see Proposition \ref{prop:Tw adjunction}). Given such a map $\phi\colon \QQ\rt \PP$, let us write
$$
\mf{g}^{\phi}=\big(\mf{g}, d+[\phi, -]\big)
$$
for the twisting of $\mf{g}$ by the Maurer--Cartan element $\phi$. This is again a dg-Lie algebra for the original bracket $[-, -]$.
\end{construction}
Let $\phi\colon \Omega\CC=\QQ\rt \PP$ be a map from an augmented to a not necessarily augmented $k$-linear operad. The deformations of $\phi$ determine a formal moduli problem 
\begin{equation}\label{eq:def operad maps}
\mm{Def}_{\phi\colon \QQ\to \PP}(A)=\Map_{\cat{Op}}(\QQ, \PP\otimes A^+)\times_{\Map_{\cat{Op}}(\QQ, \PP)} \{\phi\}
\end{equation}
sending each $A\in \Alg^{\mm{art}}_{\Com}$ to the space of deformations of $\phi$ (as in Diagram \eqref{diag:def endomorphism operad}). This deformation problem is classified by the Lie algebra $\mf{g}^{\phi}$, as illustrated by the following two observations:
\begin{lemma}\label{lem:maurer-cartan com}
Suppose that $\mf{g}$ is a Lie algebra and that $A$ is \emph{strictly} Artin. If one chooses a fibrant simplicial resolution of $\mf{g}$, then the space $\MC_\mf{g}(A)$ can be modeled by the simplicial set of Maurer--Cartan elements
\[
\MC_{\mathfrak{g}}(A)=\MC(\mathfrak{g}_\bullet\otimes A)\,.
\]
\end{lemma}
\begin{proof}
For strictly finite-dimensional $A$, the Lie algebra $\mf{D}(A)=\mf{D}_{\Lie}(A)$ freely generated by $A^\vee[-1]$, with differential given on generators by the linear dual of the product. If $A$ is furthermore nilpotent, then $\mf{D}_{\Lie}(A)$ is cofibrant: the dual of the adic filtration on $A$ yields a filtration on $\mf{D}(A)$ where each stage is obtained from the previous one by adding generators whose differential is contained in the previous stage. The mapping space $\Map_{\Lie}(\mf{D}(A), \mf{g})$ can then be modeled by the simplicial set of maps $\mf{D}(A)\rt \mf{g}_\bullet$. Since $\mf{D}(A)$ is quasi-free on $A^\vee[-1]$, such maps are determined by degree $1$ elements in $A\otimes \mf{g}_\bullet$ and compatibility with the differential translates into the Maurer--Cartan equation \cite[Corollary 11.1.4]{LodayVallette2012}.
\end{proof}
\begin{proposition}\label{prop:comm deformations from convolution}
Let $\phi\colon \Omega\CC=\QQ\rt \PP$ be as above and let $A$ be a strictly Artin commutative dg-algebra. Then there is an equivalence $\mm{Def}_{\phi\colon \QQ\to \PP}(A)\simeq \MC_{\mf{g}^{\phi}}(A)$.
\end{proposition}
\begin{proof}
Let us start by recalling the following property of the twisting of a Lie algebra by a Maurer--Cartan element: a degree $1$ element $\phi_A\in \mf{g}^\phi\otimes A$ is Maurer--Cartan element if and only if the element $\phi\otimes 1+\phi_A$ defines a Maurer--Cartan element in $\mf{g}\otimes A^+=(\mf{g}\otimes k)\oplus (\mf{g}\otimes A)$. Since $A$ is finite-dimensional, the Lie algebra $\mf{g} \otimes A^+$ coincides with the Lie algebra from Construction \ref{cons:def complex twisted} applied to $\CC$ and $\PP\otimes A^+$. By the discussion there, we obtain bijections
$$
\MC(\mf{g}^\phi\otimes A)\cong \MC(\mf{g}\otimes A^+)\times_{\MC(\mf{g})}\{\phi\} =\Hom_{\cat{Op}^{\dg}}(\Omega\CC, \PP\otimes A^+)\times_{\Hom_{\cat{Op}^{\dg}}(\Omega\CC, \PP)}\{\phi\}
$$
to the set of maps of dg-operads $\Omega\CC\rt \PP\otimes A^+$ which reduce $\phi$ modulo $A$.

Now note that Construction \ref{cons:def complex twisted} is natural in $\PP$. A fibrant simplicial resolution $\PP_\bullet$ of the operad $\PP$ with $\PP_0=\PP$ therefore gives rise to a simplicial Lie algebra $\mf{g}^\phi_\bullet$, which forms a simplicial resolution of $\mf{g}^\phi$. Lemma \ref{lem:maurer-cartan com} and the previous argument then show that the space $\MC_{\mf{g}^\phi}(A)$ can be modeled by the simplicial set
$$
\Hom_{\cat{Op}^{\dg}}(\Omega\CC, \PP_{\bullet}\otimes A^+)\times_{\Hom_{\cat{Op}^{\dg}}(\Omega\CC, \PP_{\bullet})}\{\phi\}.
$$
Since $\Omega\CC$ is a cofibrant operad and $\PP_\bullet\otimes A^+$ is a simplicial resolution of $\PP\otimes A^+$, the simplicial set $\Hom_{\cat{Op}^{\dg}}(\Omega\CC, \PP_{\bullet}\otimes A^+)$ is a model for the space $\Map_{\cat{Op}}(\Omega\CC, \PP\otimes A^+)$. 
Furthermore, note that the map of simplicial resolutions $\PP_\bullet\otimes A^+\rt \PP_\bullet$ is a Reedy fibration between simplicial resolutions (the relative matching maps are given by the surjections $\PP_n\otimes A^+\rt M_n(\PP_\bullet)\otimes A^+\times_{M_n(\PP_\bullet)} \PP_n$). Consequently, the map of simplicial sets $\Hom_{\cat{Op}^{\dg}}(\Omega\CC, \PP_{\bullet}\otimes A^+)\rt \Hom_{\cat{Op}^{\dg}}(\Omega\CC, \PP_{\bullet})$ is a Kan fibration \cite[Theorem 16.5.2]{hirschhorn2003modelcategories}.  The above pullback is therefore a homotopy pullback, so that it indeed models the space ${\mm{Def}}_{\phi\colon \QQ\to\PP}(A)$. 
\end{proof}

\subsection{Permutative deformation theory}\label{sec:permutative}
In this section we will spell out the contents of Theorem \ref{thm:mainthmkoszul} in a bit more detail for a less classical pair of Koszul dual operads: we will consider deformation problems whose associated Lie algebra arises from a pre-Lie algebra.

\begin{definition}
A \emph{pre-Lie algebra} is a vector space $V$ equipped with a bilinear operation $\{-,-\}$ such that for every $x,y,z\in V$, 
\[
\{\{x,y\},z\}-\{x,\{y,z\}\}=\{\{x,z\},y\}-\{x,\{z,y\}\}\,.
\]
Such pre-Lie algebras are algebras over a $k$-linear operad $\preLie$.
\end{definition}
\begin{definition}[\cite{Chapo}]
A \emph{permutative algebra}, or $\Perm$-algebra, is an associative algebra $(X,\cdot)$ such that for 
every $x,y,z\in X$, 
\begin{equation}\label{eq-perm-alg}
x\cdot(y\cdot z)=x\cdot (z\cdot y)\,. 
\end{equation}
One easily sees that permutative algebras are algebras over a $k$-linear operad $\Perm$ (which in fact arises from an operad in sets).
\end{definition}
The operads $\Perm$ and $\preLie$ are both binary quadratic, and it is not hard to verify that they are each others quadratic dual. Consequently, there is a functor of $\infty$-categories
\begin{equation}\label{diag:dual prelie}\begin{tikzcd}
\mf{D}\colon \Alg_{\Perm}\arrow[r] & \Alg_{\preLie}^{\op}; \quad A\arrow[r, mapsto] & \Bar_{\mathsf{copreLie}}(A)^\vee
\end{tikzcd}\end{equation}
sending a permutative algebra $A$ to the $k$-linear dual of its bar construction, i.e.\ of the cofree pre-Lie coalgebra generated by the suspension $A[1]$, with differential determined by the permutative structure on $A$ \cite{ginzburg1994koszul, LodayVallette2012} (or see Definition \ref{def:bar-cobar algebras}). This construction preserves quasi-isomorphisms and hence descends to a functor of $\infty$-categories.

Since the operads $\Perm$ and $\preLie$ are Koszul \cite{ChaLiv}, we then have the following special case of Theorem \ref{thm:mainthmkoszul} (made slightly more precise, as in Theorem \ref{thm:biduals}):
\begin{theorem}\label{thm:main prelie}
For every pre-Lie algebra $\mf{g}$, consider the functor
$$\begin{tikzcd}
\MC_{\mf{g}}\colon \Art_{\Perm}\arrow[r] & \sS; \quad A\arrow[r, mapsto] & \Map_{\preLie}\big(\mf{D}(A), \mf{g}\big).
\end{tikzcd}$$
Here the domain is the $\infty$-category of Artin permutative algebras, i.e.\ $A$ such that $H^*(A)$ is finite-dimensional and concentrated in degrees $\leq 0$, and such that $H^0(A)$ is nilpotent (see Lemma \ref{lem:smallhomotopygroups}). This determines an equivalence of $\infty$-categories
$$\begin{tikzcd}
\MC\colon \cat{Alg}_{\preLie}\arrow[r, "\sim"] & \cat{FMP}_{\Perm}.
\end{tikzcd}$$
\end{theorem}
\begin{remark}
The operad $\Perm$ fits into a sequence of Koszul binary quadratic operads 
\[
\As\longrightarrow\Perm\longrightarrow\Com\,,
\]
(compatible with quadratic data) whose quadratic dual sequence is 
\[
\As\longleftarrow \preLie\longleftarrow\Lie\,.
\]
This dual sequence sends the Lie bracket to the commutator of the pre-Lie structure (respectively, of the associative product). The equivalence from Theorem \ref{thm:mainthmkoszul} is natural in these operads, in the sense that the equivalence of Theorem \ref{thm:main prelie} and the equivalences from Section \ref{sec:comm def} fit into a commuting diagram of $\infty$-categories  (Proposition \ref{prop:naturalityfmp})
\[
\begin{tikzcd}
\cat{Alg}_{\As} \arrow[d, "\sim"{swap}] \arrow[r]		& \cat{Alg}_{\preLie} \arrow[d, "\sim"] \arrow[r]	& \cat{Alg}_{\Lie} \arrow[d, "\sim"]\\
\cat{FMP}_{\As} \arrow[r] & \cat{FMP}_{\Perm}\arrow[r] & \cat{FMP}_{\Com}. 
\end{tikzcd}\]
Here the top horizontal maps forget algebraic structure, while the bottom horizontal maps restrict formal moduli problems along the forgetful functors $\Art_{\Com}\rt \Art_{\Perm}\rt \Art_{\As}$. This tells us in particular that a commutative formal moduli problem lifts to a permutative one 
(respectively, to an associative one) if and only if the Lie bracket on its tangent 
complex arises from a pre-Lie structure (respectively, an associative structure). 

In fact, this same remark applies to any other map $\mathscr{P}\rt \mathscr{Q}$ of Koszul binary quadratic operads, with Koszul dual map $\mathscr{Q}^!\rt \mathscr{P}^!$.
\end{remark}

Before providing several examples, let us give a more explicit description of the value of the permutative formal moduli problem $\MC_{\mf{g}}$ classified by a pre-Lie algebra $\mathfrak{g}$ on a \emph{strictly} Artin permutative algebra $A$. To this end, note that the tensor product $\mathfrak{g}\otimes A$ of a pre-Lie algebra $\mf{g}$ and a permutative algebra $A$ is a Lie algebra under
\[
[x\otimes a,y\otimes b]\coloneqq (-1)^{|a||y|}\{x,y\}\otimes a\cdot b - (-1)^{(|a|+|x|)|y|}\{y, x\}\otimes a\cdot b\,.
\]
\begin{lemma}\label{lem:maurer-cartan prelie}
Suppose that $\mf{g}$ is a pre-Lie algebra and that $A$ is strictly finite-dimensional nilpotent permutative algebra $A$. If one chooses a fibrant simplicial resolution $\mf{g}_\bullet$ of $\mf{g}$, then the space $\MC_\mf{g}(A)$ can be modeled by the simplicial set of Maurer--Cartan elements
\[
\MC_{\mathfrak{g}}(A)=\MC(\mathfrak{g}_\bullet\otimes A)\,.
\]
\end{lemma}
\begin{proof}
The proof of Lemma \ref{lem:maurer-cartan com} carries over verbatim (or see Remark \ref{rem:maurer-cartan different resolution}).
\end{proof}
A similar result applies to any Koszul dual pair of binary quadratic operads.

\subsubsection*{Example: two-parameter permutative deformations.}
Let $A$ be the quotient of the free permutative algebra on two degree zero generators $\hbar,\epsilon$ by the relations $\hbar\cdot \epsilon=0=\epsilon^2$. 
As a vector space, it admits the following basis: $\{\hbar^n,\epsilon\cdot \hbar^m|n\geq1,m\geq0\}$. 
\begin{lemma}\label{lem:permutativedefs}
For every pre-Lie algebra $\mathfrak{g}$, the Maurer--Cartan set of $\mathfrak{g}\otimes A$ consists of pairs 
$(X,Y)$ of degree one elements in $\mathfrak{g}[\hbar]=\mf{g}\otimes_k k[\hbar]$ such that 
\[
X(0)=0, \qquad\quad dX+\{X,X\}=0 \qquad\text{and}\qquad \nabla_X(Y)=0,
\]
where 
$\nabla_X=d-(-1)^{|\mathbf{-}|}\{\mathbf{-},X\}$. 
\end{lemma}
\begin{proof}
A degree one element $\gamma$ in $\mathfrak{g}\otimes A$ is a (finite) linear combination 
\[
\gamma=\sum_{n\geq1}X_n\otimes \hbar^n+\sum_{m\geq0}Y_m\otimes \epsilon\hbar^m\,,
\]
where the $X_n$ and $Y_m$ have degree one. 
The Maurer--Cartan equation $d\gamma+\{\gamma,\gamma\}=0$ then translates into two infinite 
families of equations: looking at the coefficient of $\hbar^n$ for each $n\geq 1$ and the coefficient of $\epsilon\hbar^m$ for each $m\geq 0$ gives
\[
dX_n+\sum_{k+l=n}\{X_k,X_l\}=0\qquad\qquad\text{and}\qquad\qquad dY_m+\sum_{k+l=m}\{Y_k,X_l\}=0\,. 
\]
Writing $X\coloneqq\sum_{n\geq1}X_n\hbar^n\in \hbar\cdot \mathfrak{g}[\hbar]$ and $Y\coloneqq\sum_{m\geq0}Y_n\hbar^n\in\mathfrak{g}[\hbar]$, these two families of equations are equivalent to the two equations $dX+\{X,X\}=0$ and $dY+\{Y,X\}=0$.
\end{proof}
\begin{remark}
Observe that this is different from what we would get by looking at the Maurer--Cartan set of $\mathfrak{g}\otimes C$, where $C$ is the nonunital \emph{commutative} algebra generated by $\hbar, \epsilon$ subject to the relation $\epsilon^2=0$: the equation $\nabla_X(Y)=0$ would have 
to be replaced by $d_X(Y)=0$, where $d_X=d+[X,-]$. 
\end{remark}
The permutative algebra $A$ introduced above is not Artin, but each finite-dimensional quotient $A_n=A/(\hbar^n)$ is. Using Lemma \ref{lem:permutativedefs}, one sees for instance that the space of Maurer--Cartan elements in $\mathfrak{g}\otimes A_2$ 
is the space of pairs of $1$-cocycles $(Y,X)$ in $\mathfrak{g}$ together with a null-homotopy of $\{Y,X\}$: 
\[
\mm{MC}_{\mathfrak{g}}(A_2)\simeq 
\mathrm{hofib}\big((\tau^{\leq1}\mathfrak{g})^{\times 2}\xrightarrow{\{\,,\,\}}\tau^{\leq2}\mathfrak{g}\big)\,.
\]

\subsubsection*{Deforming trivial morphisms of operads}
A standard source of pre-Lie algebras is given by convolution pre-Lie algebras \cite[Section 6.4]{LodayVallette2012} (see also Remark \ref{rem:convolutionoperad}). We have already seen these pre-Lie algebras implicitly in Construction \ref{cons:def complex twisted}: if $\mathscr{C}$ is a $k$-cooperad, $\ol{\CC}$ the cokernel of its coaugmentation and $\mathscr{P}$ be a $k$-linear operad (not necessarily augmented), then the \emph{convolution pre-Lie algebra} is given by
$$
\mf{g}=\prod_{p\geq 0} \Hom\big(\ol{\mathscr{C}}(p),\mathscr{P}(p)\big)^{\Sigma_p}
$$
with pre-Lie structure given by the operation $\star$ built from the partial composition of $\PP$ and partial cocomposition of $\ol{\CC}$ (see Construction \ref{cons:def complex twisted} or Construction \ref{constr:k-twisting morphisms}). Note that $\mf{g}$ arises as the totalization of a $\mathbb{Z}_{\geq 0}$-graded pre-Lie algebra $\mf{g}^\mm{gr}$, where
$$
\mathfrak{g}^{\mm{gr}}(p)=\Hom\big(\ol{\mathscr{C}}(p),\mathscr{P}(p)\big)^{\Sigma_p}
$$
and the pre-Lie operation $\star$ has weight $-1$ with respect to the $\mathbb{Z}_{\geq 0}$-grading.

To describe the permutative deformation problem classified by the pre-Lie algebra $\mf{g}$, observe that for every permutative algebra $A$, there is a \emph{nonunital} operad $\PP\otimes A$, such that:
\begin{itemize}
\item the underlying symmetric sequence is given by $(\mathscr{P}\otimes A)(n)\coloneqq\mathscr{P}(n)\otimes A$;
\item the composition operation reads as 
\[
(\psi_0\otimes a_0)\circ(\psi_1\otimes a_1,\dots,\psi_p\otimes a_p)\coloneqq\pm\big(\psi_0\circ (\psi_1,\dots,\psi_p)\big)\otimes a_0\cdot a_1\cdots a_p\,
\]
where $\pm$ is a Koszul sign. Associativity of the composition follows from the associativity of permutative algebras and the permutative axiom \eqref{eq-perm-alg}, while the equivariance/commutativity directly follows from \eqref{eq-perm-alg}. 
\end{itemize}
Now let $\ol{\Omega}\CC$ denote the augmentation ideal of the cobar construction of $\CC$ and consider the functor
\begin{equation}\label{eq:perm def of operad map}\begin{tikzcd}
\mm{Def}_{\ul{0}\colon \Omega\CC\to \PP}\colon \Art_{\Perm}\arrow[r] & \sS;\hspace{4pt} A\arrow[r, mapsto] & \Map_{\cat{Op}^\mm{nu}}\Big(\ol{\Omega}\CC, \PP\otimes A\Big)
\end{tikzcd}\end{equation}
sending every Artin permutative algebra to the space of nonunital operad maps $\ol{\Omega}\CC\rt \PP\otimes A$. Using that $\PP\otimes (-)$ sends homotopy pullbacks of permutative algebras to homotopy pullbacks of nonunital operads, one sees that $F$ is a permutative FMP. Adding units to our operads, one can think of $F$ as the FMP describing permutative deformations of the \emph{trivial} map of operads
$$\begin{tikzcd}
\ul{0}\colon \Omega\CC\arrow[r, "\epsilon"] & k\arrow[r, "1"] & \PP.
\end{tikzcd}$$
In particular, the restriction of $\Def_{\ul{0}\colon \Omega\CC\to \PP}$ to Artin commutative algebras coincides with the commutative deformation problem \eqref{eq:def operad maps} for the trivial map $\ul{0}$. Note that this map $\ul{0}$ corresponds to the zero Maurer--Cartan element in $\mf{g}$, i.e.\ the Lie algebra $\mf{g}^{\ul{0}}$ is simply the Lie algebra underlying the convolution pre-Lie algebra $\mf{g}$.

\begin{proposition}\label{prop:permutative deformations from convolution}
Let $\CC$ be a $k$-cooperad, $\PP$ a (not necessarily augmented) $k$-linear operad, and $\mf{g}$ their convolution pre-Lie algebra. For every strictly Artin permutative algebra $A$, there is an equivalence $$
\mm{Def}_{\ul{0}\colon \Omega\CC\to \PP}(A) \simeq \MC_{\mf{g}}(A).
$$
\end{proposition}
In other words, the convolution pre-Lie algebra $\mf{g}$ classifies deformations of the trivial operad map $\Omega\CC\rt \PP$. When $\PP=\End(V)$ is the endomorphism operad of a complex $V$, $\mf{g}$ therefore classifies deformations of the \emph{trivial} $\Omega\CC$-algebra structure on $V$.
\begin{proof}
The proof is analogous to that of Proposition \ref{prop:comm deformations from convolution}: a fibrant simplicial resolution $\PP_\bullet$ of $\PP$ induces a simplicial resolution $\mf{g}_\bullet$ of the pre-Lie algebra $\mf{g}$  and Lemma \ref{lem:maurer-cartan prelie} then asserts that $\MC_{\mf{g}}(A)$ can be modeled by the simplicial set of Maurer--Cartan elements $\MC(\mf{g}_\bullet\otimes A)$. Unraveling the definitions, one sees that this simplicial set coincides with the simplicial set of maps of nonunital operads $\ol{\Omega}\CC\rt \PP_\bullet\otimes A$, which in turn models the mapping space $\Map_{\cat{Op}^{\mm{nu}}}(\ol{\Omega}\CC, \PP\otimes A)=\mm{Def}_{\ul{0}\colon \Omega\CC\to \PP}(A)$.
\end{proof}

\subsection{Operadic deformation theory}\label{sec:operadicdeftheory}
In the previous section, we have seen how the convolution pre-Lie algebra
$$
\mf{g}=\prod_{p\geq 0} \Hom\big(\ol{\CC}(p), \PP(p)\big)^{\Sigma_p}
$$
classifies permutative deformations of the trivial map of operads $\Omega\CC\rt k\rt \PP$. In fact, the pre-Lie algebra $\mf{g}$ arises from an even richer algebraic structure: the sequence of mapping complexes $\Hom\big(\ol{\CC}(p), \PP(p)\big)$ forms a nonunital operad, the \emph{convolution operad} of $\CC$ and $\PP$. In this section, we will explain how this additional algebraic structure can be understood from a deformation theoretic point of view.

Let us start by considering formal moduli problems that are classified by (monochromatic) nonunital operads. Here we take a nonunital operad to be a symmetric sequence equipped with partial composition maps satisfying the usual associativity and equivariance conditions (see also Definition \ref{def:op of symops} below). In particular, the category of nonunital operads is equivalent to that of (\emph{augmented}) $k$-operads by the functors $\PP\mapsto \PP^+$ adding a unit (in arity $1$) and $\PP\mapsto \ol{\PP}$ taking the augmentation ideal. In particular, in analogy to the functor \eqref{diag:dual prelie}, we have a functor
$$\begin{tikzcd}
\mf{D}\colon \cat{Op}^\mm{nu}\arrow[r] & \cat{Op}^\mm{nu, op}; \quad \mathscr{P}\arrow[r, mapsto] & \mf{D}(\mathscr{P})=\ol{\Bar}(\mathscr{P}^+)^\vee
\end{tikzcd}$$
taking (the augmentation ideal of) the dual operad in the sense of Definition \ref{def:dualoperad}. We can then mimick the construction in Theorem \ref{thm:main prelie} to associate to a nonunital operad a formal moduli problem indexed by the $\infty$-category $\Art_{\cat{Op}}$ of \emph{Artin nonunital operads}. Theorem \ref{thm:mainthm} (or more precisely, Theorem \ref{thm:biduals}) then yields:
\begin{theorem}\label{thm:main operad}
For every nonunital operad $\mathscr{P}$, consider the functor
$$\begin{tikzcd}
\MC_{\mathscr{P}}\colon \Art_{\cat{Op}}\arrow[r] & \sS; \quad \mathscr{R}\arrow[r, mapsto] & \Map_{\cat{Op}^\mm{nu}}\big(\mf{D}(\mathscr{R}), \mathscr{P}\big). 
\end{tikzcd}$$
This establishes an equivalence of $\infty$-categories $\MC\colon \cat{Op}^\mm{nu}\rt \cat{FMP}_{\cat{Op}}$.
\end{theorem}
\begin{remark}\label{rem:idempotent complete}
Throughout this section, it will be convenient to slightly enlarge the subcategory $\ArtOp\subseteq \cat{Op}^\mm{nu}$ of Artin operads so that it is also closed under \emph{retracts}. This will not change the theory, since an $\sS$-valued diagram on retracts of Artin operads is determined uniquely by its value on Artin operads as in Definition \ref{def:smallalgebra}.
The Artin operads in this slightly broader sense then have the following simple description, as in Lemma \ref{lem:smallhomotopygroups}: they are those nonunital symmetric operads $\mathscr{R}$ such that $H^*(\mathscr{R})$ is concentrated in nonpositive degrees, of finite total dimensional (summing over degrees and arities) and $H^0(\mathscr{R})$ is a nilpotent operad. 
\end{remark}
\begin{example}\label{ex:coEnd FMP}
Let $V$ be a chain complex in degrees $\leq 0$ and consider the functor associating to each Artin nonunital operad $\mathscr{R}$ the space of $\mathscr{R}^+$-algebras $A$ together with an equivalence $k\circ_{\mathscr{R}^+} A\simeq V$ after inducing along the augmentation of $\mathscr{R}^+$. One can show that this defines an operadic formal moduli problem (using that the composition product $\circ$ is exact in the first variable). The operad classifying this is the \emph{coendomorphism operad} 
$$
\mm{coEnd}(V)(n)=\Hom(V, V^{\otimes n}).
$$ 
To exemplify this, suppose that $\mathscr{R}$ is strictly finite dimensional, nilpotent and in degrees $\leq 0$. Then there is an isomorphism $\mf{D}(\mathscr{R})=\ol{\Bar}(\mathscr{R})^\vee\cong \ol{\Omega}(\mathscr{R}^{+\vee})$ between the dual operad of $\mathscr{R}$ and the cobar construction of the cooperad $\mathscr{R}^{+\vee}$ (which is a cofibrant operad). An operad map $\phi\colon \ol{\Omega}(\mathscr{R}^{+\vee})\rt \mm{coEnd}(V)$ is uniquely determined by its restriction to the generators. This restriction in turn corresponds to a collection of equivariant maps $\delta_n\colon \mathscr{R}^\vee(n)\otimes V\rt V^{\otimes n}$, or equivalently (since $\mathscr{R}$ is finite-dimensional), to a map $\delta\colon V\rt \mathscr{R}\circ V$. Unraveling the definitions, $\phi$ is a map of operads precisely when $\mathscr{R}^+\circ V$ is a dg-$\mathscr{R}^+$-algebra whose differential is given on generators by $d+\delta$. Such an algebra determines a deformation of $V$ over $\mathscr{R}$.

Notice that this recovers Example \ref{ex:defmodule} by seeing an associative algebra as an operad in arity $1$.

\end{example}
To see how Theorem \ref{thm:main operad} fits into the framework of Theorem \ref{thm:mainthm}, let us recall how (monochromatic) operads themselves are algebras over a coloured operad.

\subsubsection*{The operad of nonunital operads}
Let us start by recalling the operad whose algebras are \emph{nonunital nonsymmetric operads}. To this end, consider the linear category $\nonsym$ having nonnegative integers as objects and morphisms
$$
\nonsym(m,n)\coloneqq\begin{cases}
0\quad\mathrm{if}\quad m\neq n \\
k\quad\mathrm{else}.
\end{cases}
$$
Note that $\nonsym$-operads are just (augmented) $\mathbb{Z}_{\geq 0}$-coloured operads. Nonunital nonsymmetric operads then arise as algebras over the following quadratic operad:

\begin{definition}[see \cite{van2003coloured}]\label{def:op of nsops}
Let $a,b$ and $c$ denote colours in $\mathbb Z_{\geq 0}$ and let $\mathscr{O}^{\mm{ns}}$ be the $\nonsym$-operad generated by $\circ_i\colon (a, b)\rt a+b-1$ for $i=1, \dots, a$, subject to the relations
		\begin{equation}\label{diag:compnonsymV2}\begin{tikzcd}[column sep=2.6pc]
		{(a, b, c)} \arrow[r, "{(a, \circ_j)}"]\arrow[d, "{(\circ_i, c)}"{swap}] & {(a, b+c-1)}\arrow[d, "{\circ_i}"]\\
					{(a+b-1, c)}\arrow[r, "{\circ_{i+j-1}}"{swap}] & a+b+c-2
		\end{tikzcd}\end{equation}
		for $1\leq i\leq a$ and $1\leq j\leq b$, 
		and
		\begin{equation}\label{diag:compnonsym2V2}\begin{tikzcd}[column sep=2.6pc]
					{(a, b, c)} \arrow[d, "{(\circ_j,c)}"{swap}]\arrow[r] & {(a,c,b)}\arrow[r, "{(\circ_i,b)}"] & {(a+c-1,b)}\arrow[d, "{\circ_{j+c-1}}"]\\
					{(a+b-1, c)}\arrow[rr, "\circ_{i}"{swap}] && a+b+c-2
		\end{tikzcd}\end{equation}
		for $1\leq i<j\leq a$. 
\end{definition}

To discuss the case of nonunital symmetric operads, let us introduce another linear category $k[\Sigma]$, having objects the nonnegative integers, and 
\[
k[\Sigma](p,q)\coloneqq\left\{\begin{array}{ll}
0 & \text{if } q\neq p \\
k[\Sigma_p] & \text{else.}\end{array}\right.
\]
Note that $k[\Sigma]\simeq k[\Sigma]^{\op}$ by taking inverse permutations. There is a quadratic $k[\Sigma]$-operad $\mathscr{O}^{\mm{sym}}$ whose algebras are nonunital symmetric operads:

\begin{definition}[{see \cite[Definition 1.7]{dehling2015symmetric}}]\label{def:op of symops}
	 Let $\mathscr{O}^{\mm{sym}}$ be the unital \textbf{non-augmented} operad generated by $\circ_i\colon (a, b)\rt a+b-1$ as in Definition \ref{def:op of nsops}
		and $\sigma\colon a\rt a$ for $\sigma\in\Sigma_a$, subject to the equations \eqref{diag:compnonsymV2} and \eqref{diag:compnonsym2V2}, 
		together with the group structure equations
		$$
		\sigma\cdot \tau = (\sigma\tau)\colon a\rt a \qquad\qquad \sigma, \tau\in\Sigma_a, 
		$$ 
		and the equations
		\begin{equation}\begin{tikzcd}
					{(a, b)}\arrow[r, "{(a, \tau)}"]\arrow[d, "\circ_i"{swap}] & {(a, b)}\arrow[d, "\circ_i"] & & {(a, b)}\arrow[r, "{(\sigma, b)}"]\arrow[d, "\circ_i"{swap}] & {(a, b)}\arrow[d, "\circ_{\sigma(i)}"]\\
					a+b-1\arrow[r, "\tau/i"{swap}] & a+b-1 & & a+b-1\arrow[r, "i/\sigma"{swap}] & a+b-1
		\end{tikzcd}\end{equation}
		where for $\tau/i$ and $i/\sigma$ are some permutations of $a+b-1$ determined from $\tau$, $\sigma$ and the number $i$.
				Finally, we impose the relation that the identity of the group $1_a\in \Sigma_a$ is identified with the operadic unit $1_a \in \mathscr{O}^{\mm{sym}}(a;a)$.

\end{definition}
\begin{remark}
Unlike the usual conventions in this manuscript, we are forced to consider $\mathscr{O}^{\mm{sym}}$ as a unital non-augmented operad. The reason for this is that there is no way to define an augmentation since the relations $\sigma \cdot \sigma^{-1}$ produce the unit. Comparing with \cite{dehling2015symmetric}, we notice that there are other problems with this presentation that heuristically come from seeing the symmetric groups as additional structure: the presentation is not quadratic (and e.g.\ a result like Proposition \ref{prop:vdL} is not expectable for the operad of symmetric operads).
\end{remark}
Instead, it is more convenient to consider $\mathscr{O}^{\mm{sym}}$ as an operad relative to $k[\Sigma]$. In this case, there is a natural augmentation $\mathscr{O}^{\mm{sym}}\rt k[\Sigma]$, so that $\mathscr{O}^{\mm{sym}}$ is a $k[\Sigma]$-operad in the sense of Section \ref{sec:operads over dg-cat}. We will show in Section \ref{sec:relativekoszul} that $\mathscr{O}^\mm{sym}$ is Koszul self-dual relative to $k[\Sigma]$, i.e.\ that $\mathscr{O}^\mm{sym}\simeq \mf{D}(\mathscr{O}^{\mm{sym}})$. Theorem \ref{thm:main operad} then arises as a special case of Theorem \ref{thm:mainthm}:
\begin{proof}[Proof of Theorem \ref{thm:main operad}]
We apply the main Theorem \ref{thm:mainthm} (or more precisely, Theorem \ref{thm:biduals}) to the $k[\Sigma]$-operad $\mathscr{O}^\mm{sym}$. To see that $\mathscr{O}^\mm{sym}$ indeed satisfies condition (3), note that $\mathscr{O}^\mm{sym}(1)=k[\Sigma]$, so that
	$$
	k[\Sigma]\circ^h_{\mathscr{O}^\mm{sym}} k[\Sigma] \simeq \Bar_{k[\Sigma]}(\mathscr{O}^\mm{sym})\simeq \big(\mathscr{O}^\mm{sym}\{- 1\}\big)^\vee.
	$$
	To see that this is concentrated in increasingly negative cohomological degrees as the arity increases, one uses the same argument as in the proof of Theorem \ref{thm:mainthmkoszul}. Finally, the functor $\mf{D}\colon \Alg_{\mathscr{O}^\mm{sym}}\rt \Alg_{\mathscr{O}^\mm{sym}}^{\op}$ appearing in Theorem \ref{thm:biduals} coincides with the functor taking dual operads, by Proposition \ref{prop:bar of operad is operadic bar}.
\end{proof}

\subsubsection*{Connecting permutative and operadic deformation theories}
Let us now spell out the relation between permutative and operadic deformation theory. First, we have seen in Section \ref{sec:permutative} that every permutative algebra $A$ defines functorially a nonunital symmetric operad $L(A)$, given in each arity by the permutative algebra $A$ and with all partial compositions given by the product of $A$. When $A$ is Artin, $L(A)$ is not quite Artin (it is concentrated in all arities, hence not finite-dimensional). However, one can show (Lemma \ref{lem:pro-artin}) that it arises as the limit of a (canonical) object $\hat{L}(A)\in \cat{Pro}(\Art_{\cat{Op}})$, given by a \emph{pro-system} of Artin operads $\dots\to L(A)_1\to L(A)_{0}$ which is eventually constant in each individual arity.

On the other hand, if $\PP$ is a nonunital operad, then $\prod_{p\geq 0} \PP(p)^{\Sigma_p}$ can be equipped with the structure of a pre-Lie algebra \cite[Proposition 5.3.17]{LodayVallette2012}. The equivalences from Theorem \ref{thm:main prelie} and \ref{thm:main operad} are intertwined by these constructions:
\begin{proposition}\label{prop:operadic to permutative}
There is a commuting square of $\infty$-categories
$$\begin{tikzcd}
\cat{Op}^{\mm{nu}}\arrow[d, "\prod_p (-)^{\Sigma_p}"{swap}]\arrow[r, "\sim"] & \cat{FMP}_{\mathscr{O}^\mm{sym}}\arrow[d, "\hat{L}^*"]\\
\Alg_{\preLie}\arrow[r, "\sim"] & \cat{FMP}_{\Perm}.
\end{tikzcd}$$
Here the right vertical functor sends an operadic formal moduli problem $F$ to the permutative formal moduli problem $\hat{L}^*F(A)= F(\hat{L}(A))\coloneqq\lim_n F(L(A)_n)$.
\end{proposition}
In other words, a pre-Lie algebra $\mf{g}$ arises as $\mf{g}=\prod_{p\geq 0} \PP(p)^{\Sigma_p}$ if and only if the corresponding permutative deformation problem lifts to an operadic deformation problem.

Before addressing the proof of Proposition \ref{prop:operadic to permutative}, let us first describe its implications to the deformation theory of operad maps.

\begin{example}[Deforming trivial morphisms of operads (continued)]
Let $\mathscr{C}$ be a $k$-cooperad and $\PP$ a $k$-linear operad. The convolution pre-Lie algebra $\mathfrak{g}=\prod_{p\geq 0}\Hom_k\big(\ol{\mathscr{C}}(p),\PP(p)\big){}^{\Sigma_p}$ then arises from the (nonunital) \emph{convolution operad}
$$
\Conv(\mathscr{C},\PP)(p)\coloneqq\Hom_k\big(\ol{\mathscr{C}}(p),\PP(p)\big),
$$
whose operad structure arises from the convolution of the (nonunital) cocomposition on $\ol{\mathscr{C}}$ and the composition on $\PP$. 

The operadic deformation problem associated to the convolution operad sends an Artin operad $\mathscr{R}$ to the space of nonunital operad maps $\ol{\Omega}(\CC)\rt \PP\otimes_{\mm{H}} \mathscr{R}$ to the Hadamard (i.e.\ aritywise) tensor product of $\PP$ and $\mathscr{R}$. To see this, note that for a strictly Artin operad $\mathscr{R}$, the following maps are in 1-1 correspondence:
\begin{enumerate}\setlength{\itemsep}{1pt}
\item nonunital operad maps $\ol{\Omega}(\CC)\rt \PP\otimes_{\mm{H}} \mathscr{R}$,
\item twisting morphisms $\ol{\CC}\drt \PP\otimes_{\mm{H}} \mathscr{R}$,
\item twisting morphisms $\mathscr{R}^\vee\drt \Conv(\CC, \PP)$,
\item nonunital operad maps $\mf{D}(\mathscr{R})\rt \Conv(\CC, \PP)$.
\end{enumerate}
Here (1) $\Longleftrightarrow$ (2) and (3) $\Longleftrightarrow$ (4) follow from the universal property of the cobar construction, together with the fact that $\mf{D}(\mathscr{R})= \ol{\Bar}(\mathscr{R}^+)^\vee\cong \ol{\Omega}(\mathscr{R}^{+\vee})$ for finite dimensional $\mathscr{R}$. The bijection (2) $\Longleftrightarrow$ (3) follows by unraveling the definition of a twisting morphism (see e.g.\ Definition \ref{constr:k-twisting morphisms}); at the level of the underlying maps of symmetric sequences, it simply sends a $\Sigma_p$-invariant map $\ol{\CC}(p)\rt \PP(p)\otimes \mathscr{R}(p)$ to the adjoint map $\mathscr{R}(p)^\vee\rt \Hom(\CC(p), \PP(p))^{\Sigma_p}$. Replacing $\PP$ by a fibrant resolution, one obtains an equivalence between spaces of operad maps as in (1) and (4).

Proposition \ref{prop:permutative deformations from convolution} now asserts that the permutative deformation functor associated with $\mathfrak{g}=\prod \Conv(\CC, \PP)(p)^{\Sigma_p}$ sends a strictly Artin permutative algebra $A$ to the mapping space
$$
\lim_n \Map_{\cat{Op}^\mm{nu}}\big(\ol{\Omega}(\mathscr{C}), \mathscr{P}\otimes_{\mm{H}} L(A)_n\big)\simeq \Map_{\cat{Op}^{\mm{nu}}}(\ol{\Omega}(\mathscr{C}), \mathscr{P}\otimes_{\mm{H}} L(A)\big)
$$
(here we use that the tower of $L(A)_n$ is eventually constant in each arity). But the Hadamard tensor product $\mathscr{P}\otimes_{\mm{H}} L(A)$ simply coincides with the levelwise tensor product of $\PP$ with the permutative algebra $A$. In other words, we precisely recover the permutative deformation problem $\mm{Def}_{\ul{0}\colon \Omega\CC\to \PP}$ \eqref{eq:perm def of operad map}.
\end{example}

In the remainder of this section, we will prove Proposition \ref{prop:operadic to permutative} by describing the relation between the operad $\mathscr{O}^\mm{sym}$ and the operad $\Perm$. This requires comparing operads defined over a different base: $\Perm$ is defined over the base field $k$ and $\mathscr{O}^\mm{sym}$ over $k[\Sigma]$. 
To do this, note that each $k$-linear symmetric sequence $\mathscr{M}$ gives rise to a $k[\Sigma]$-symmetric sequence 

	$$
L(\mathscr{M})(n_1, \dots, n_k; n_0)=\left\{\begin{array}{cl} \mathscr{M}(k) & \text{if } n_1+\dots+n_k=n_0+k-1\\
	0 & \text{otherwise}\end{array}\right.
$$
carrying a trivial $\Sigma_{n_1} \times \dots \times \Sigma_{n_0}$ action. This defines a functor 
$L\colon \cat{BiMod}^{\Sigma, \dg}_{k} \to \cat{BiMod}^{\Sigma, \dg}_{k[\Sigma]}$.

\begin{proposition}\label{prop:adjunction Perm Osym}
The functor $L$ extends to a functor 
$$L\colon \cat{Op}^{\dg}_k\rt \cat{Op}^{\dg}_{k[\Sigma]}$$
which preserves (Koszul) quadratic operads and their quadratic duals. Furthermore, at the level of algebras there is an adjoint pair
\[\begin{tikzcd}
	L\colon \cat{Alg}_{\mathscr{Q}}^{\dg}\arrow[r, yshift=0.8ex] & \cat{Alg}_{L(\mathscr{Q})}^{\dg}\colon R.\arrow[l, yshift=-0.8ex]
\end{tikzcd}\]
If $A$ is a $\QQ$-algebra, then the underlying $k[\Sigma]$-module of $L(A)$ is the constant one $L(A)(p)=A$. The right adjoint $R$ is given by $R(\cat{A})=\prod_p \cat{A}(p)^{\Sigma_p}$.
\end{proposition}

The proof of Proposition \ref{prop:adjunction Perm Osym} is not difficult, but due to some technical points we leave the details to Section \ref{sec:adjunctions for Perm and Osym} (where quadratic duality of $k[\Sigma]$-operads is discussed as well). For now, let us point out that the explicit formula of the functor $L\colon \cat{Op}^{\dg}_k\rt \cat{Op}^{\dg}_{k[\Sigma]}$ shows that it commutes with linear duality and preserves quasi-isomorphisms. Furthermore, Proposition \ref{prop:adjunction Perm Osym} implies that $\mf{D}(L(\QQ))\simeq L(\QQ)^!\cong L(\QQ^!)$ for any binary Koszul operad $\QQ$, and that Theorem \ref{thm:mainthmkoszul} applies to such operads as well (cf.\ Observation \ref{obs:relative koszul case}).

We can now express the fact that every permutative algebra gives rise to an operad (constant in every arity) in terms of a map of $k[\Sigma]$-operads.

\begin{lemma}
	There is a natural map of binary quadratic $k[\Sigma]$-operads $\mathscr{O}^\mm{sym}\rt L(\Perm)$. Koszul dually, this induces a map of binary quadratic $k[\Sigma]$-operads $L(\preLie)\rt \mathscr{O}^\mm{sym}$.
\end{lemma}
\begin{proof}
	By Proposition \ref{prop:adjunction Perm Osym}, the $k[\Sigma]$-operad $L(\Perm)$ is generated by operations $\mu\colon (a, b)\to a+b-1$ which are $\Sigma_a\times\Sigma_b\times \Sigma_{a+b-1}$-invariant, subject to the associative and permutative relation \eqref{eq-perm-alg}. At the level of quadratic data, the map $\mathscr{O}^\mm{sym}\rt L(\Perm)$ then sends each operation $\circ_i\colon (a, b)\to a+b-1$ to the operation $\mu\colon (a, b)\to a+b-1$.
\end{proof}

Notice that both adoint functors $L$ and $R$ from Proposition \ref{prop:adjunction Perm Osym} preserve quasi-isomorphisms (they actually form a Quillen pair), so that they induce an adjoint pair on $\infty$-categories.
\begin{lemma}\label{lem:pro-artin}
Let $\QQ$ be a Koszul binary quadratic operad in degree $0$ and $A\in \Art_{\QQ}$. Then the formal moduli problem
$$\begin{tikzcd}
\mm{Spf}\big(L(A)\big)=\Map_{L(\QQ)}\big(L(A), -\big)\colon \Art_{L(\QQ)}\arrow[r] & \sS.
\end{tikzcd}$$
is corepresentable by a pro-Artin $L(\QQ)$-algebra $\hat{L}(A)$ which is eventually constant in each fixed arity. 
\end{lemma}
\begin{proof}
The formal moduli problem $\mm{Spf}(L(A))$ is classified by the dual $L(\QQ^!)$-algebra $\mf{D}(L(A))$. It then suffices to verify that this $L(\QQ^!)$-algebra can be written as the colimit of a sequence $0=\mf{D}(L(A))_0\to \mf{D}(L(A))_1\to \dots$ where each $\mf{D}(L(A))_n$ is obtained from the previous one by adding a positive degree cell (cf.\ \eqref{diag:cellattachment}), so that in total we add only finitely many cells in each arity. Indeed, by Theorem \ref{thm:biduals}, this means that each $\mf{D}(L(A))_n$ is the dual of an Artin $L(\QQ)$-algebra $L(A)_n$, giving the desired pro-system\footnote{Technically, the $L(A)_n$ thus obtained are not Artin algebras, but only retracts of such. One can always enlarge the subcategory of Artin $L(\QQ^!)$-algebras to include such retracts, cf.\ Remark \ref{rem:idempotent complete}.}. 

Now $L$ is monoidal and preserves duals, so we can identify $\mf{D}(L(A))=L(\mf{D}(A))$. Since $A$ is Artin, $\mf{D}(A)$ is obtained by such a finite process of cell attachments (Theorem \ref{thm:biduals}). Concretely, this means that $\mf{D}(A)$ arises as a quasifree $\QQ^!$-algebra generated by $x_1, \dots, x_n$, where $d(x_i)$ is an expression in $x_1, \dots, x_{i-1}$. Then $\mf{D}(L(A))\simeq L(\mf{D}(A))$ is a quasifree $L(\QQ^!)$-algebra generated by $x_{1, p}, \dots, x_{n, p}$ for each arity $p$. One now obtains the desired sequence by giving the generator $x_{i, p}$ weight ${p+i\choose 2}+i$ and letting $\mf{D}(L(A))_n$ be the subalgebra on the generators of weight $\leq n$.
\end{proof}
\begin{lemma}
Let $\QQ$ be a Koszul binary quadratic operad over $k$ with Koszul dual $\QQ^!$. Then there is a commuting diagram of $\infty$-categories
\begin{equation}\label{diag:Q to LQ}\begin{tikzcd}
\Alg_{L(\QQ^!)}\arrow[r, "\sim"]\arrow[d, "R"{swap}] & \FMP_{L(\QQ)}\arrow[d, "\hat{L}^*"]\\
\Alg_{\QQ^!}\arrow[r, "\sim"] & \FMP_{\QQ}
\end{tikzcd}\end{equation}
where the right vertical functor is given by $\hat{L}^*F(A)=F(\hat{L}(A))\coloneqq\lim_n F(L(A)_n)$.
\end{lemma}
In other words, $\hat{L}^*F$ parametrizes deformations along pro-Artin $L(\QQ)$-algebras of the form $L(A)$, with $A$ a Artin $\QQ$-algebra. 
\begin{proof}
The functor $L\colon \Mod_k\rt \Mod_{k[\Sigma]}$ taking constant symmetric sequences preserves tensor products and linear duals, and hence commutes with taking the dual of the (operadic) bar construction. In other words, we obtain a commuting diagram of $\infty$-categories
$$\begin{tikzcd}
\Alg_{\QQ}^{\op}\arrow[d, "L"{swap}]\arrow[r, "\Bar(-)^\vee"] & \Alg_{\QQ^!}\arrow[r, "\sim"]\arrow[d, "L"] & \cat{FMP}_{\QQ}\arrow[d, dotted, "\hat{L}_!"]\\
\Alg^{\op}_{L(\QQ)}\arrow[r, "\Bar(-)^\vee"] & \Alg_{L(\QQ)^!}\arrow[r, "\sim"] & \cat{FMP}_{L(\QQ)}.
\end{tikzcd}$$
The composite horizontal functors have a very simple description: they send a $\QQ$-algebra $A$ to the formal moduli problem $\mm{Spf}(A)=\Map_{\QQ}(A, -)$ of Example \ref{ex:formal spectrum}. Consequently, the right vertical functor $\hat{L}_!$ (which is defined uniquely by the above diagram) sends the formal moduli problem corepresented by an Artin $\QQ$-algebra $A$ to the formal moduli problem $\mm{Spf}(L(A))$. By Lemma \ref{lem:pro-artin}, this formal moduli problem is pro-represented by $\hat{L}(A)$, i.e.\ $\mm{Spf}(L(A))=\colim_n \mm{Spf}(L(A)_n)$ with $L(A)_n$ Artin. Passing to right adjoints then yields the desired square \eqref{diag:Q to LQ}.
\end{proof}

\begin{proof}[Proof of Proposition \ref{prop:operadic to permutative}]
Compose the square \eqref{diag:Q to LQ}, with $\QQ=\Perm$ and $\QQ^!=\preLie$, with the square 
\[
\begin{tikzcd}
	\cat{Op}^{\mm{nu}} \arrow[d] \arrow[r, "\sim"] &	\cat{FMP}_{\mathscr{O}^{\mm{sym}}} \arrow[d]										\\
	\cat{Alg}_{L(\preLie)}     \arrow[r, "\sim"]
&	\cat{FMP}_{L(\Perm)}
\end{tikzcd}
\]
obtained from naturality with respect to the map of $k[\Sigma]$-operads $\mathscr{O}^\mm{sym}\rt L(\Perm)$ (Proposition \ref{prop:naturalityfmp}).
\end{proof}

\subsection{Splendid operads}\label{sec:splendid}
The main technical condition of Theorem \ref{thm:mainthm} is Condition \ref{it:splendidness}, which asserts that the operad is splendid in the following sense:
\begin{definition}\label{def:splendid}
Let $\cat{P}$ be a $\base$-operad. We will say that $\cat{P}$ is \emph{splendid} if its 0-reduced part $\mathscr{P}^{\geq 1}$ (the suboperad such that $\mathscr{P}^{\geq 1}(0)=0$ and agrees in other arities) satisfies the following condition: the derived relative composition product
$$
\mathscr{P}(1)\circ^h_{\mathscr{P}^{\geq 1}} \mathscr{P}(1)
$$
is eventually highly connective (Definition \ref{def:highlyconn}).
\end{definition}
\begin{remark}
At least for connective $\mathscr{P}$, this definition should be considered as a homotopy-invariant reformulation of the following condition: $\mathscr{P}^{\geq 1}$ admits a free resolution whose generators are in increasingly negative degrees (as the arity increases). See Section \ref{sec:freeoperads}.
\end{remark} 
An immediate natural question to ask is therefore whether a given operad is splendid. Let us start by making some general observations about the property of being splendid. First of all, let us observe that more Koszul operads are splendid than just the binary ones considered in the Introduction, so that Theorem \ref{thm:mainthmkoszul} applies to these as well:

\begin{observation}\label{obs:koszul case}
	A (non-necessarily binary) Koszul quadratic operad $T(E)/(R)$ living in nonpositive degrees generated by a symmetric sequence $E$ with generators in bounded arity (i.e. $E(n)=0$ for $n\gg 0$) is splendid. Indeed, its Koszul resolution has generators sitting in increasingly negative degrees by the same argument as in the proof of Theorem \ref{thm:mainthmkoszul}.	
\end{observation}

\begin{example}\label{ex:RHT}
For a pair of Koszul dual quadratic operads $(\PP, \PP^!)$ in degree $0$, Observation \ref{obs:koszul case} is of course symmetric in $\PP$ and $\PP^!$. For example, in addition to commutative formal moduli problems being classified by Lie algebras, formal moduli problems over \emph{Artin Lie algebras} are classified by (nonunital) commutative algebras. We do not know of a good geometric interpretation of this equivalence, but let us point out the following. 

Suppose we are working over $k=\mathbb{Q}$ and consider $\mathbb{Q}$ as a nonunital commutative algebra. Then the formal moduli problem $\MC_{\mathbb{Q}}\colon \Art_{\Lie}\rt \sS$ sends an Artin Lie algebra $\mf{g}$, i.e.\ one with $H^*(\mf{g})$ finite dimensional, nilpotent and in nonpositive degrees, to the corresponding rational homotopy type. Indeed, Theorem \ref{thm:biduals} identifies $\MC_{\mathbb{Q}}(\mf{g})\simeq \Map_{\Com}(\ol{C}_{\mm{CE}}^*(\mf{g}), \mathbb{Q})$ with the spatial realization of the corresponding Sullivan model. More generally, for any unital commutative $A$ one can identify $\MC_A(\mf{g})$ with the $A$-points of the (rational) schematic homotopy type corresponding to $\mf{g}$.
\end{example}

\begin{example}[and non-example]
A quadratic operad that does not fit the constraints of the previous observation is the \emph{gravity operad} $\overline{\mathsf{Grav}}$ \cite[Theorem 4.5]{getzler1994two}. The operad $\overline{\mathsf{Grav}}$ is generated by a sequence $E$ such that $E(n)$ is 1-dimensional and concentrated in degree $-1$. Clearly such an operad cannot be splendid, as the generators of a resolution need to cover all generators of $\overline{\mathsf{Grav}}$.

In fact, there is some ambiguity in the literature regarding the degrees of these operads. We denote by $\mathsf{Grav}$ what we will also call the gravity operad, which has the same quadratic presentation but with generators $V(n)$ a 1-dimensional space concentrated in degree $2-n$ (in other words, $\mathsf{Grav}$ is obtained from $\overline{\mathsf{Grav}}$ by reversing the degrees and operadicaly shifting down by $1$).

The operad $\mathsf{Grav}$ is Koszul and its Koszul dual is the operad $\mathsf{HyperCom}$ of hypercommutative algebras \cite{getzler1995operads}, generated by one operation in arity $n$ in degree $2(n-2)$ for all $n\geq 2$.
It follows that $\mathsf{Grav}$ is splendid and from Theorem \ref{thm:mainthm} we deduce that the $\infty$-category $\cat{FMP}_{\mathsf{Grav}}$ is equivalent to the $\infty$-category of hypercommutative algebras. 

Note that one cannot exchange the roles of $\mathsf{Grav}$ and $\mathsf{HyperCom}$ in this statement: $\mathsf{HyperCom}$ is not splendid and Theorem \ref{thm:mainthm} does \emph{not} hold for hypercommutative formal moduli problems.
\end{example}
\begin{remark}\label{rem:BM}
Suppose that $\mathscr{P}$ is a monochromatic augmented operad which is 1-reduced, i.e.\ $\PP(0)=0$ and $\PP(1)=k\cdot 1$. If $\PP$ is connective, then the shifted operad $\mathscr{P}\{1\}$ satisfies the conditions of Theorem \ref{thm:mainthm}. The case where $\PP$ is in addition aritywise finite-dimensional also appears in work of Brantner--Mathew \cite[Corollary 5.59]{brantner2019deformation} (see also \cite{ching2019derivedkoszul}).
\end{remark}

Next, note that an operad typically satisfies the conditions of Theorem \ref{thm:mainthm} as soon as its cohomology does:

\begin{lemma}\label{lem:homology is splendid}
Let $\PP$ be a (coloured) connective operad over $\mathbb{Q}$. If $H^*(\PP)$ is splendid, then $\PP$ is splendid as well.
\end{lemma}
\begin{proof}
We can assume that $\PP$ is $0$-reduced and consider the simplicial resolution
$$\begin{tikzcd}
\dots\arrow[r]\arrow[r, yshift=1ex] \arrow[r, yshift=-1ex] & \PP(1)\circ\PP\circ\PP(1)\arrow[r, yshift=0.5ex]\arrow[r, yshift=-0.5ex] & \PP(1)\circ\PP(1)\arrow[r, dotted] & \PP(1)\circ^h_{\PP}\PP(1).
\end{tikzcd}$$
Taking (at each tuple of colours) the corresponding normalized cochains, we obtain a (cohomologically) $\mathbb{Z}_{\leq 0}\times\mathbb{Z}_{\leq 0}$-graded bicomplex, with an associated convergent spectral sequence
$$
E_2^{r, *}=H^r\Big(H^*\PP(1)\circ^h_{H^*\PP} H^*\PP(1)\Big)\quad\Longrightarrow\quad H^{r+*}\big(\PP(1)\circ^h_{\PP}\PP(1)\big).
$$
Here $H^*\PP(1)\circ^h_{H^*\PP} H^*\PP(1)$ is computed in the category of (nonpositively) graded symmetric sequences of complexes. If $H^*\PP$ is splendid, then $p$-ary part of the $E_1$-page is concentrated in degrees $*\leq 0$ and $r\leq f(p)\leq 0$, with $f(p)\xrightarrow{p\to \infty} -\infty$. Then the $p$-ary part of $\PP(1)\circ^h_{\PP}\PP(1)$ is also concentrated in cohomological degrees $\leq f(p)$, and we conclude that $\PP$ is splendid.
\end{proof}
\begin{example}[Variants of the little discs operads]
	Using this lemma one can show that the little $n$-discs operad $\mathbb{E}_n$ is splendid without having to use that it is quasi-isomorphic to its homology $e_n$. Indeed, $e_n$ is a binary quadratic Koszul operad \cite[Section 13.3.16]{LodayVallette2012} and is therefore splendid.
		In the next section we will show that the homology of the framed little $n$-discs operad is also splendid.
	
	Another application of Lemma \ref{lem:homology is splendid} involves the result of Hoefel and Livernet \cite{hoefel2013spectral} that the homology of the Swiss--Cheese operad ($\mathfrak{sc}^{\text{vor}}$ in loc. cit.) is a quadratic binary Koszul colored operad. This shows the Swiss--Cheese operad is  splendid, even though we do not know a simple model for its dual.

\end{example}

In the following subsections we will look at a few examples in a bit more detail.

\subsubsection*{The homology of the framed little discs operad}

Recall that the non-unital little $n$-discs operad $\mathbb{E}_n$ carries an action of $\mathrm{SO}_n$ and, following \cite{salvatore2003framed}, the \emph{framed} little $n$-discs operad arises as the associated semi-direct product $\mathbb{E}_n^\mm{fr}=\mathbb{E}_n\rtimes \mathrm{SO}_n.$
At the homological level, $\ee_n\coloneqq H_*(\mathbb E_n)$ is an operad in the  category of modules over the cocommutative Hopf algebra $H_*(\mathrm{SO}_n)$.
 Likewise, one can express $$H_*(\mathbb{E}_n^\mm{fr}) \eqqcolon \ee_n^\mm{fr}= \ee_n \rtimes H_*(\mathrm{SO}_n)$$ 
as the semi-direct product of $\ee_n$ with $H_*(\mathrm{SO}_n)$. We will show that $\ee_n^\mm{fr}$ and (hence) $\mathbb{E}_n^\mm{fr}$ are both splendid.
 
More generally, let $H$ be a cocommutative Hopf algebra and let $\PP$ be a $k$-operad in the (symmetric monoidal) category of modules over $H$. 
The semi-direct product $\PP \rtimes H$ is a $k$-operad arising from a distributive law (see \cite[8.6.1]{LodayVallette2012} or Section \ref{sec:relativekoszul}) as follows.

The underlying symmetric sequence of $\PP \rtimes H$ is $\PP\circ H$, viewing $H$ as an operad in arity $1$. The action of $H$ on $\PP$  gives a map
$$
\Delta\colon H\circ \PP\rt \PP\circ H,
$$
sending $h\otimes \psi\in H(1)\otimes \PP(p)$ to $(h^{(1)}\cdot \psi)\otimes (h^{(2)}\otimes \dots\otimes h^{(1+p)})$, using the $p$-fold coproduct of $h$. One verifies property (I) from  \cite[8.6.1]{LodayVallette2012} using the coassociativity and cocommutativity of the coproduct in $H$, while property (II) follows from the compatibility of the product and the coproduct in $H$.
We then have

$$\PP \rtimes  H\coloneqq \PP{\circ_\Delta} H.$$

\begin{lemma}\label{lem:semi-direct}
	Let $\PP$ be a $1$-reduced $k$-operad equipped with an action of a cocommutative Hopf algebra $H$ in degrees $\leq 0$. If $\PP$ is splendid as a $k$-operad, then $\PP\rtimes H$ is a splendid $k$-operad.
\end{lemma}

\begin{proof}
	The operad $\PP \rtimes  H = \PP\circ_{\Delta} H$ is given by $H$ in artiy $1$, so we have to show that $H\circ^h_{\PP\circ_{\Delta} H} H$ is eventually highly connective. To see this, we will resolve $H$ as a right $\PP\circ_{\Delta} H$-module. Let us consider the symmetric sequence $\Bar\PP\circ_{\pi} \PP\circ H$, where $\Bar\PP\circ_{\pi} \PP$ is the twisted composition product associated to the universal twisting morphism $\pi\colon \Bar\PP\drt \PP$ \cite[Section 6.5.4]{LodayVallette2012} (or see Definition \ref{def:twisted comp product}). Explicitly, $\Bar\PP\circ_{\pi} \PP\circ H$ is spanned by trees with vertices labeled by $\ol{\PP}[1]$, onto which we graft to each leaf a $2$-level tree with root vertex labeled by $\PP$ and leaf vertices labeled by $H$. The differential is given by (a) applying $d_{\PP}$ or $d_H$ to vertices, (b) contracting edges between $\ol{\PP}[1]$-labeled vertices and composing the labels, and (c) for each $\ol{\PP}[1]$-labeled vertex furthest from the root, compose with all $\PP$-labeled vertices above it.
	
	Note that $\Bar\PP\circ_{\pi} \PP\circ H$ has a manifest right $\PP\circ_{\Delta} H$-module structure which is compatible with the differential (since parts (b) and (c) of the differential only involved composition in $\PP$, without any interference of $H$). The augmentation $\Bar\PP\circ_{\pi} \PP\rt k$ is a quasi-isomorphism \cite[Lemma 6.5.9]{LodayVallette2012} (or Lemma \ref{lem:twistedcomp}), so that the induced map $\Bar\PP\circ_{\pi} \PP\circ H\rt H$ is a quasi-isomorphism as well. This is readily seen to be a map of right $\PP\circ_{\Delta} H$-modules.
	
	One can endow $\Bar\PP\circ_{\pi} \PP\circ H$ with an increasing filtration by the number of $\ol{\PP}[1]$-labeled vertices, whose associated graded is the \emph{free} right $\PP\circ_{\Delta} H$-module on $\mm{gr}(\Bar\PP)$. Using this filtration as in Lemma \ref{lem:twistedcomp}, ones sees that the functor $(\Bar\PP\circ_{\pi} \PP\circ H)\circ_{\PP\circ_{\Delta} H}(-)$ preserves quasi-isomorphisms. This implies that we can compute $H\circ_{\PP\circ_{\Delta} H}^h H$ as the (strict) relative composition product
	$$
	H\circ_{\PP\circ_{\Delta} H}^h H\simeq (\Bar\PP\circ_{\pi} \PP\circ H)\circ_{\PP\circ_{\Delta} H} H \cong \Bar\PP\circ H.
	$$
	Since $H$ is connective and $\Bar(\PP)$ is eventually highly connective (since $\PP$ was splendid), it follows that $H\circ_{\PP\circ_{\Delta} H}^h H\simeq \Bar\PP\circ H$ is eventually highly connective as well. 
\end{proof}

\begin{corollary}
	The $k$-operad $\ee_n^\mm{fr}$ is splendid. Furthermore, $\ee_n^\mm{fr}$ is naturally a $H_*(\mm{SO}_n)$-operad which is splendid.
\end{corollary}
\begin{proof}
	The first statement follows directly from Observation \ref{obs:koszul case} and the previous lemma. 
	The second statement follows from Example \ref{ex:klawbar}.
\end{proof}

\subsubsection*{The BD-operad}
For each $n\geq 0$, there is a $k[\hbar]$-linear operad $\BD_n$ which agrees with the ($0$-reduced) $\mathbb{E}_n$-operad away from $\hbar=0$ and with the ($0$-reduced) shifted Poisson operad at $\hbar=0$ \cite{CPTVV2017}:
$$
\BD_n\otimes^h_{k[\hbar]}k[\hbar^\pm]\simeq \mathbb{E}_n[\hbar^\pm],\qquad\qquad\qquad \BD_n\otimes^h_{k[\hbar]}k[\hbar]/\hbar \simeq \Pois_n.
$$
For example, the $\BD_0$-operad from \cite{costello_gwilliam_vol1,costello_gwilliam_vol2} (see also \cite{beilinson2004chiral}) 
is the $k[\hbar]$-operad generated by a commutative product and a Lie bracket of degree $1$ satisfying the Leibniz rule, and 
equipped with the differential $d(-\cdot -)=\hbar [-, -]$. 

Similarly, the operad $\BD_1$ is obtained as the Rees construction of the associative operad, equipped with the PBW-filtration 
\cite{costello_gwilliam_vol2,CPTVV2017}; explicitly, a $\BD_1$-algebra is a $k[\hbar]$-module equipped with a (nonunital) 
associative product $\ast$ and a Lie bracket $[-, -]$ satisfying 
$$
[a, b\ast c]= [a, b]\ast c + b\ast [a, c] \qquad\qquad a\ast b-b\ast a=\hbar [a, b].
$$
\begin{proposition}\label{prop:BD}
The $k[\hbar]$-operads $\BD_0$ and $\BD_1$ are Koszul self-dual 
$$
\mf{D}(\BD_0)\simeq \BD_0 \qquad\qquad \mf{D}(\BD_1)\simeq \BD_1\{-1\}
$$
(relative to $k[\hbar]$) and satisfy the conditions of Theorem \ref{thm:mainthm}, so that there are equivalences
$$\begin{tikzcd}
\cat{FMP}_{\BD_n}\arrow[r, "\sim"] & \cat{Alg}_{\BD_n}; \hspace{4pt} X\arrow[r, mapsto] & T_X[-n], \quad \text{ for }n=0,1.
\end{tikzcd}$$
\end{proposition}
\begin{proof}
\textbf{Case $n=0$:} note that $\BD_0=\mm{Free}(E)/R$ is a binary quadratic operad on two generators $\mu=(-\cdot -)$ and $\lambda= [-, -]$, with differential $d\mu=\hbar\cdot \lambda$. Since the relations are the ones of the usual Poisson operad, its quadratic dual $\BD_0^!=\mm{Free}(E^\vee)/R^\perp$ is isomorphic to $\BD_0\{1\}$. To see that it is Koszul, it suffices to see that 
$$\begin{tikzcd}
\BD_0^{\antishriek}=\mm{coFree}\big(E[1], R[2]\big)\arrow[r] &  \Bar\big(\BD_0\big)
\end{tikzcd}$$
is a quasi-isomorphism of $k[\hbar]$-modules. To see this, it suffices to verify that the maps
$$\begin{tikzcd}[row sep=0pc]
\BD_0^{\antishriek}\otimes^h_{k[\hbar]} k[\hbar^{\pm}] \arrow[r] & \Bar(\BD_0)\otimes^h_{k[\hbar]} k[\hbar^{\pm}]\\ \BD_0^{\antishriek}\otimes^h_{k[\hbar]} k[\hbar]/\hbar \arrow[r] & \Bar(\BD_0)\otimes^h_{k[\hbar]} k[\hbar]/\hbar
\end{tikzcd}$$ 
are both quasi-isomorphisms (by derived Nakayama, the second condition implies that the localizations at $\hbar=0$ are quasi-isomorphic). Because extension of scalars is symmetric monoidal and all $\BD_0(p)$ and $\BD_0^{\antishriek}(p)$ are finite complexes of free $k[\hbar]$-modules (so that we do not have to derive the tensor product), the above two maps agree with the maps
$$\begin{tikzcd}[row sep=0pc]
\Big(\BD_0\otimes_{k[\hbar]} k[\hbar^{\pm}]\Big)^{\antishriek} \arrow[r] & \Bar\Big(\BD_0\otimes_{k[\hbar]} k[\hbar^{\pm}]\Big)\\ \Big(\BD_0\otimes_{k[\hbar]} k[\hbar]/\hbar\Big)^{\antishriek} \arrow[r] & \Bar\Big(\BD_0\otimes_{k[\hbar]} k[\hbar]/\hbar\Big)
\end{tikzcd}$$ 
The first map is a quasi-isomorphism between two cooperads which are both quasi-isomorphic to the trivial cooperad $k[\hbar^{\pm}]$, while the second map is a quasi-isomorphism because $\BD_0\otimes_{k[\hbar]} k[\hbar]/\hbar\cong P_0$ is a quadratic Koszul operad.\\

\textbf{Case $n=1$:} note that $\BD_1=\mm{Free}(E)/R$ is a quadratic operad on two binary generators $\mu=-\cdot -$, $\lambda=[-,-]$, on which $\Sigma_2$ acts trivially, resp.\ by the sign representation. The module of relations $R$ is generated by:
\begin{description}[leftmargin=*, labelindent={\labelwidth+1pt}, font=\normalfont]
\item[(J)] Jacobi relation $[a, [b, c]] + [b, [c, a]] + [c, [a, b]]$.
\item[(L)] Leibniz rule $[a, b\cdot c] -c\cdot [a, b]+ b\cdot [a, c]$.
\item[(A)] associativity for $a\ast b \coloneqq a\cdot b + \hbar [a, b]$, or explicitly:
$$
\Big(a\cdot (b\cdot c)- c\cdot (a\cdot b)\Big) + \hbar\Big(a\cdot [b, c] - c\cdot [a, b]+[a, b\cdot c]+[c, a\cdot b]\Big)+\hbar^2\Big([a, [b, c]]+ [c, [a, b]]\Big).
$$
\end{description}
Note that $\mm{Free}(E)(p)$ and $\BD_1(p)$ are finitely generated projective (equivalently, torsion free) $k[\hbar]$-modules for all $p$; for $\BD_1(p)$ this follows from the fact that it arises as the Rees construction of a vector space with an increasing filtration. It follows that $R$ is also finitely generated and projective. Note that $R$ has rank $6$, since its fiber at $\hbar=0$ is the vector space of relations for the Poisson operad $P_1$, which has dimension $6$.

Now consider the inner product on $E$ of signature $(1, 1)$, determined by $\big<\mu, \lambda\big>=1$. This induces an inner product on $\mm{Free}(E)(3)$ of signature $(6, 6)$, and an explicit computation shows that $R\subseteq \mm{Free}(E)(3)$ is isotropic, hence Lagrangian. For example, one has
\begin{align*}
\big<(A); (J)\big> &= \big<a\cdot(b\cdot c) ; [a, [b, c]]\big>-\big<c\cdot(a\cdot b); [c\cdot [a, b]]\big> = 1 -1 =0\\
\big<(A); (L)\big> &= \big<\hbar\, a\cdot[b, c] ; [a, b\cdot c]\big>-\big<\hbar\, [c, a\cdot b]; c\cdot [a, b]\big> = \hbar -\hbar =0
\end{align*}
and $\big<(A); (A)\big>$ is given by $2\hbar^2$ times
$$
\big<a\cdot(b\cdot c) ; [a, [b, c]]\big>-\big<c\cdot(a, b); [c, [a, b]]\big>+ \big<a\cdot [b, c]; [a\cdot (b\cdot c)]\big> -\big<c\cdot [a, b]; [c, (a\cdot b)]\big> = 0.
$$
Now consider the quadratic dual $\BD_1^!=\mm{Free}(E^\vee)/R^\perp$. Identifying $\mu^\vee\leftrightarrow \lambda$ and $\lambda^\vee\leftrightarrow \mu$ using the inner product described above and using that the inner product identifies the Lagrangian $R$ with $R^\perp$, we obtain an isomorphism $\BD_1^!\cong \BD_1$.

It remains to verify that $\BD_1$ is Koszul. This follows as in the case of $\BD_0$: we have to show that the map $\BD_1^\antishriek\rt \Bar(\BD_1)$ is a quasi-isomorphism, which can be checked at $\hbar=0$ and after inverting $\hbar$. Since each $\BD_1(p)$ is a finitely generated projective $k[\hbar]$-module and extension of scalars is symmetric monoidal, one then reduces to checking that $\BD_1\otimes_{k[\hbar]}k[\hbar^{\pm}]$ and $\BD_1\otimes_{k[\hbar]} k$ are Koszul operads. But these are just the associative and $P_1$-operads.
\end{proof}
\begin{remark}
Proposition \ref{prop:BD} also applies to the operads $\BD_n$ with $n\geq 2$,  which are defined as the Rees construction of the $\mathbb{E}_n$-operads, endowed with their Postnikov filtration. Indeed, by the (rational) formality of the $\mathbb{E}_n$-operad \cite{willwacher2018little}, this filtration splits and one can identify $\BD_n\simeq {e}_n[\hbar]$. Probably one can also deduce Proposition \ref{prop:BD} directly from the self-duality of the $\mathbb{E}_n$-operad \cite{Fresse2011}.
\end{remark}

\subsubsection*{$G$-equivariant algebras}
Suppose that $\base$ is a connective symmetric monoidal dg-category. Recall that this symmetric monoidal structure can be encoded by a non-augmented $\base$-operad $\base^{\otimes}$, defined by $\base^{\otimes}(c_1, \dots, c_n; c_0)=\base(c_1\otimes\dots\otimes c_n; c_0)$. Note that each $\base^{\otimes}(c_1, \dots, c_n; -)$ is a free left $\base$-module (on $c_1\otimes\dots\otimes c_n$). Consequently, as a symmetric $\base$-bimodule $\base^{\otimes}$ is isomorphic to its $\base$-linear dual $(\base^{\otimes})^\vee$ and comes with a cocomposition $\base^{\otimes}\rt \base^{\otimes}\circ_\base \base^{\otimes}$ dual to the composition of $\base^{\otimes}$.
	
	Now let $\PP$ be a $k$-operad and consider the $\base$-operad $\PP\otimes\base^{\otimes}$ given by the exterior Hadamard tensor product \eqref{eq:exterior hadamard}.  Unraveling the definitions, one sees that a $\PP\otimes\base^{\otimes}$-algebra is a $\PP$-algebra in the symmetric monoidal category $\cat{LMod}_{\base}^{\dg}$ of $\base$-modules as in Section \ref{sec:conventions} (with $\otimes$ given by Day convolution). The fact that $\base^{\otimes}(c_1, \dots, c_n; -)$ is a free left $\base$-module implies that there is an isomorphism
	$$\begin{tikzcd}
		\Bar_\base(\PP\otimes\base^{\otimes})\arrow[r] & (\Bar\PP)\otimes\base^{\otimes}
	\end{tikzcd}$$ 
	between the bar construction of $\PP\otimes\base^{\otimes}$ relative to $\base$ and the exterior Hadamard tensor product of $\base^{\otimes}$ with the bar construction of $\PP$ over $k$. This is a map of $\base$-cooperads if one gives $\Bar\PP\otimes\base^{\otimes}$ the cooperad structure coming from the one on $\base^{\otimes}\cong (\base^{\otimes})^\vee$ and $\Bar\PP$.
	
	In particular, if $\PP$ is a finite type binary Koszul operad, then $\PP\otimes\base^{\otimes}$ is splendid and its dual operad (relative to $\base$) is $\PP^!\{-1\}\otimes\base^{\otimes}$.
\begin{example}\label{ex:G-reps}
	Suppose that $G$ is a reductive algebraic group over $k$ and let $\cat{Rep}_G=\cat{QCoh}(\mm{B}G)$ denote the symmetric monoidal $\infty$-category of $G$-representations. It follows from \cite[Corollary 3.22]{ben-zvi2010centers} that $\cat{Rep}_G$ is compactly generated by the finite dimensional $G$-representations (concentrated in degree $0$). Let $\base=\cat{Rep}^{\mm{fd}}_G$ be the symmetric monoidal dg-category of these representations and note that $\base$ is simply a category enriched over vector spaces (in degree $0$): there are no higher Ext-groups since $G$ is reductive. The symmetric monoidal model category $\cat{LMod}^{\dg}_{\base}$ then presents the symmetric monoidal $\infty$-category $\cat{Rep}_G$.
	
	Let us now apply the previous discussion to the operad $\PP=\Lie$. Then Theorem \ref{thm:mainthm} provides an equivalence between the $\infty$-category of formal moduli problems $\Art_{\Lie\otimes\base^{\otimes}}\rt \sS$ indexed by Artin Lie algebras carrying a $G$-representation, and that of nonunital commutative algebras in $\cat{Rep}_G$. This correspondence has been considered extensively in \cite{pridham2007fundamentalgroup, pridham2008proalgebraic} in the study of pro-algebraic homotopy types.
\end{example}

\subsubsection*{Operads with only nullary operations}
The following baby-example might also be useful to illustrate what happens for operads with nullary operations, when the category $\base$ has nontrivial (endo)morphisms. 
Let $\base$ be a connective dg-algebra and let $V$ be a connective left $\base$-module. 
There is a $\base$-operad $\mathscr{P}$ whose algebras are left $\base$-modules $W$ with 
a $\base$-linear map $V\rt W$: $\mathscr{P}(0)=V$, $\mathscr{P}(1)=\base$, and 
$\mathscr{P}(n)=0$ for every $n\geq2$. The dual operad (relative to $\base$) is the 
$\base^{\op}$-operad $\mf{D}_{\base}(\mathscr{P})$ whose algebras are left $\base^{\op}$-modules 
$W$ endowed with an $\base^{\op}$-linear map $V^\vee[-1]=\Hom_{\base}(V[1], \base)\rt W$.

In this case, the Koszul duality functor 
$\mf{D}\colon \cat{Alg}_{\mathscr{P}} \rt \cat{Alg}_{\mf{D}_{\base}(\mathscr{P})}^{\op}$ 
can be identified with the functor
\begin{equation}\label{diag:koszuldualitybaby}
\begin{tikzcd}[row sep=0pc]
V\big/\cat{LMod}_{\base} \arrow[r] & \Big(V^\vee[-1]\big/\cat{RMod}_{\base}\Big)^{\op} \\
\big(V\rt W\big)\arrow[r, mapsto] & \mm{fib}(W^\vee\rt V^\vee).
\end{tikzcd}
\end{equation}
The category of \emph{Artin} $\mathscr{P}$-algebras can be identified with the category 
of $V\rt W$, where $W$ is a finitely presented $\base$-module, with generators in nonpositive degrees. 
Using this, one sees that
$$
\cat{FMP}_{\mathscr{P}}\simeq \mm{Ind}\big(\cat{C}_1^\op\big)\qquad\qquad \cat{C}_1=\big\{V\rt W: \text{ perfect }W\big\}\subseteq V\big/\cat{LMod}_{\base}.
$$ 
Similarly, the category of right $\base$-modules under $V^\vee$ is compactly generated, so that there is an equivalence
$$
\cat{Alg}_{\mf{D}_{\base}(\mathscr{P})}\simeq \mm{Ind}\big(\cat{C}_2\big)\qquad\qquad \cat{C}_2=\big\{V^\vee\rt W: \text{ perfect cofiber}\big\}\subseteq V^\vee\big/\cat{RMod}_{\base}.
$$
Theorem \ref{thm:mainthm} then reduces to the assertion that the functor \eqref{diag:koszuldualitybaby} establishes a contravariant equivalence between $\cat{C}_1$ and $\cat{C}_2$.

\section{From FMPs to algebras}\label{sec:operadicfmps}
In this section we introduce the main ingredients that will be used to relate formal moduli problems of algebras over an operad $\mathscr{P}$ to algebras over its dual operad $\mf{D}(\PP)$ as in Definition \ref{def:dualoperad}. 
In particular, we describe an adjoint pair of $\infty$-categories
$$\begin{tikzcd}
\mf{D}\colon \cat{Alg}_{\PP}\arrow[r, yshift=0.8ex] & \cat{Alg}_{\mf{D}(\PP)}^\op\colon \mf{D}'\arrow[l, yshift=-0.8ex]
\end{tikzcd}$$
sending an algebra to its (bar) dual algebra (see Section \ref{sec:weakkoszul}). This adjunction is an example of a \emph{weak Koszul duality context} in the sense of \cite{calaque2018formal} and will be the main actor in the proof of our main theorem (Theorem \ref{thm:mainthm}). Indeed, the axiomatic framework developed in \cite{DAGX,calaque2018formal} provides explicit conditions under which this adjoint pair induces an equivalence between $\mf{D}(\PP)$-algebras and formal moduli problems over $\PP$. We will recall these conditions in Section \ref{sec:outline} (see Theorem \ref{thm:axiomatic}).

We follow Assumption \ref{ass:cofibrancy}: all $\base$-(co)operads are assumed to be (co)augmented and (filtered) cofibrant as left $\base$-modules.

\subsection{Duality for algebras over operads}\label{sec:weakkoszul}
Let $\base$ be a dg-category and let $\phi\colon \mathscr{C}\drt \mathscr{P}$ be a twisting morphism from a $\base$-cooperad to a $\base$-operad (see Construction \ref{constr:k-twisting morphisms}). Recall our convention that $\CC$ (resp.\ $\PP$) is always assumed to be filtered-cofibrant (resp.\ cofibrant) as a left $\base$-module, see Assumption \ref{ass:cofibrancy}. 

Recall that (Proposition \ref{prop:cobardjointtobar}) the twisting morphism $\phi$ gives rise to an adjoint pair
$$\begin{tikzcd}
\Omega_\phi\colon \cat{CoAlg}_{\mathscr{C}}^{\dg} \arrow[r, yshift=0.8ex] & \cat{Alg}_{\mathscr{P}}^{\dg}\colon \Bar_\phi.\arrow[l, yshift=-0.8ex]
\end{tikzcd}$$
Taking the linear dual of the bar construction, we obtain a functor
$$\begin{tikzcd}[column sep=3pc]
\cat{Alg}_{\mathscr{P}}^{\dg}\arrow[r, "\Bar_\phi"] & \cat{CoAlg}_{\mathscr{C}}^{\dg}\arrow[r, "(-)^\vee"] & \cat{Alg}_{\mathscr{C}^\vee}^{\dg, \op}
\end{tikzcd}$$
with values in algebras over the dual $\base^{\op}$-operad $\mathscr{C}^\vee$ (cf.\ Proposition \ref{prop:dual of cooperad is operad}). By Lemma \ref{lem:algbarpreswe}, this functor preserves quasi-isomorphisms between algebras which are cofibrant as left $\base$-modules. Consequently, it induces a functor of $\infty$-categories
$$\begin{tikzcd}
\mf{D}_\phi\colon \cat{Alg}_{\mathscr{P}}\arrow[r] & \cat{Alg}_{\mathscr{C}^\vee}^{\op}.
\end{tikzcd}$$

If $\phi$ is weakly Koszul (Definition \ref{def:koszultwisting}) we can identify the $\infty$-category of algebras over $\mathscr{C}^\vee$ with algebras over the dual operad $\mf{D}(\mathscr{P}) =  \Bar(\mathscr P)^\vee$.
\begin{lemma}\label{lem:adjoint}
Suppose that $\phi\colon\CC\drt \PP$ is weakly Koszul. Then the following assertions hold:
\begin{enumerate}
\item\label{it:Disderiv} For any $\PP$-algebra $A$, there is a natural equivalence of $\base^\op$-modules 
$$
\mf{D}_\phi(A)\simeq \mathbb{R}\Der_{\PP}(A, \base)
$$
to the (derived) $\base^\op$-module of $\PP$-algebra derivations of $A$ with coefficients in the trivial $A$-module $\base$.
\item The functor $\mf{D}_\phi$ preserves all colimits, so it is the left adjoint in an adjoint pair
\begin{equation}\label{diag:basicadjunction}
\begin{tikzcd}
\mf{D}_\phi\colon \cat{Alg}_{\PP}\arrow[r, yshift=0.8ex] & \cat{Alg}_{\CC^\vee}^\op\colon \mf{D}_\phi'.\arrow[l, yshift=-0.8ex]
\end{tikzcd}\end{equation}
\end{enumerate}
\end{lemma}\label{lem:weak koszul duality context}
In the terminology of \cite{calaque2018formal}, the adjoint pair \eqref{diag:basicadjunction} is an example of a \emph{weak Koszul duality context} (this is essentially the assertion of Corollary \ref{cor:adjointisderivations}).
\begin{proof}
The first assertion implies the second: indeed, the functor $\mf{D}_\phi$ preserves colimits (and hence admits a right adjoint by the adjoint functor theorem \cite[Corollary 5.5.2.9]{HTT}) if and only if the composite
\begin{equation}\label{diag:dualandforget}\begin{tikzcd}[column sep=2.8pc]
\cat{Alg}_{\mathscr{P}}^{\dg}\arrow[r, "\Bar_\phi"] & \cat{CoAlg}_{\mathscr{C}}^{\dg}\arrow[r, "(-)^\vee"] & \cat{Alg}_{\cat{C}^\vee}^{\dg, \op}\arrow[r, "\mm{forget}"] & \cat{Mod}^{\dg, \op}_{\base^{\op}}
\end{tikzcd}\end{equation}
preserves homotopy colimits. The functor $\mathbb{R}\Der(-, \base)$ taking derived modules of derivations clearly has this property.

Since $\cat{C}\rt \Bar\PP$ is a quasi-isomorphism between cofibrant left $\base$-modules, the functor \eqref{diag:dualandforget} is naturally equivalent to the functor associated to the universal twisting morphism $\phi^\mm{uni}\colon \Bar\PP\rt \mathscr{P}$. It will therefore suffice to prove assertion \ref{it:Disderiv} for $\phi=\phi^\mm{uni}$. In this case, consider the Quillen pair
$$\begin{tikzcd}
I\colon \cat{Alg}_{\mathscr{P}}^{\dg}\arrow[r, yshift=0.8ex] & \cat{Mod}_{\base}^{\dg}\colon \mm{triv}\arrow[l, yshift=-0.8ex]
\end{tikzcd}$$ 
where the right adjoint takes the trivial $\PP$-algebra (using the augmentation $\mathscr{P}\rt \base$) and $I$ sends a $\PP$-algebra to its module of indecomposables. Unraveling the definitions, one sees that there is an isomorphism of $\base$-modules
$$
\Bar_{\phi}(A)^\vee\cong I\big(\Omega_\phi \Bar_{\phi}(A)\big)^\vee\cong\Der_{\PP}\big(\Omega_\phi \Bar_{\phi}(A), \base\big)
$$
where we have used that $I(B)^\vee=\Der_{\PP}(B, \base)$. By Lemma \ref{lem:algbarresol}, the map $\Omega_\phi \Bar_\phi(A)\rt A$ is a quasi-isomorphism whenever $A$ is cofibrant as a $\base$-module, i.e.\ it provides a functorial cofibrant replacement of $A$. It follows that $\mathfrak D_\phi$ computes indeed the derived functor of derivations with coefficients in the trivial $A$-module $\base$.
\end{proof}
Let us note that the adjoint pair \eqref{diag:basicadjunction} depends naturally on $\phi$, in the following sense (we will come back to this in Section \ref{sec:naturality}):
\begin{lemma}\label{lem:naturality}
Consider a commuting square
\begin{equation}\label{diag:mapoftwists}\begin{tikzcd}
\CC\arrow[d, "g"{swap}]\arrow[r, dashrightarrow, "\phi"] & \PP\arrow[d, "f"]\\
\DD\arrow[r, dashrightarrow, "\psi"{swap}] & \QQ.
\end{tikzcd}\end{equation}
where $g$ is a map of $\base$-cooperads, $f$ is a map of $\base$-operads and $\phi$ and $\psi$ are weakly Koszul twisting morphisms. Then there is a natural transformation of $\DD^\vee$-algebras 
$$\begin{tikzcd}
\mu\colon \Bar_{\psi}(f_!A)^\vee\arrow[r] & g^*\Bar_\phi(A)^\vee.
\end{tikzcd}$$
When $A$ is a cofibrant $\PP$-algebra, this map is a weak equivalence.
\end{lemma}
In other words, a diagram like \eqref{diag:mapoftwists} induces a square of $\infty$-categories
\begin{equation}\label{diag:hocomm}\begin{tikzcd}[column sep=3.2pc]
\cat{Alg}_{\PP}\arrow[d, "f_!"{swap}]\arrow[r, "\mf{D}_\phi"] & \cat{Alg}_{\CC^\vee}^{\op}\arrow[d, "g^*"]\arrow[ld, Rightarrow, "\sim", shorten=2ex]\\
\cat{Alg}_{\QQ}\arrow[r, "\mf{D}_{\psi}"{swap}] & \cat{Alg}_{\DD^\vee}^{\op}
\end{tikzcd}\end{equation}
commuting up to a natural equivalence $\mu$. In particular, $\mf{D}_\phi$ is a homotopy invariant of the map $\phi$, in the following sense: if $f$ (and, by Lemma \ref{lem:operadicbarpreservesqisos}, also $g$) is a quasi-isomorphism, then the vertical functors in \eqref{diag:hocomm} are equivalences by Corollary \ref{cor:quasiiso gives equivalence} which intertwine $\mf{D}_\phi$ and $\mf{D}_{\psi}$.
\begin{proof}
We define $\mu$ to be the dual of a natural map of $\DD$-coalgebras
$$\begin{tikzcd}
g^*\Bar_\phi(A)\arrow[r] & \Bar_{\psi}(f_!A).
\end{tikzcd}$$
Without differentials, this map is given by the map $\CC(A)\rt \DD(f_!A)$, defined on cogenerators by $\CC(A)\rt A\rt f_!A$. This map of $\DD$-coalgebras indeed preserves the bar differential. To see that it is a weak equivalence when $A$ is cofibrant, we can work at the level of the underlying $\base$-modules. In that case, we have a weak equivalence
$$\begin{tikzcd}
\Der(A, \base)\arrow[r, "\sim"]  & \Der\big(\Omega_\phi \Bar_\phi(A), \base\big)\cong \Bar_\phi(A)^\vee
\end{tikzcd}$$
from the complex of $\PP$-algebra derivations of $A$ (see Lemma \ref{lem:adjoint}). We obtain a commuting square of chain complexes
$$\begin{tikzcd}
\Der_{\QQ}(f_!(A), \base)\arrow[d, "\sim"{swap}] \arrow[r] & \Der_{\PP}(A, \base)\arrow[d, "\sim"]\\
\Bar_{\psi}(f_!A)^\vee \arrow[r, "\mu"{swap}] & \Bar_\phi(A)^\vee
\end{tikzcd}$$
The top horizontal map is an isomorphism, so the result follows.
\end{proof}
\begin{corollary}\label{cor:adjointisderivations}
For any weakly Koszul twisting morphism $\phi\colon \CC\drt\PP$ and any $\base$-module $V$, there is a natural equivalence of $\CC^\vee$-algebras
$$\begin{tikzcd}
\mf{D}_\phi\big(\PP(V)\big)\arrow[r] & \mm{triv}\big(V^\vee\big).
\end{tikzcd}$$
Consequently, for any algebra $\mf{g}$ over the $\base^{\op}$-operad $\mf{D}(\PP)=(\Bar\PP)^\vee$, the \emph{underlying $\base$-module} of $\mf{D}'_\phi(\mf{g})$ is given by the derived functor of derivations
$$
\mf{D}'_\phi(\mf{g})\simeq \mathbb{R}\mm{Der}_{\mf{D}(\PP)}(\mf{g}, \base^{\op}).
$$
\end{corollary}
\begin{proof}
The first assertion is a special case of Lemma \ref{lem:naturality} and the second assertion follows by passing to right adjoints.
\end{proof}
\begin{definition}\label{def:Pbar}
Let $\mathscr{P}$ be a $\base$-operad and let $\mf{D}(\PP)= (\Bar\PP)^\vee$ be its dual operad. We will denote by
\begin{equation}\label{diag:coreadjunction}\begin{tikzcd}
\mf{D}\colon \cat{Alg}_{\mathscr{P}}\arrow[r, yshift=0.8ex] & \cat{Alg}_{\mf{D}(\mathscr{P})}^{\op}\arrow[l, yshift=-0.8ex]\colon \mf{D}'
\end{tikzcd}\end{equation}
the adjunction associated to the universal twisting morphism $\pi\colon \Bar\PP\drt \PP$. By the discussion after Lemma \ref{lem:naturality}, we are allowed to model $\mf{D}(\PP)$ and $\PP$ using any quasi-isomorphic twisting morphism $\phi\colon \CC\drt \PP'$.
\end{definition}
We conclude by giving an explicit description of the right adjoint $\mf{D}'$. 
\begin{theorem}\label{thm:DprimeisD}
For any $\base$-operad $\PP$, there exists a natural map 
$$\begin{tikzcd}
\eta\colon \PP\arrow[r] & \mf{D}(\mf{D}(\PP))
\end{tikzcd}$$
in the $\infty$-category of $\base$-operads, and the right adjoint functor $\mf{D}_\phi'\colon \Alg^{\op}_{\mf{D}(\PP)}\rt \Alg_{\PP}$ is naturally equivalent to the functor
$$\begin{tikzcd}
\Alg^{\op}_{\mf{D}(\PP)}\arrow[r, "\mf{D}"] & \Alg_{\mf{D}(\mf{D}(\PP))}\arrow[r, "\eta^*"] & \Alg_{\PP}.
\end{tikzcd}$$
\end{theorem}
To define the map $\eta$, which will arise from a zig-zag of maps at the chain level, let us make the following observation:
\begin{construction}\label{con:twocooperads}
Let $\CC$ be a $\base$-cooperad and $\DD$ a $\base^{\op}$-cooperad. Then there is an isomorphism of convolution Lie algebras (cf.\ Remark \ref{rem:convolutionoperad})$$
\Hom_{\cat{BiMod}_{\base^{\op}}^\Sigma}\big(\DD, \CC^\vee\big) \cong \Hom_{\cat{BiMod}_{\base}^\Sigma}\big(\CC, \DD^\vee\big)
$$
sending a linear map $\psi\colon \DD\rt \CC^\vee$ to its adjoint $\psi^\top\colon \CC\rt \DD^\vee$. In particular, this restricts to a bijection between twisting morphisms. For a twisting morphism $\psi$ and a $\DD$-coalgebra $X$, there is a natural map of $\CC^\vee$-algebras
\begin{equation}\label{diag:cobartobarbidual}\begin{tikzcd}
\Omega_{\psi}(X)\arrow[r] & \big(\Bar_{\psi^\top}(X^\vee)\big)^\vee
\end{tikzcd}\end{equation}
given on generators by the obvious inclusion $X\rt X^{\vee\vee}\rt \big(\DD\circ X^\vee\big)^\vee$.
\end{construction}
Let us now fix a $\base$-cooperad $\CC$ which is filtered-cofibrant as a left $\base$-module and let $\epsilon\colon \QQ\rto{\sim} \CC^\vee$ denote a replacement of $\CC^\vee$ by a $\base^{\op}$-operad which is cofibrant as a left $\base^\op$-module. Consider the canonical twisting morphisms
$$\begin{tikzcd}
\phi\colon\CC\arrow[r, dashrightarrow] & \Omega\CC & & \phi^\dagger\colon \Bar\QQ\arrow[r, dashrightarrow] & \QQ.
\end{tikzcd}$$
Applying Construction \ref{con:twocooperads} to the case where $\DD=\Bar\QQ$, the twisting morphism $\epsilon\circ \phi^\dagger\colon\Bar\QQ\drt \CC^\vee$ has an adjoint twisting morphism $(\epsilon\circ \phi^\dagger)^\top\colon \CC\drt \Bar(\QQ)^\vee$. We can write $(\epsilon\circ \phi^\dagger)^\top=\eta\circ \phi$, where
$$\begin{tikzcd}
\eta\colon \Omega\CC\arrow[r] & \Bar(\QQ)^\vee
\end{tikzcd}$$
is the corresponding map of $\base$-operads out of the cobar construction. In this setting, we have the following identification of the right adjoint $\mf{D}'_\phi$:
\begin{proposition}\label{prop:DprimeisD}
In the above situation, the right adjoint $\mf{D}_\phi'\colon \Alg^{\op}_{\CC^\vee}\rt \Alg_{\Omega\CC}$ is naturally equivalent to the functor
$$\begin{tikzcd}
\Alg^{\op}_{\CC^\vee}\arrow[r, "\epsilon^*", "\sim"{swap}] & \Alg^{\op}_{\QQ} \arrow[r, "\mf{D}_{\phi^\dagger}"] & \Alg_{\Bar(\QQ)^\vee}\arrow[r, "\eta^*"] & \Alg_{\Omega\CC}.
\end{tikzcd}$$
\end{proposition}
\begin{remark}
When $\CC$ is finite dimensional  the map $\eta\colon \Omega(\CC)\rt \Bar(\CC^\vee)^\vee$ is an isomorphism and we can simply identify $\mf{D}'_\phi$ with $\mf{D}_{\phi^\dagger}$. 
\end{remark}
\begin{proof}
Our first goal will be to define a natural map of $\Omega\CC$-algebras
\begin{equation}\label{diag:compD}\begin{tikzcd}
\eta^*\mf{D}_{\phi^\dagger}(\epsilon^*\mf{g})\arrow[r]& \mf{D}_\phi'(\mf{g})
\end{tikzcd}\end{equation}
for every $\CC^\vee$-algebra $\mf{g}$. By adjunction, it suffices to provide a natural map
\begin{equation}\label{diag:adjcompD}\begin{tikzcd}
\mf{g} \arrow[r] & \mf{D}_\phi\big(\eta^*\mf{D}_{\phi^\dagger}(\epsilon^*\mf{g})\big)
\end{tikzcd}\end{equation}
in the $\infty$-category of $\CC^\vee$-algebras. To do this, note that $\Bar_\phi(\eta^*-)\cong \Bar_{\eta\phi}(-)$ and $\Bar_{\phi^\dagger}(\epsilon^*-)=\Bar_{\epsilon\phi^\dagger}(-)$ both preserve objects that are cofibrant as left modules. Consequently, we can compute
$$
\mf{D}_\phi\Big(\eta^*\mf{D}_{\phi^\dagger}(\mf{g})\Big) \cong \Big(\Bar_{\eta\phi}\big(\Bar_{\epsilon\phi^{\dagger}}\mf{g}\big)^\vee\Big)^\vee
$$
whenever $\mf{g}$ is cofibrant as a left $\base$-module. Now apply Construction \ref{con:twocooperads} to the case where $\DD=\Bar(\QQ)$ and to the twisting morphisms
$$\begin{tikzcd}
\psi=\epsilon\phi^\dagger\colon \DD=\Bar(\QQ)\arrow[r, dashrightarrow] & \CC^\vee & & \psi^\top=\eta\phi\colon \CC\arrow[r, dashrightarrow] & \DD^\vee.
\end{tikzcd}$$
For the $\DD$-coalgebra $X=\Bar_{\epsilon\phi^{\dagger}}\mf{g}$, the map \eqref{diag:cobartobarbidual} then gives a natural map of $\CC^\vee$-algebras
\begin{equation}\label{diag:adjcompchain}\begin{tikzcd}
\Omega_{\epsilon\phi^\dagger}\big(\Bar_{\epsilon\phi^\dagger}(\mf{g})\big)\arrow[r] & \Big(\Bar_{\eta\phi}\big(\Bar_{\epsilon\phi^{\dagger}}\mf{g}\big)^\vee\Big)^\vee.
\end{tikzcd}\end{equation}
The domain is the usual bar-cobar construction of $\mf{g}$, which comes with a natural quasi-isomorphism $\Omega_{\epsilon\phi^\dagger}\big(\Bar_{\epsilon\phi^\dagger}(\mf{g})\big)\rt \mf{g}$ when $\mf{g}$ is cofibrant as a $\base^{\op}$-module (Lemma \ref{lem:algbarresol}). At the level of $\infty$-categories, we therefore obtain the desired map \eqref{diag:adjcompD} and the adjoint comparison map \eqref{diag:compD}.

We now have to check that the comparison map \eqref{diag:compD} is an equivalence, for which it suffices to see that the underlying map of $\base$-modules is an equivalence. 
Note that Lemma \ref{lem:adjoint} and Corollary \ref{cor:adjointisderivations} produce natural equivalences of $\base$-modules
$$
\mf{D}_{\epsilon\phi^\dagger}(\mf{g}) \simeq \mathbb{R}\Der_{\CC^\vee}(\mf{g}, \base^{\op})\qquad\quad\qquad \mf{D}_\phi'(\mf{g})\simeq \mathbb{R}\Der_{\CC^\vee}(\mf{g}, \base^{\op}).
$$
Under these equivalences, the comparison map \eqref{diag:compD} corresponds to a natural endomorphism of 
$$
\mathbb{R}\Der_{\CC^\vee}\big(\mf{g}, \base^{\op}\big)\simeq \Der_{\CC^\vee}\big(\Omega_{\epsilon\phi^\dagger}\Bar_{\epsilon\phi^\dagger}(\mf{g}), \base^{\op}\big)\cong \Bar_{\epsilon\phi^\dagger}(\mf{g})^\vee.
$$
Unravelling the definitions, this endomorphism can be described as follows: an element $\alpha\in \Bar_{\epsilon\phi^\dagger}(\mf{g})^\vee$ is sent to the $\base^{\op}$-linear map
$$\begin{tikzcd}[column sep=1.9pc]
f_\alpha\colon \Bar_{\epsilon\phi^\dagger}\mf{g}\arrow[r] & \Omega_{\epsilon\phi^\dagger}\Bar_{\epsilon\phi^\dagger}\mf{g}\arrow[r, "\eqref{diag:adjcompchain}"] & \Big(\Bar_{\eta\phi}\big(\Bar_{\epsilon\phi^{\dagger}}\mf{g}\big)^\vee\Big)^\vee\arrow[r] & \Bar_{\epsilon\phi^{\dagger}}(\mf{g})^{\vee\vee}\arrow[r, "\mm{ev}_\alpha"] & \base^\op.
\end{tikzcd}$$
Here the first map is the inclusion of the generators and the third map is the projection, dual to the inclusion of $\Bar_{\epsilon\phi^{\dagger}}(\mf{g})^\vee$ into its bar construction (as the primitive elements). The last map evaluates an element of the bidual at $\alpha$. One easily sees that the assignment $\alpha\longmapsto f_\alpha$ is an isomorphism, so that \eqref{diag:compD} is indeed an equivalence.
\end{proof}
\begin{proof}[Proof (of Theorem \ref{thm:DprimeisD})]
Suppose that $\PP$ is cofibrant as a left $\base$-module and consider the situation of Proposition \ref{prop:DprimeisD} in the case where $\CC=\Bar\PP$. Then
$$
\CC^\vee\simeq \QQ\simeq \mf{D}(\PP) \qquad \text{and} \qquad \Bar(\QQ)^\vee\simeq \mf{D}(\mf{D}(\PP))
$$
so that the natural zig-zag $\PP\lto{\sim} \Omega\Bar\PP\rto{\eta} \Bar(\QQ)^\vee$ defines a natural map in the $\infty$-category of $\base$-operads $\eta\colon \PP\rt \mf{D}(\mf{D}(\PP))$.
Theorem \ref{thm:DprimeisD} then follows from Proposition \ref{prop:DprimeisD}.
\end{proof}

\subsection{Axiomatic argument}\label{sec:outline}
We will now describe the strategy of the proof of our main result, Theorem \ref{thm:mainthm}. Our strategy follows the axiomatic frameworks developed in \cite{DAGX,calaque2018formal}. More precisely, let us consider the adjunction
$$\begin{tikzcd}
\mf{D}\colon \cat{Alg}_{\mathscr{P}}\arrow[r, yshift=0.8ex] & \cat{Alg}_{\mf{D}(\mathscr{P})}^{\op}\arrow[l, yshift=-0.8ex]\colon \mf{D}'.
\end{tikzcd}$$
This adjunction is essentially never an equivalence, because it involves taking duals: both $\mf{D}$ and $\mf{D}'$ send an algebra to its module of derivations with coefficients in $\base$ (Lemma \ref{lem:adjoint} and Corollary \ref{cor:adjointisderivations}). Instead, one can try to refine the above adjunction to an equivalence between $\mf{D}(\PP)$-algebras and formal moduli problems over $\PP$, using the following construction: every $\mf{D}(\PP)$-algebra $\mf{g}$ defines a functor
$$\begin{tikzcd}
\Art_{\PP}\arrow[r] & \sS; \hspace{4pt} A\arrow[r, mapsto] & \Map_{\mf{D}(\PP)}\big(\mf{D}(A), \mf{g}\big).
\end{tikzcd}$$
Under suitable conditions on the functor $\mf{D}$, this functor will satisfy the axioms of a formal moduli problem (Definition \ref{def:formalmoduli}). Furthermore, the results leading to \cite[Theorem 1.3.12]{DAGX} provide general conditions on $\mf{D}$ under which this construction becomes an equivalence. In the current situation, we can summarize these results as follows:
\begin{theorem}\label{thm:axiomatic}
Let $\PP$ be a $\base$-operad. Then there is an equivalence of $\infty$-categories
$$\begin{tikzcd}
\MC\colon \cat{Alg}_{\mf{D}(\mathscr{P})}\arrow[r, "\sim"] & \cat{FMP}_{\mathscr{P}}; \hspace{4pt} \mf{g}\arrow[r, mapsto] & \Map_{\mf{D}(\mathscr{P})}(\mf{D}(-), \mf{g}).
\end{tikzcd}$$
if the following conditions are satisfied:
\begin{enumerate}[label=(\Alph*)]
\item\label{it:axiomA} For every Artin $\PP$-algebra $A$, the unit map $A\rt \mf{D}'\mf{D}(A)$ is an equivalence.
\item\label{it:axiomB} For every trivial algebra $\triv{c}{n}$ generated by a single element of degree $n\geq 0$, the $\mf{D}(\PP)$-algebra $\mf{D}\big(\triv{c}{n}\big)$ is freely generated by $\triv{c}{n}^\vee$.
\item\label{it:axiomC} The functor $\mf{D}$ sends every pullback square of Artin algebras
$$\begin{tikzcd}
A'\arrow[r]\arrow[d] & 0\arrow[d]\\
A\arrow[r] & \triv{c}{n}
\end{tikzcd}$$
with $n\geq 1$ to a pushout square of $\mf{D}(\PP)$-algebras.
\end{enumerate}
In this case, the inverse of the functor $\MC$ sends a formal moduli problem $F$ to its tangent complex $T(F)$, endowed with some $\mf{D}(\PP)$-algebra structure.
\end{theorem}
\begin{remark}
The notation $\MC$ is supposed to be suggestive: when $\PP$ satisfies suitable finite-dimensionality conditions, the formal moduli problem $\MC_\mf{g}$ can indeed be described concretely in terms of Maurer--Cartan simplicial sets of Lie algebras. We will discuss this in more detail in Section \ref{sec:Dprime}. 
\end{remark}
The technical part of the proof of our main result (Theorem \ref{thm:mainthm}) will consist of verifying the above conditions for a suitable class of operads. This will be done in Section \ref{sec:cohsmall}. In the remainder of this section, we will describe how Theorem \ref{thm:axiomatic} follows from the results of \cite{DAGX, calaque2018formal}.
\begin{proof}
Condition \ref{it:axiomC} guarantees that for every $\mf{D}(\PP)$-algebra $\mf{g}$, the functor
$$\begin{tikzcd}
\MC_\mf{g}\colon \Art_{\PP}\arrow[r] & \sS; \hspace{4pt} A\arrow[r, mapsto] & \Map_{\mf{D}(\PP)}\big(\mf{D}(A), \mf{g}\big)
\end{tikzcd}$$
does indeed define a formal moduli problem. Consequently, we obtain a well-defined functor $\MC\colon\Alg_{\mf{D}(\PP)}\rt\FMP_{\PP}$. By condition \ref{it:axiomB}, we have that
\begin{align*}
\MC_\mf{g}(\triv{c}{n}) &= \Map_{\mf{D}(\PP)}\big(\mf{D}(\triv{c}{n}),\mf{g}\big)\\
&\simeq \Map_{\mf{D}(\PP)}\big(\mm{Free}(\triv{c}{n}^\vee), \mf{g}\big) \simeq \Map_{\base^{\op}}\big(\trivop{c}{-n}, \mf{g}\big).
\end{align*}
It then follows from Lemma \ref{lem:tangentcomplexmodule} that the tangent complex of the formal moduli problem $\MC_\mf{g}$ is given by
\begin{equation}\label{eq:T=g}
T\big(\MC_\mf{g}\big)\simeq \mf{g}.
\end{equation}
In particular, if $\MC$ admits an inverse, then this inverse will necessarily send a formal moduli $F$ to $T(F)$, endowed with a $\mf{D}(\PP)$-algebra structure. To see that $\MC$ indeed does admit an inverse, let us recall the following terminology \cite[Definition 2.15]{calaque2018formal}. The class of \emph{good} $\mf{D}(\PP)$-algebras is the smallest class of algebras such that:
\begin{enumerate}
\item\label{it:good1} It contains the free algebras $\Free\big(\trivop{c}{n}\big)$ for $n\leq 0$.
\item\label{it:good2} For any pushout square
\begin{equation}\label{diag:cellattachment}\begin{tikzcd}
\Free\big(\trivop{c}{n}\big)\arrow[r]\arrow[d] & \mf{g}\arrow[d]\\
0\arrow[r] & \mf{h}
\end{tikzcd}\end{equation}
where $\mf{g}$ is good and $n\leq -1$, $\mf{h}$ is good as well.
\end{enumerate}
By condition \ref{it:axiomA}, the functor $\mf{D}$ restricts to a fully faithful embedding of the $\infty$-category of Artin $\PP$-algebras into $\Alg_{\mf{D}}^{\op}$. By conditions \ref{it:axiomB} and \ref{it:axiomC}, the essential image of this embedding is (the opposite of) a full subcategory of $\Alg_{\mf{D}(\PP)}$ which satisfies conditions \ref{it:good1} and \ref{it:good2}. In particular, it contains the good $\mf{D}(\PP)$-algebras. But then the image of the good $\mf{D}(\PP)$-algebras under $\mf{D}'$ is a full subcategory of the Artin $\PP$-algebras that satisfies the conditions of Definition \ref{def:smallalgebra}. Since the Artin algebras were the smallest subcategory with these properties, we conclude that $\mf{D}$ and $\mf{D}'$ induce an equivalence
\begin{equation}\label{diag:equivsmallgood}\begin{tikzcd}
\mf{D}\colon \Art_{\PP}\arrow[r, yshift=1ex, "\sim"{below}] & \big(\Alg_{\mf{D}(\PP)}^{\mm{good}}\big)^{\op}\colon \mf{D}'\arrow[l, yshift=-1ex].
\end{tikzcd}\end{equation}
It will now follow from \cite[Theorem 1.3.12]{DAGX} that the functor $\MC$ is an equivalence. Indeed, the conditions of \cite[Definition 1.3.1]{DAGX} hold precisely because $\mf{D}$ restricts to the equivalence \eqref{diag:equivsmallgood}. The remaining condition \cite[Definition 1.3.9]{DAGX}  asserts that the functor
$$\begin{tikzcd}
\Alg_{\mf{D}(\PP)}\arrow[r, "\MC"] & \FMP_{\PP}\arrow[r, "T"] & \prod_{c} \cat{Sp}
\end{tikzcd}$$
preserves sifted colimits. But it follows from \eqref{eq:T=g} that this functor is naturally equivalent to the composite 
$$\begin{tikzcd}[column sep=3pc]
\Alg_{\mf{D}(\PP)}\arrow[r, "\mm{forget}"] & \cat{Mod}_{\base^\op}\arrow[r] & \prod_{c} \cat{Sp}.
\end{tikzcd}$$
Forgetting the structure of an algebra over an operad always preserves sifted homotopy colimits \cite[Appendix A]{HarpazNuitenPrasma} and the second functor preserves all colimits (see Remark \ref{rem:modulestospectra}).
\end{proof}

\section{Cohomology of Artin algebras}\label{sec:cohsmall}
Let $\PP$ be a $\base$-operad and consider the adjoint pair whose left adjoint sends a $\PP$-algebra to its dual $\mf{D}(\PP)$-algebra
$$\begin{tikzcd}
\mf{D}\colon \Alg_{\PP}\arrow[r, yshift=0.8ex] & \Alg_{\mf{D}(\PP)}^{\op}\colon \mf{D}'.\arrow[l, yshift=-0.8ex]
\end{tikzcd}$$
The purpose of this section is to show that under certain conditions on the operad $\PP$, the functor $\mf{D}$ is well-behaved when restricted to the class of \emph{Artin $\PP$-algebras}. In particular, it satisfies the assumptions of Theorem \ref{thm:axiomatic}, so that the above adjunction can be refined to an equivalence between $\mf{D}(\PP)$-algebras and formal moduli problems over $\PP$. More precisely, will prove the following:
\begin{theorem}[Theorem {\ref{thm:mainthm}}]\label{thm:biduals}
Let $\base$ be a dg-category and $\mathscr{P}$ an (augmented) $\base$-operad. Assume that the following conditions hold:
\begin{enumerate}
\item $\base$ and $\mathscr{P}$ are both connective.
\item $\base$ is cohomologically bounded, i.e.\ there exists an $n\in\mathbb{N}$ such that all $H^*(\base)(c, d)$ are concentrated in degrees $[-n, 0]$.
\item $\mathscr{P}$ is splendid.
\end{enumerate}
Then the following assertions hold:
\begin{enumerate}[label=(\Alph*)]
\item\label{it:checkA} For any Artin $\PP$-algebra (Definition \ref{def:smallalgebra}), the unit map $A\rt \mf{D}'\mf{D}(A)$ is an equivalence.
\item\label{it:checkB} $\mf{D}\big(\triv{c}{n}\big)$ is freely generated by $\triv{c}{n}^\vee$, for all $c$ and $n\geq 0$.
\item\label{it:checkC} The functor $\mf{D}$ sends every pullback square of Artin algebras
\begin{equation}\label{diag:squarezero}\begin{tikzcd}
A'\arrow[r]\arrow[d] & 0\arrow[d]\\
A\arrow[r] & \triv{c}{n}
\end{tikzcd}\end{equation}
with $n\geq 1$ to a pushout square of $\mf{D}(\PP)$-algebras.
\end{enumerate}
In particular, Theorem \ref{thm:axiomatic} applies and there is an equivalence
$$
\FMP_{\PP}\simeq \Alg_{\mf{D}(\PP)}.
$$
\end{theorem}
\begin{assumption}\label{ass:Abounded}
We will assume, as usual, that $\PP$ is cofibrant as a left $\base$-module. Because $\base$ is assumed to be connective and cohomologically bounded, we will furthermore make the following chain-level assumption throughout this section: we will assume that $\base$ is a dg-category such that every $\base(c, d)$ is concentrated in degrees $[-N, 0]$, for some fixed $N$.
\end{assumption}

\subsection{Polynomial subalgebras}
Let us start with the following general observation. Let $\phi\colon \mathscr{C}\drt \mathscr{P}$ be a weakly Koszul twisting morphism from an $\base$-cooperad to a $\base$-operad. Then $\mf{D}_\phi(A)=\Bar_\phi(A)^\vee$ is given by
\begin{align*}
\mf{D}_\phi(A) &= \Big(\bigoplus_{p\geq 0} \mathscr{C}(p)\otimes_{\Sigma_{p}\ltimes \base^{\otimes p}} A^{\otimes p}\Big)^\vee\\
&\cong \prod_{p\geq 0} \Big(\mathscr{C}(p)\otimes_{\Sigma_{p}\ltimes \base^{\otimes p}} A^{\otimes p}\Big)^\vee.
\end{align*}
Consider the graded $\CC^\vee$-subalgebra
\begin{equation}\label{eq:poly}
\mf{D}_\phi^{\mm{poly}}(A) \coloneqq \bigoplus_{p\geq 0} \Big(\mathscr{C}(p)\otimes_{\Sigma_{p}\ltimes \base^{\otimes p}} A^{\otimes p}\Big)^\vee\subseteq \mf{D}_\phi(A).
\end{equation}
Note that this is not necessarily closed under the differential, but it will be if $A$ satisfies the following condition:
\begin{definition}\label{def:nilpotent}
A $\mathscr{P}$-algebra $A$ is \emph{nilpotent} if $A$ is annihilated by all operations of arity $\geq p$, for some $p$.
\end{definition}
\begin{remark}
When $\mathscr{P}$ is \emph{concentrated in arity $\geq 2$}, then a nilpotent algebra is annihilated by any composition of $\geq p$ operations in $\mathscr{P}$, for some $p$. Conversely, if $\mathscr{P}$ is generated by operations in finitely many arities and $A$ is annihilated by any composition of $\geq p$ operations, then $A$ is nilpotent.
\end{remark}
\begin{remark}
The algebra $\mf{D}_\phi^{\mm{poly}}(A)$ is not homotopy invariant: it depends on the point-set choices for $A$ and the twisting morphism $\phi$. Note that an algebra that is quasi-isomorphic to a nilpotent algebra need not be nilpotent itself.
\end{remark}
To prove Theorem \ref{thm:biduals}, it will be much more convenient to work with $\mf{D}^\mm{poly}(A)$ instead of $\mf{D}(A)$. Indeed, the following result shows that $\mf{D}^\mm{poly}(A)$ is typically much better behaved than $\mf{D}(A)$:
\begin{lemma}\label{lem:verysmalltocof}
Let $\phi\colon \mathscr{C}\drt \mathscr{P}$ be a Koszul twisting morphism. If $A$ is a strictly Artin, nilpotent $\mathscr{P}$-algebra, then $\mf{D}^{\mm{poly}}_\phi(A)$ is a cofibrant $\mathscr{C}^\vee$-algebra.
\end{lemma}
\begin{proof}
Let us start with the following general observation: if $A\rt B$ is a square zero extension of $\PP$-algebras by $\triv{c}{n}$, with $n\geq 0$, then their bar constructions fit into a pullback square of $\mathscr{C}$-coalgebras
\begin{equation}\label{diag:squarezerocoalg}\begin{tikzcd}
\Bar_\phi(A)\arrow[d]\arrow[r] & \mathscr{C}(\triv{c}{n, n+1})\arrow[d]\\
\Bar_\phi(B)\arrow[r] & \mathscr{C}(\triv{c}{n+1}).
\end{tikzcd}\end{equation}
Indeed, this follows from writing $A\cong B\oplus \triv{c}{n}$ as $\base$-modules (without differential), so that $\Bar_\phi(A)$ is obtained from $\Bar_\phi(B)$ by adding cogenerators from $\triv{c}{n}$. Assuming that $A$ and (hence) $B$ are nilpotent, we can take duals and restrict to `polynomial' subalgebras to obtain a square
\begin{equation}\label{diag:pushoutpointset}\begin{tikzcd}
\mathscr{C}^\vee\big(\trivop{c}{-n-1}\big)\arrow[d]\arrow[r] & \mf{D}_\phi^{\mm{poly}}(B)\arrow[d]\\
\mathscr{C}^\vee\big(\trivop{c}{-n, -n-1}\big)\arrow[r] & \mf{D}_\phi^{\mm{poly}}(A).
\end{tikzcd}\end{equation}
Without differentials, this square is a pushout square of $\CC^\vee$-algebras, so the same is true with differentials. Since the left vertical map is a (generating) cofibration of $\CC^\vee$-algebras, it follows that $\mf{D}_\phi^\mm{poly}(A)$ is cofibrant as soon as $\mf{D}_\phi^\mm{poly}(B)$ is.

Now suppose that $A$ is strictly Artin and nilpotent. By definition, $A$ fits into a sequence $A=A^{(n)}\rt \dots \rt A^{(0)}=0$ of square zero extensions by various $\triv{c_i}{p_i}$. Proceeding by induction, it follows that $\mf{D}_\phi^{\mm{poly}}(A)$ is cofibrant.
\end{proof}
To reduce statements about $\mf{D}(A)$ to statements about the more tractable algebra $\mf{D}^\mm{poly}(A)$, we will use of the following result:
\begin{proposition}\label{prop:filtrationissues}
Suppose that $\base$ is as in Assumption \ref{ass:Abounded} and that $\mathscr{P}$ is a splendid \ $\base$-operad, concentrated in nonpositive degrees. Let $\pi\colon \Bar\PP\drt \mathscr{P}$ be the universal Koszul twisting morphism and let $A$ be a $\mathscr{P}$-algebra which is strictly Artin and nilpotent. Then the map of $\Bar\PP^\vee$-algebras
$$\begin{tikzcd}
\mf{D}_\pi^\mm{poly}(A)\arrow[r] & \mf{D}_\pi(A)
\end{tikzcd}$$
is a quasi-isomorphism.
\end{proposition}
This proposition forms the technical heart of our proof of Theorem \ref{thm:biduals} (and hence Theorem \ref{thm:mainthm}). In particular, its proof is somewhat involved and is proven in increasing levels of generality, using some of the results of Appendix \ref{sec:operads}. We will therefore postpone the proof to Section \ref{sec:proofoftechnicalprop} and instead discuss how it can be used to prove Theorem \ref{thm:biduals}. 

As a first application of Proposition \ref{prop:filtrationissues}, we find that every Artin $\PP$-algebra can be modelled by a nilpotent algebra. More precisely, we have the following:
\begin{lemma}\label{lem:verysmall}
Consider a retract diagram of $\base$-operads $\PP\rto{\sim} \Omega\CC\rto{\sim} \PP$, where $\CC$ is filtered-cofibrant as a left $\base$-module. Then the following assertions hold:
\begin{enumerate}
\item\label{it:verysmall1} Every Artin $\PP$-algebra is quasi-isomorphic to a strictly Artin $\PP$-algebra (Definition \ref{def:verysmall}).
\item\label{it:verysmall2} Suppose that the (Koszul) twisting morphism $\phi\colon \CC\drt\PP$ associated to $\Omega\CC\rt\PP$ has the following property: for every $A$ which is strictly Artin and nilpotent, the map $\mf{D}_\phi^\mm{poly}(A)\rt \mf{D}_\phi(A)$ is a quasi-isomorphism. Then every Artin $\PP$-algebra is quasi-isomorphic to a strictly Artin $\PP$-algebra which is furthermore nilpotent.
\end{enumerate}
\end{lemma}
Every cofibrant $\base$-operad $\PP$ fits into a retract diagram $\PP\rto{\sim} \Omega\Bar\PP\rto{\sim} \PP$. Consequently, part \ref{it:verysmall1} asserts that Artin algebras over cofibrant operads are quasi-isomorphic to strictly Artin algebras.
\begin{proof}
The Artin $\PP$-algebras form the smallest class of $\PP$-algebras that is closed under homotopy pullbacks along the maps of trivial algebras $0\rt \triv{c}{n+1}$ (with $n\geq 0$). It therefore suffices to show the following: let $A$ be a strictly Artin $\cat{P}$-algebra, let $\Omega_\phi\Bar_\phi(A)\rt A$ be its bar-cobar resolution (using Lemma \ref{lem:algbarresol}, since strictly Artin algebras are $\base$-cofibrant by Remark \ref{rem:smallisperfect}) and consider any (homotopy) pullback diagram of the form
\begin{equation}\label{diag:pullbacksquarezero}\begin{tikzcd}
Y\arrow[d]\arrow[r]\arrow[d] & \triv{c}{n, n+1}\arrow[d]\\
\Omega_\phi\Bar_\phi(A)\arrow[r, "\chi"] & \triv{c}{n+1}.
\end{tikzcd}\end{equation}
Then the map $Y\rt \Omega_\phi\Bar_\phi(A)$ is naturally quasi-isomorphic to a square zero extension $B\rt A$ with kernel $\triv{c}{n}$. For part \ref{it:verysmall2}, we must furthermore show that $B$ can be taken nilpotent, assuming $A$ is nilpotent.

We will only prove this assertion for part \ref{it:verysmall2}; the argument for part \ref{it:verysmall1} is similar but easier. Let us denote by $i\colon \PP\rt \Omega\CC$ and $r\colon \Omega\CC\rt\PP$ the inclusion and retraction, and let $\psi\colon \CC\drt \Omega\CC$ denote the universal twisting morphism. There are natural maps
$$\begin{tikzcd}
B_\phi(A)\arrow[r, "\cong"] & B_\psi(r^*A) & & i^*\Omega_\psi(C)\arrow[r] & \Omega_\phi(C)
\end{tikzcd}$$
for a $\PP$-algebra $A$ and a $\CC$-coalgebra $C$. The first map is an isomorphism and the second map is obtained by applying $i^*$ to the natural map $\Omega_\psi(C)\rt r^*\Omega_\phi(C)$. Now observe that there are bijections
$$
\chi\colon \Omega_\phi\Bar_\phi(A)\rt \triv{c}{n+1} \qquad \Longleftrightarrow \qquad \chi\colon \Bar_\phi(A)\rt \mathscr{C}(\triv{c}{n+1})
$$
where $\CC\big(\triv{c}{n+1}\big)$ is the cofree $\CC$-coalgebra on a single generator of degree $n+1$ at place $c$. A map $\chi\colon \Omega_\phi\Bar_\phi(A)\rt \triv{c}{n+1}$ therefore corresponds to a degree $-(n+1)$ cycle $\chi\in \mf{D}_\phi(A)(c)$. Homologous cycles correspond to homotopic maps, and hence give rise to weakly equivalent homotopy pullbacks $Y$. We may therefore change $\chi$ by a coboundary and assume that it is contained in the image of the quasi-isomorphism
$$\begin{tikzcd}
\mf{D}_\phi^\mm{poly}(A)\arrow[r] & \mf{D}_\phi(A).
\end{tikzcd}$$
Now consider the pullback square of $\cat{C}$-coalgebras
\begin{equation}\label{diag:pullbackcoalgebras}\begin{tikzcd}
C'\arrow[d]\arrow[r] & \cat{C}\big(\triv{c}{n, n+1}\big)\arrow[d]\\
\Bar_\phi(A)\arrow[r, "\chi"{swap}] & \cat{C}\big(\triv{c}{n+1}\big).
\end{tikzcd}\end{equation}
Unravelling the definitions, one sees that the map $C'\rt \Bar_\phi(A)$ is isomorphic to a map of the form $\Bar_\psi(A')\rt \Bar_\psi(r^*A)$, where $A'\rt r^*A$ is a square zero extension of $\Omega\CC$-algebras with kernel $\triv{c}{n}$. In particular, $A'$ is a strictly Artin $\Omega\CC$-algebra. 

To see that $A'$ is a nilpotent $\Omega\CC$-algebra, we use that $A'\cong A\oplus \triv{c}{n}$ is a square zero extension of $r^*A$ by a trivial $r^*A$-module. Each generator $\mu\in \mc{C}[-1]\subseteq \Omega\mc{C}$ acts on $A'$ by
$$\begin{tikzcd}[column sep=2.1pc]
\mu_{A'}\colon \big(A\oplus \triv{c}{n}\big)^{\otimes p}\arrow[r, two heads] & A^{\otimes p}\arrow[rr, "{(\mu_A, \chi(\mu, -))}"] & & A\oplus \triv{c}{n}.
\end{tikzcd}$$
Here $\chi(\mu, -)$ denotes the composite
$$\begin{tikzcd}[column sep=3pc]
A^{\otimes p}\arrow[r, "\mu\otimes \mm{id}"] & {\CC(p)[-1]\otimes A^{\otimes p}\subseteq B_\phi(A)}\arrow[r, "{\chi[-1]}"] & \triv{c}{n}.
\end{tikzcd}$$
By our assumption that $\chi$ lies in the image of $\mf{D}^\mm{poly}(A)$, the generating operations $\chi(\mu, -)$ vanish when the arity of $\mu$ is high enough. Furthermore, the composition of at least two such generating operations maps $A$ to $A$ and vanishes on $\triv{c}{n}$. Because $A$ was assumed to be a nilpotent $\PP$-algebra, it follows that such composite operations also vanish if their arity is high enough. We conclude that $A'$ is a nilpotent $\Omega\CC$-algebra.

Now, applying functor $i^*\Omega_\psi$ to \eqref{diag:pullbackcoalgebras} and using that there is a natural map $i^*\Omega_\psi\rt \Omega_\phi$, we obtain a diagram of $\PP$-algebras
\begin{equation}\label{diag:threesquares}\begin{tikzcd}
i^*A'\arrow[d] & i^*\Omega_\psi\Bar_\psi(A')\arrow[d]\ar[l, "\sim"{above}]\arrow[r] & \Omega_\phi\Bar_\phi(\triv{c}{n+1, n})\arrow[r, "\sim"]\arrow[d] & \triv{c}{n+1, n}\arrow[d]\\
A & \Omega_\phi\Bar_\phi(A)\arrow[l, "\sim"{above}]\arrow[r] & \Omega_\phi\Bar_\phi(\triv{c}{n+1})\arrow[r, "\sim"] & \triv{c}{n+1}.
\end{tikzcd}\end{equation}
Taking $B=i^*A'$, we obtain a nilpotent square zero extension of $A$. The above diagram shows that it is related to the pullback $Y$ of \eqref{diag:pullbacksquarezero} by a zig-zag
$$\begin{tikzcd}
B=i^*A' & i^*\Omega_\psi\Bar_\psi(A')\arrow[l, "\sim"{above}] \arrow[r] & Y.
\end{tikzcd}$$
It remains to verify that the right map is a quasi-isomorphism, for which we can work at the level of the underlying complexes. But forgetting $\PP$-algebra structures, there are natural sections $i^*A'\rt i^* \Omega_\psi\Bar_\psi(A')$ and $A\rt \Omega_\phi\Bar_\phi(A)$ that make the composition of the three squares a (homotopy) pullback square of chain complexes. Consequently, we find maps of complexes
$$\begin{tikzcd}
B=A\times^h_{\triv{c}{n+1}}0\arrow[r] & \Omega_\psi\Bar_\psi(A')\arrow[r] & \Omega_\phi\Bar_\phi(A)\times_{\triv{c}{n+1}}^h 0=Y.
\end{tikzcd}$$
The first map and the composite map are quasi-isomorphisms, so that $i^*\Omega_\psi\Bar_\psi(A')\rt Y$ is a quasi-isomorphism, as desired.
\end{proof}
\begin{corollary}\label{cor:verysmallpseudonilmodels}
Suppose that $\base$ is as in Assumption \ref{ass:Abounded} and that $\mathscr{P}$ is a splendid cofibrant $\base$-operad, concentrated in nonpositive cohomological degrees. Then every Artin $\PP$-algebra is quasi-isomorphic to a strictly Artin $\PP$-algebra which is nilpotent.
\end{corollary}
\begin{proof}
Apply part \ref{it:verysmall2} of Lemma \ref{lem:verysmall} to the retract diagram $\PP\rt \Omega \Bar\PP\rt \PP$, where first map exists since $\PP$ is assumed cofibrant.
\end{proof}

\subsection{Proof of Theorem \ref{thm:biduals}}
In this section, we will prove Theorem \ref{thm:biduals}, and hence Theorem \ref{thm:mainthm}, using Proposition \ref{prop:filtrationissues} (whose proof will be taken up in Section \ref{sec:proofoftechnicalprop}). Since the statement of Theorem \ref{thm:biduals} only depends on the quasi-isomorphism classes of $\base$ and $\PP$, we are allowed to make the following assumptions throughout this section: we will assume that $\base$ is bounded, as in Assumption \ref{ass:Abounded}, and that $\PP$ is a cofibrant $\base$-operad which is concentrated in nonpositive cohomological degrees. We denote by
$$\begin{tikzcd}
\pi\colon\Bar\PP\arrow[r, dashrightarrow] & \PP
\end{tikzcd}$$
the universal Koszul twisting morphism and will model $\mf{D}\colon \Alg_{\PP}\rt \Alg_{\mf{D}(\PP)}^{\op}$ by $\mf{D}_\pi$.

\begin{proof}[Proof of Theorem \ref{thm:biduals} \ref{it:checkA}]
Suppose that $A$ is an Artin $\PP$-algebra. By Corollary \ref{cor:verysmallpseudonilmodels}, we can assume that $A$ is strictly Artin and nilpotent. To verify that the unit map 
$$\begin{tikzcd}
A\arrow[r] & \mf{D}_\pi'\mf{D}_\pi(A)
\end{tikzcd}$$
is an equivalence, it suffices to work at the level of the underlying $\base$-modules. By Corollary \ref{cor:adjointisderivations}, the functor $\mf{D}'_\pi$ is given at the level of $\base$-modules by the derived functor of $B\mapsto \mm{Der}(B, \base)$. By Lemma \ref{lem:verysmalltocof} and Proposition \ref{prop:filtrationissues}, a cofibrant resolution of $\mf{D}_\pi(A)$ is given by the polynomial subalgebra $\mf{D}_\pi^\mm{poly}(A)$. It therefore suffices to verify that the natural map
$$\begin{tikzcd}
A\arrow[r] & \mm{Der}\big(\mf{D}_\pi^\mm{poly}(A), \base\big)
\end{tikzcd}$$
is a quasi-isomorphism. Since $\mf{D}_\pi^\mm{poly}(A)$ is the free graded algebra on $A^\vee$, one can identify the underlying map of graded $\base$-modules with the canonical map
$$\begin{tikzcd}
A\arrow[r] & A^{\vee\vee}.
\end{tikzcd}$$
This is an isomorphism since $A$ is a finitely generated quasi-free $\base$-module (Remark \ref{rem:smallisperfect}).
\end{proof}
For part \ref{it:checkB} of Theorem \ref{thm:biduals}, let us make the following more general observation:
\begin{proposition}\label{prop:functorialitysmall}
Let $\base$ be a bounded connective dg-category and $f\colon \PP\rt \QQ$ a map of augmented $\base$-operads which are connective and splendid. Let $\mf{D}(f)\colon \mf{D}(\QQ)\rt \mf{D}(\PP)$ be the induced map on bar dual operads. For every $\QQ$-algebra $A$, there is a natural map of $\mf{D}(\PP)$-algebras
$$\begin{tikzcd}
(\mf{D}f)_!\big(\mf{D}(A)\big)\arrow[r] & \mf{D}\big(f^*(A)\big).
\end{tikzcd}$$
This map is an equivalence whenever $A$ is a Artin $\QQ$-algebra.
\end{proposition}
\begin{proof}
We can assume that $\PP$ and $\QQ$ are cofibrant $\base$-operads and consider the map between twisting morphisms 
$$\begin{tikzcd}
\Bar\PP\arrow[d, "\Bar f"{swap}]\arrow[r, dashrightarrow, "\phi"] & \PP\arrow[d, "f"]\\
\Bar\QQ\arrow[r, dashrightarrow, "\psi"{swap}] & \QQ.
\end{tikzcd}$$
Let $\Bar(f)^*$ denote the forgetful functor from $\Bar\PP$-coalgebras to $\Bar\QQ$-coalgebras. Then there is a natural map of $\Bar\QQ$-coalgebras for every $\QQ$-algebra $A$
$$\begin{tikzcd}
B(f)^*\Bar_\phi(f^*(A))\arrow[r] & \Bar_{\psi}(A).
\end{tikzcd}$$
Without differentials, this is given by the obvious map $\Bar\PP(A)\rt \Bar\QQ(A)$ into the cofree $\Bar\QQ$-coalgebra on $A$. Taking duals gives a map of $\mf{D}\QQ$-algebras 
$$\begin{tikzcd}
\mf{D}_\psi(A)\arrow[r] & \mf{D}(f)^*\mf{D}_\phi\big(f^*(A)\big).
\end{tikzcd}$$
The desired natural map of algebras over $\mf{D}(\PP)$ is then obtained by adjunction, i.e.\ by (derived) inducing up along $\mf{D}\QQ\rt \mf{D}\PP$.

Now suppose that $A$ is a Artin $\QQ$-algebra. By Corollary \ref{cor:verysmallpseudonilmodels}, we may assume that $A$ is strictly Artin and nilpotent. By Proposition \ref{prop:filtrationissues} and Lemma \ref{lem:verysmalltocof}, there are cofibrant resolutions 
$$\begin{tikzcd}
\mf{D}_{\psi}^\mm{poly}(A)\arrow[r, "\sim"] & \mf{D}_{\psi}(A) & \mf{D}_\phi^\mm{poly}(f^*(A))\arrow[r, "\sim"] & \mf{D}_\phi(f^*(A)).
\end{tikzcd}$$
In particular, we obtain a commuting square
$$\begin{tikzcd}
\mf{D}(f)_!\mf{D}_{\psi}^\mm{poly}(A)\arrow[d, "\sim"{swap}]\ar[r] & \mf{D}_\phi^\mm{poly}(f^*(A))\arrow[d, "\sim"]\\
\mf{D}(f)_!\mf{D}_{\psi}(A)\arrow[r] & \mf{D}_\phi(f^*(A))
\end{tikzcd}$$
where in the second row, $\mf{D}(f)_!$ (implicitly) denotes the derived functor. Unravelling the definitions, the top horizontal map is given without differentials by the natural map
$$\begin{tikzcd}
\mf{D}(\PP)\circ_{\mf{D}(\QQ)} \Big(\mf{D}(\QQ)\circ_{\base^{\op}} A^\vee\Big)\arrow[r] & \mf{D}(\PP)\circ_{\base^{\op}} A^\vee.
\end{tikzcd}$$
This map is an isomorphism, so the result follows.
\end{proof}
\begin{proof}[Proof of Theorem \ref{thm:biduals} \ref{it:checkB}]
This is the special case of Proposition \ref{prop:functorialitysmall} where the map $\mathscr{P}\rt \mathscr{P}'=\base$ is the augmentation map.
\end{proof}
\begin{proof}[Proof of Theorem \ref{thm:biduals} \ref{it:checkC}]
Let $A$ be an Artin $\PP$-algebra and consider a pullback square \eqref{diag:squarezero} in the $\infty$-category of $\mathscr{P}$-algebras. Inspecting the proof of Lemma \ref{lem:verysmall} (cf.\ Diagram \eqref{diag:threesquares}), one can present such a square in the $\infty$-category of $\PP$-algebras by a strict diagram of $\PP$-algebras of the form
$$\begin{tikzcd}
B\arrow[d, two heads, "p"{swap}] & \tilde{B}\arrow[r]\ar[l, "\sim"{above}]\arrow[d, two heads] & \triv{c}{n, n+1}\arrow[d, two heads]\\
A & \tilde{A}\arrow[l, "\sim"{below}]\arrow[r] & \triv{c}{n+1},
\end{tikzcd}$$
where $p\colon B\rt A$ is a square zero extension of strictly Artin, nilpotent $\PP$-algebras. By a standard model categorical argument, one can in fact assume that the surjective map $\tilde{B}\rt \tilde{A}$ is given by $\Omega\Bar(p)\colon \Omega_\phi\Bar_\phi(B)\rt \Omega_\phi\Bar_{\phi}(A)$, and that the left two quasi-isomorphisms are the canonical maps from the bar-cobar resolution. 

Now apply the bar construction $\Bar_\phi$ to the above diagram. Then the left two weak equivalences admit canonical sections. Using these canonical sections, one obtains a composite square of $\Bar\PP$-coalgebras of the form \eqref{diag:squarezerocoalg}, which is cartesian. After dualizing, one obtains a square of the form
$$\begin{tikzcd}
\mf{D}_\phi\big(\triv{c}{n+1}\big)\arrow[r]\arrow[d] & \mf{D}_\phi(A)\arrow[d]\\
\mf{D}_\phi\big(\triv{c}{n, n+1}\big)\arrow[r] & \mf{D}_\phi(B).
\end{tikzcd}$$
We have to show that this square is a homotopy pushout square of $\mf{D}(\PP)$-algebras. Since all $\PP$-algebras involved in this square are strictly Artin and conilpotent, Proposition \ref{prop:filtrationissues} implies the above square is naturally equivalent to the square \eqref{diag:pushoutpointset} of polynomial subalgebras. But then the proof of Lemma \ref{lem:verysmalltocof} shows that this square is a (homotopy) pushout square of $\mf{D}(\PP)$-algebras (cf.\ Diagram \eqref{diag:pushoutpointset}).
\end{proof}
We conclude that the functor $\mf{D}\colon \Alg_{\PP}\rt \Alg_{\mf{D}(\PP)}^{\op}$ satisfies the conditions of Theorem \ref{thm:axiomatic}. In particular, this says that the functor
$$\begin{tikzcd}
\MC\colon \cat{Alg}_{\mf{D}(\mathscr{P})}\arrow[r, "\sim"] & \cat{FMP}_{\mathscr{P}}; \hspace{4pt}\mf
{g}\arrow[r, mapsto] & \Map_{\mf{D}(\mathscr{P})}(\mf{D}(-), \mf{g}).
\end{tikzcd}$$
is an equivalence of $\infty$-categories, with inverse sending a formal moduli problem $F$ to $T(F)$. This proves Theorem \ref{thm:mainthm}. 

\begin{variant}\label{var:cobarC}
Let $\CC$ be a $\base$-cooperad which is filtered-cofibrant as a $\base$-module and let $\iota\colon \CC\rt \Omega\CC=\PP$ be the universal twisting morphism. Inspecting the above proof, one sees that the conclusions of Theorem \ref{thm:biduals} remain valid as long as $\mf{D}_\iota^{\mm{poly}}(A) \rt \mf{D}_\iota(A)$ is a quasi-isomorphism for every $A$ that is strictly Artin and nilpotent. Consequently, Theorem \ref{thm:mainthm} then holds for the operad $\PP=\Omega\CC$. 

As an important example of this situation, let us record the following. Suppose that $\base$ is a dg-category such that all $\base(c,d)$ are concentrated in some fixed interval $[a, b]$, and suppose that $\CC$ is a $\base$-cooperad with the following property: $\CC(p)$ is concentrated in degrees $\leq f(p)$, with 
$$\begin{tikzcd}[column sep=3pc]
f(p)\arrow[r, "p\rightarrow \infty"] & -\infty.
\end{tikzcd}$$
Note that when $A$ is strictly Artin, there is an $n$ such that any $n$-fold composition of generating operations acts trivially on $A$. Since $A$ is concentrated in finitely many degrees (Remark \ref{rem:smallisperfect}), this means that such $A$ is automatically nilpotent (Definition \ref{def:nilpotent}). Furthermore, the map $\mf{D}_\iota^\mm{poly}(A)\rt \mf{D}_\iota(A)$ is then an \emph{isomorphism} for degree reasons (cf.\ the proof of Lemma \ref{lem:filtrationissuestrunc}). The above proof and Theorem \ref{thm:axiomatic} then imply that there is equivalence of $\infty$-categories
$$\begin{tikzcd}
\Alg_{\CC^\vee}= \Alg_{\mf{D}(\PP)}\arrow[r, "\sim"] & \FMP_{\PP}.
\end{tikzcd}$$
Note that $\CC$ may have contributions from positive degrees, as long as it is \emph{eventually} concentrated in sufficiently negative degrees. In particular, this hold when $\CC$ concentrated in finitely many arities.
\end{variant}

\subsection{Proof of Proposition \ref{prop:filtrationissues}}\label{sec:proofoftechnicalprop}
This section is devoted to the proof of Proposition \ref{prop:filtrationissues}. Throughout, we assume that $\base$ is as in Assumption \ref{ass:Abounded}, i.e.\ concentrated in cohomological degrees $[-N, 0]$, and that $\PP$ is a $\base$-operad in nonpositive degrees. We will prove Proposition \ref{prop:filtrationissues} in increasing levels of generality, starting with the following special case:
\begin{lemma}\label{lem:filtrationissuestrunc}
Suppose that $\mathscr{P}=\mathscr{P}^{\leq p}$ is nonpositively graded and concentrated in arities $\leq p$, and let $\pi\colon \Bar\PP\rt \mathscr{P}$ be the universal twisting morphism. If $A$ is a nonpositively graded $\PP$-algebra, then the map of $\Bar\PP^\vee$-algebras
$$\begin{tikzcd}
\mf{D}_\pi^\mm{poly}(A)\arrow[r] & \mf{D}_\pi(A)
\end{tikzcd}$$
is an isomorphism.
\end{lemma}
\begin{proof}
Since $\mathscr{P}$ is concentrated in arities $\leq p$, its bar construction is generated by operations in arities $\leq p$ and degrees $\leq -1$. This means that the arity $q$ part of $\Bar\PP$ is concentrated in degrees $\leq -q/p$. Consequently, each term
$$
\Bar\mathscr{P}(q)\otimes_{\Sigma_q\ltimes \base^{\otimes q}} A^{\otimes q}
$$
is concentrated in degrees $\leq -q/p$. Since $\base$ is concentrated in degrees $[-N, 0]$, the $\base$-linear dual is concentrated in degrees $\geq q/p-N$ in arity $q$. 
Consequently, in each degree there are only finitely many arities that contribute to $\mf{D}(A)$, i.e.\ the map
$$\begin{tikzcd}
\bigoplus\limits_{q\geq 0} \Big(\Bar\mathscr{P}(q)\otimes_{\Sigma_q\ltimes \base^{\otimes q}} A^{\otimes q}\Big)^\vee\arrow[r] & \prod\limits_{q\geq 0} \Big(\Bar\mathscr{P}(q)\otimes_{\Sigma_q\ltimes \base^{\otimes q}} A^{\otimes q}\Big)^\vee
\end{tikzcd}$$
is an isomorphism in each individual degree.
\end{proof}
Let us next consider the case of a \emph{0-reduced} $\base$-operad $\mathscr{P}$, i.e.\ $\mathscr{P}(0)=0$. Then the tower of quotients
$$\begin{tikzcd}
\mathscr{P}\arrow[r] & \dots \arrow[r] & \mathscr{P}^{\leq p}\arrow[r] & \mathscr{P}^{\leq p-1}\arrow[r] & \dots\ar[r] & \mathscr{P}^{\leq 1}
\end{tikzcd}$$
is a tower of operads. By definition, every nilpotent $\mathscr{P}$-algebra $A$ can be considered as a $\mathscr{P}^{\leq p_0}$-algebra, for some $p_0$.
\begin{lemma}\label{lem:dpolyascolimit}
Let $\mathscr{P}$ be a 0-reduced $\base$-operad, concentrated in nonpositive degrees, and let 
$$\begin{tikzcd}
\pi\colon \Bar\PP\arrow[r, dashrightarrow] & \mathscr{P} & \text{and} & \pi^{\leq p}\colon \Bar(\mathscr{P}^{\leq p})\arrow[r, dashrightarrow] & \mathscr{P}^{\leq p}
\end{tikzcd}$$
denote the universal twisting morphisms. For each $\mathscr{P}^{\leq p_0}$-algebra $A$ in nonpositive degrees there is a natural square of chain complexes
$$\begin{tikzcd}
\colim_{p\geq p_0}\mf{D}^{\mm{poly}}_{\pi^{\leq p}}(A)\arrow[r, "\cong"] \arrow[d, "\cong"{swap}] & \mf{D}^{\mm{poly}}_\pi(A)\arrow[d]\\
\colim_{p\geq p_0}\mf{D}_{\pi^{\leq p}}(A)\arrow[r] & \mf{D}_\pi(A).
\end{tikzcd}$$
in which the two marked arrows are isomorphisms.
\end{lemma}
\begin{proof}
Recall that for every map of twisting morphisms $\phi\rt \phi'$, there is a natural map of chain complexes $\mf{D}_{\phi'}(A)\rt \mf{D}_{\phi}(A)$. When $A$ is nilpotent, this restricts to polynomial subalgebras. This gives the desired square. The vertical arrow is an isomorphism by Lemma \ref{lem:filtrationissuestrunc} and the horizontal arrow is given without differentials by the map
$$\begin{tikzcd}
\bigoplus\limits_{q\geq 0} \colim_{p\geq p_0} \Big(\Bar(\mathscr{P}^{\leq p})(q)\otimes_{\Sigma_q\ltimes \base^{\otimes q}} A^{\otimes q}\Big)^\vee\arrow[r] & \bigoplus\limits_{q\geq 0} \Big(\Bar\mathscr{P}(q)\otimes_{\Sigma_q\ltimes \base^{\otimes q}} A^{\otimes q}\Big)^\vee.
\end{tikzcd}$$
This map is an isomorphism. Indeed, the tower
$$\begin{tikzcd}
\Bar\mathscr{P}(q)\arrow[r] & \dots\arrow[r] & B(\mathscr{P}^{\leq p+1})(q)\arrow[r] & B(\mathscr{P}^{\leq p})(q)\arrow[r] & \dots
\end{tikzcd}$$
becomes stationary as soon as $p\geq q$ , so that the sequence obtained by tensoring with $A^{\otimes q}$ and taking $\base$-linear duals becomes stationary for $p\geq q$ as well.
\end{proof}
\begin{corollary}\label{cor:polyhoinv}
Let $\mathscr{P}$ be a 0-reduced $\base$-operad in nonnegative degrees and let $A$ be a $\mathscr{P}^{\leq p_0}$-algebra in nonnegative degrees, for some $p_0$. Let $\pi\colon \Bar\PP\drt \mathscr{P}$ be the universal twisting morphism. Then the map of complexes
$$\begin{tikzcd}
\mf{D}^\mm{poly}_\pi(A)\arrow[r] & \mf{D}_\pi(A)
\end{tikzcd}$$
can be identified with the natural map
$$\begin{tikzcd}
\hocolim_{p\geq p_0}\mf{D}^{\leq p}(A)\arrow[r] & \mf{D}(A)
\end{tikzcd}$$
where $\mf{D}^{\leq p}\colon \cat{Alg}_{\mathscr{P}^{\leq p}}\rt \cat{Alg}_{\mf{D}(\mathscr{P}^{\leq p})}$ and $\mf{D}\colon \cat{Alg}_{\mathscr{P}}\rt \cat{Alg}_{\mf{D}(\mathscr{P})}$.
\end{corollary}
In particular, Corollary \ref{cor:polyhoinv} furnishes a homotopy-invariant characterization of the map $\mf{D}^\mm{poly}_\pi(A)\rt \mf{D}_\pi(A)$,  as long as we take all our operads and algebras to be nonpositively graded: it no longer depends on the specific point-set models for the twisting morphism $\pi$ or $A$ (as long as these models are nonpositively graded).
\begin{proposition}\label{prop:filtrationissues0-reduced}
Let $\base$ be a dg-category in degrees $[-N, 0]$ and let $\mathscr{P}$ be a splendid, 0-reduced $\base$-operad, concentrated in degrees $\leq 0$. Let $A$ be a $\mathscr{P}^{\leq p_0}$-algebra which is freely generated as a $\base$-module by generators of degrees $\leq 0$, with \emph{finitely many} generators of degree $0$. Then the map
$$\begin{tikzcd}
\hocolim_{p\geq p_0}\mf{D}^{\leq p}(A)\arrow[r] & \mf{D}(A)
\end{tikzcd}$$
is an equivalence. 
\end{proposition}
\begin{proof}
We can work at the level of chain complexes. Since $\mf{D}$ and $\mf{D}^{\leq p}$ are homotopy invariant, we may resolve the tower $\mathscr{P}\rt \dots \rt \mathscr{P}^{\leq p}\rt \dots $ by a tower of cofibrant $\base$-operads
$$\begin{tikzcd}
\QQ\arrow[r] & \dots \arrow[r] & \QQ^{(p)}\arrow[r] & \QQ^{(p-1)}\arrow[r] & \dots
\end{tikzcd}$$
with the properties described in Proposition \ref{prop:genfortruncations}. In particular, each $\QQ^{(p)}$ is a quasi-free $\base$-operad generated by a nonnegatively graded, cofibrant $\base$-symmetric sequence $V^{(p)}$.

For each the quasi-free $\base$-operad $\QQ=\mm{Free}(V)$, the right $\QQ$-module $\base$ admits a cofibrant resolution of the form
$$
\mathscr{K}=\mm{Cone}\Big(V\circ_{\base} \QQ\rt \QQ\Big).
$$
By Remark \ref{rem:bar of algebra as derived comp}, the underlying complex of $\mf{D}(A)$ can be identified with
$$
\Big(\mathscr{K}\circ_{\QQ} A\Big)^\vee \cong \Big(\big(\mb{1}\oplus V[1]\big)\circ_{\base} A\Big)^\vee
$$
with some differential. Let us now decompose $A=A_0\oplus \ol{A}$, where $A_0$ is the $\base$-module generated by the (finitely many) degree $0$ generators and $\ol{A}$ is generated by elements of degree $<0$. By Proposition \ref{prop:genfortruncations}, $V^{(p)}$ is concentrated in increasingly negative degrees as its arity increases. Using that $A_0$ is free on finitely many generators, one then sees that
$$
\Big(\mathscr{K}\circ_{\QQ^{(p)}} A\Big)^\vee \cong \prod_{q, r\geq 0} M(q, r)
$$
with some differential, where
$$
M(q, r)\coloneqq \Big((\mb{1}\oplus V[1]\big)(q+r)\otimes_{\Sigma_q\ltimes \base^{\otimes q}} \ol{A}^{\otimes q}\Big)^{\vee}\otimes_{\Sigma_r\ltimes \base^{\otimes r}} (A_0^\vee)^{\otimes r}.
$$
Since $\ol{A}$ is concentrated in degrees $\leq -1$ and $\base$ is concentrated in degrees $[-N, 0]$, $M(q, r)$ is concentrated in degrees $\geq q-N$. Consequently, in each fixed cohomological degree there are only contributions of the $M(q, r)$ for finitely many $q$. 

Similarly, $V$ is concentrated in increasingly negative degrees as the arity increases. Consequently, in each fixed cohomological degree there are only contributions of the $M(q, r)$ for finitely many $r$. It follows that the above product over $q$ and $r$ is isomorphic to a direct sum, so that
\begin{equation}\label{eq:doublesum}
\mf{D}(A)\simeq \Big(\mathscr{K}\circ_{\QQ^{(p)}} A\Big)^\vee \cong \bigoplus_{q, r\geq 0} M(q, r).
\end{equation}
The same analysis applies to each of the graded-free operads $\tilde{\mathscr{P}}^{(p)}=\mm{Free}(V^{(p)})$ with $p\geq p_0$. Consequently, one finds that the sequence of chain complexes $\dots\rt \mf{D}^{\leq p}(A)\rt \mf{D}^{\leq p+1}(A)\rt \dots\rt \mf{D}(A)$ is quasi-isomorphic to the sequence of sums
\begin{equation}\label{eq:doublesumindexed}\begin{tikzcd}[column sep=1pc]
\dots\arrow[r] & \bigoplus\limits_{q, r\geq 0} M^{(p)}(q, r) \arrow[r] & \bigoplus\limits_{q, r\geq 0} M^{(p+1)}(q, r)\arrow[r] & \dots\ar[r] & \bigoplus\limits_{q, r\geq 0} M(q, r)
\end{tikzcd}\end{equation}
endowed with some differential. We claim that this sequence is a colimit sequence of complexes. This means that it is also a homotopy colimit, which proves the proposition.

To see that \eqref{eq:doublesumindexed} is a colimit sequence, it suffices to prove that for every fixed $r$, the sequence of graded vector spaces
$$
\bigoplus_{q\geq 0} \Big((\mb{1}\oplus V^{(p)}[1]\big)(q+r)\otimes_{\Sigma_q\ltimes \base^{\otimes q}} \ol{A}^{\otimes q}\Big)^{\vee}
$$
is a colimit sequence. We claim that this sequence becomes stationary in every fixed cohomological degree. Indeed, since the $q$-th summand is concentrated in degrees $\geq q$ (since $\ol{A}$ is generated by elements of degree $\leq -1$), only finitely many summands contribute to each individual degree. It therefore suffices to verify that for each $q$ and $r$, the sequence of graded vector spaces
$$
\Big((\mb{1}\oplus V^{(p)}[1]\big)(q+r)\otimes_{\Sigma_q\ltimes \base^{\otimes q}} \ol{A}^{\otimes q}\Big)^{\vee}
$$
becomes stationary as $p\rightarrow \infty$. But this follows from the construction of Proposition \ref{prop:genfortruncations}, which guaranteed that $V^{(p)}(q+r)$ is constant for $p\geq q+r$.
\end{proof}
To deduce Proposition \ref{prop:filtrationissues} from Proposition \ref{prop:filtrationissues0-reduced}, we now only have deal with the extra operations in arity zero that obstruct the use of the tower of quotients $\mathscr{P}\rt \mathscr{P}^{\leq p}$. This is done by a filtration argument:
\begin{construction}\label{con:removingnull}
Let $\mathscr{P}$ be a $\base$-operad which is cofibrant as a left $\base$-module and let $\pi\colon \Bar\PP\rt \mathscr{P}$ be the universal twisting morphism. For any $\mathscr{P}$-algebra $A$ which is cofibrant as a $\base$-module, we can filter the bar construction $\Bar_\pi(A)=\Bar\mathscr{P}\circ_{\base} A$ by word length in the nullary operations of $\mathscr{P}$. This is an increasing filtration by left $\base$-modules which preserves the bar differential.

The associated graded can be described as follows: let $\mathscr{P}^{\geq 1}$ denote part of $\mathscr{P}$ in nonzero arity and let $\pi^{\geq 1}\colon B(\mathscr{P}^{\geq 1})\rt \mathscr{P}^{\geq 1}$ be the universal twisting morphism. Then we can identify
$$
\mm{gr}\big(\Bar_\pi(A)\big) = B_{\pi^{\geq 1}}\big(A\oplus \mathscr{P}(0)[1]\big)
$$
where $A\oplus \mathscr{P}(0)[1]$ is the product of $A$, considered as a $\mathscr{P}^{\geq 1}$-algebra by restriction, and the trivial algebra $\mathscr{P}(0)[1]$. Since all pieces are cofibrant as $\base$-modules, dualizing yields a complete Hausdorff filtration on $\mf{D}_\pi(A)$ whose associated graded is
$$
\mm{gr}\big(\mf{D}_\pi(A)\big) = \mf{D}_{\pi^{\geq 1}}\big(A\oplus \mathscr{P}(0)[1]\big)
$$
\end{construction}
\begin{lemma}\label{lem:filtrationnullary}
Suppose that $\mathscr{P}$ is a $\base$-operad in nonpositive degrees and let $A$ be a $\mathscr{P}$-algebra which is strictly Artin and nilpotent. Then the complete Hausdorff filtration on $\mf{D}_\pi(A)$ from Construction \ref{con:removingnull} restricts to a complete Hausdorff filtration on $\mf{D}^{\mm{poly}}_\pi(A)$. Furthermore, the map $\mf{D}^{\mm{poly}}_\pi(A)\rt \mf{D}_\pi(A)$ induces the obvious map
$$\begin{tikzcd}
\mf{D}_{\pi^{\geq 1}}^{\mm{poly}}\big(A\oplus \mathscr{P}(0)[1]\big)\arrow[r] & \mf{D}_{\pi^{\geq 1}}\big(A\oplus \mathscr{P}(0)[1]\big)
\end{tikzcd}$$
at the level of the associated graded.
\end{lemma}
\begin{proof}
Let us first check that the induced filtration on $\mf{D}^{\mm{poly}}_\pi(A)$ is complete Hausdorff. By Construction \ref{con:removingnull}, we can write
$$
\Bar_\pi(A)\cong \bigoplus_{q, r\geq 0} \Big(\Bar(\mathscr{P}^{\geq 1})(q+r)\otimes_{\Sigma_q\ltimes\base^{\otimes q}} \mathscr{P}(0)^{\otimes q}\Big)[q]\otimes_{\Sigma_r\ltimes \base^{\otimes r}} A^{\otimes r}
$$
as left $\base$-modules, with some differential. The filtration is indexed by $q$. Since $A$ is finitely generated quasi-free over $\base$, the $\base$-linear dual of each of summand is given by
$$
N(q, r)\coloneqq \Big(\Bar(\mathscr{P}^{\geq 1})(q+r)\otimes_{\Sigma_q\ltimes \base^{\otimes q}} \mathscr{P}(0)^{\otimes q}\Big)^\vee[-q]\otimes_{\Sigma_r\ltimes(\base^{\op})^{\otimes r}} (A^\vee)^{\otimes r}
$$
and we have that 
$$
\mf{D}_\pi(A)\cong \prod_{q, r\geq 0} N(q, r) \qquad \text{and} \qquad \mf{D}^\mm{poly}_\pi(A) = \bigoplus_r \prod_{q\geq 0} N(q, r).
$$
Since $\base$ is concentrated in degrees $[-N, 0]$ and both $A$ and $\mathscr{P}$ are concentrated in nonpositive degrees, we have that $N(q, r)$ is concentrated in degrees $\geq q-N$, for all values of $r$. Consequently, the natural map
$$\begin{tikzcd}
\mf{D}^\mm{poly}_{\pi}(A) = \bigoplus_r \prod_{q\geq 0} N(q, r)\arrow[r] & \prod_{q\geq 0} \bigoplus_r N(q, r)
\end{tikzcd}$$
is an isomorphism in each cohomological degree. Now note that $\mf{D}^\mm{poly}_{\pi}(A)\cong \prod_{q\geq 0}\bigoplus_r N(q, r)$ is manifestly complete Hausdorff with respect to the filtration by $q$. Furthermore, we see that, without differential, there is an inclusion
$$
\mm{gr}\Big(\mf{D}^\mm{poly}_\pi\Big)\subseteq \mm{gr}\Big(\mf{D}_\pi\Big)
$$
given in degree $q$ by the obvious inclusion $\bigoplus_r N(q, r)\rt \prod_r N(q, r)$. Since $\mm{gr}\big(\mf{D}_\pi(A)\big)\cong \mf{D}_{\pi^{\geq 1}}\big(A\oplus \mathscr{P}(0)[1]\big)$, the second part of the lemma then follows by unravelling the definitions.
\end{proof}
\begin{proof}[Proof (of Proposition \ref{prop:filtrationissues})]
Suppose that $\base$ is concentrated in degrees $[-N, 0]$, that $\mathscr{P}$ is concentrated in nonpositive degrees and that $A$ is nilpotent and strictly Artin. In particular, $A$ is quasi-free and finitely generated over $\base$ (Remark \ref{rem:smallisperfect}). To see that $\mf{D}^\mm{poly}_\pi(A)\rt \mf{D}_\pi(A)$ is a quasi-isomorphism, we can work at the level of the underlying $\base$-modules.

Endow both $\mf{D}^\mm{poly}_\pi(A)$ and $\mf{D}_\pi(A)$ with the filtration by the number of nullary operations from $\PP$, as in Lemma \ref{lem:filtrationnullary}. Since these filtrations are complete and Hausdorff, it suffices to show that the induced map on the associated graded
$$\begin{tikzcd}
\mf{D}_{\pi^{\geq 1}}^{\mm{poly}}\big(A\oplus \mathscr{P}(0)[1]\big)\arrow[r] & \mf{D}_{\pi^{\geq 1}}\big(A\oplus \mathscr{P}(0)[1]\big)
\end{tikzcd}$$
is a quasi-isomorphism, where $\pi^{\geq 1}\colon B(\mathscr{P}^{\geq 1})\drt \mathscr{P}^{\geq 1}$ is the universal twisting morphism. Note that the $\mathscr{P}^{\geq 1}$-algebra $A\oplus \mathscr{P}(0)[1]$ satisfies the conditions of Proposition \ref{prop:filtrationissues0-reduced}: it is quasi-free over $\base$ and it has finitely many generators in degree $0$, all coming from $A$. The result then follows from Corollary \ref{cor:polyhoinv} and Proposition \ref{prop:filtrationissues0-reduced}.
\end{proof}

\section{Change of operads}\label{sec:naturality}
In this section we describe the functoriality of the equivalence
$$\begin{tikzcd}
\mm{MC}\colon\Alg_{\mf{D}(\PP)}\arrow[r] & \FMP_{\PP}
\end{tikzcd}$$
in the operad $\PP$ and use it to give a modular interpretation of the category of $\QQ$-algebras for \emph{any} $k$-operad $\QQ$ (Theorem \ref{thm:meta2}). We will start by considering the functoriality of the adjoint pair $(\mf{D}_\phi, \mf{D}'_\phi)$ in the twisting morphism $\phi$.

\subsection{Naturality of weak Koszul duality}\label{sec:naturalityofkd}
To study the dependence of the adjoint pair $\big(\mf{D}_\phi, \mf{D}'_\phi\big)$ on the twisting morphism $\phi\colon \cat{C}\drt \mathscr{P}$, let us consider the following category of twisting morphisms:
\begin{definition}\label{def:koszulcat}
Let $\Koszul$ denote the category whose
\begin{itemize}
\item objects are weakly Koszul twisting morphisms $\phi\colon \mathscr{C}\drt \mathscr{P}$ from a $\base$-cooperad $\mathscr{C}$ to a $\base$-operad $\mathscr{P}$. When considered as a left $\base$-module, $\CC$ is filtered-cofibrant and $\PP$ is cofibrant by assumption.
\item morphisms consist of a $\base$-operad map $f\colon \mathscr{P}\rt \mathscr{Q}$ and a $\base$-cooperad map $g\colon \mathscr{C}\rt \mathscr{D}$, fitting into a commuting square
\begin{equation}\label{diag:koszulintertwiner}\begin{tikzcd}
\mathscr{C}\arrow[d, "g"{swap}]\arrow[r, dashrightarrow, "\phi"] & \mathscr{P}\arrow[d, "f"]\\
\DD\arrow[r, dashrightarrow, "\psi"{swap}"] & \QQ.
\end{tikzcd}\end{equation}
\end{itemize}
A map between such twisting morphisms is a weak equivalence if $f$ (and hence also $g$) is a quasi-isomorphism.
\end{definition}
\begin{remark}\label{rem:koszul=operads}
There is an obvious projection map $\pi\colon \Koszul\rt \cat{Op}^{\dg}_{\base}$ sending a twisting morphism to its codomain. This projection admits a section $\sigma\colon \cat{Op}^{\dg}_{\base}\rt \Koszul$ sending $\PP$ to the universal twisting morphism $\Bar\PP\drt \PP$. In addition to the isomorphism $\pi\sigma\cong \mm{id}$, there is a natural weak equivalence $\mm{id}\rt \sigma\pi$: every weakly Koszul twisting morphism $\phi\colon \CC\drt\PP$ admits a natural weak equivalence to the universal one $\Bar\PP\drt \PP$. It follows that $\pi$ and $\sigma$ induce an equivalence of $\infty$-categories after inverting the weak equivalences.
\end{remark}
Consider the following functors with values in the $\infty$-category of $\infty$-categories and left adjoint functors between them
\begin{equation}\label{diag:algalgdual}\begin{tikzcd}
\Alg\colon \Koszul\ar[r] & \cat{Cat^L_\infty} & & \Alg^\mm{dual}\colon \Koszul\ar[r] & \cat{Cat^L_\infty}.
\end{tikzcd}\end{equation}
These two functors send a map \eqref{diag:koszulintertwiner} to the left adjoint functors
$$\begin{tikzcd}
f_!\colon \Alg_{\PP}\arrow[r] & \Alg_{\QQ} & (g^\vee)^*\colon \Alg_{\CC^\vee}^{\op}\arrow[r] & \Alg_{\DD^\vee}^{\op}.
\end{tikzcd}$$
We then have the following homotopy coherent upgrade of Lemma \ref{lem:naturality}:
\begin{proposition}\label{prop:naturality}
There is a natural transformation of functors
\begin{equation}\label{diag:naturalD}\begin{tikzcd}[column sep=2.3pc]
\Koszul\arrow[r, bend left, "\Alg"{above}]\arrow[r, phantom,  ""{name=s}, ""{name=t, below}]\arrow[r, bend right, "\Alg^\mm{dual}"{below}] & \cat{Cat^L_\infty}\arrow[Rightarrow, from=s, to=t, start anchor={[yshift=2ex]}, end anchor={[yshift=-2ex]}, "\mf{D}"]
\end{tikzcd}\end{equation}
whose value at a weakly Koszul twisting morphism $\phi\colon\CC\drt\PP$ is given by
$$\begin{tikzcd}
\cat{Alg}_{\mathscr{P}}\arrow[r, "\mf{D}_\phi"] &\cat{Alg}_{\CC^\vee}^\op.
\end{tikzcd}$$
\end{proposition}
Note that the functors $\Alg$ and $\Alg^\mm{dual}$ send weak equivalences between twisting morphisms to equivalences of $\infty$-categories (Corollary \ref{cor:quasiiso gives equivalence}), and hence descend to functors on the $\infty$-categorical localizations. By Remark \ref{rem:koszul=operads}, we therefore obtain the following:
\begin{corollary}\label{cor:naturality}
Let $\cat{Op}_{\base}$ be the $\infty$-category of (augmented) $\base$-operads. Then there is a natural transformation of functors
$$\begin{tikzcd}[column sep=2.3pc]
\cat{Op}_{\base}\arrow[r, bend left, "\Alg"{above}]\arrow[r, phantom,  ""{name=s}, ""{name=t, below}]\arrow[r, bend right, "\Alg^\mm{dual}"{below}] & \cat{Cat^L_\infty}\arrow[Rightarrow, from=s, to=t, start anchor={[yshift=2ex]}, end anchor={[yshift=-2ex]}, "\mf{D}"]
\end{tikzcd}$$
given on objects by 
$\begin{tikzcd}
\cat{Alg}(\mathscr{P})\coloneqq \cat{Alg}_{\mathscr{P}}\arrow[r, "\mf{D}"] &\cat{Alg}_{\mf{D}(\mathscr{P})}^\op\eqqcolon\cat{Alg}^\mm{dual}(\mathscr{P}).
\end{tikzcd}$
\end{corollary}
Recall from Lemma \ref{lem:naturality} that a single map of twisting morphisms induces a square of $\infty$-categories commuting up to natural equivalence. For this reason, it will be more convenient to establish Proposition \ref{prop:naturality} in terms of fibrations.
\begin{construction}\label{con:algebrasandoperads}
Let $\cat{Alg}^{\dg}$ denote the category whose
\begin{itemize}\setlength{\itemsep}{3pt}
\item objects are tuples $(\phi \colon \CC\drt\mathscr{P}, A)$ consisting of a weakly Koszul twisting morphism, together with a cofibrant $\PP$-algebra $A$.
\item morphisms $(\phi \colon \CC\drt\mathscr{P}, A)\rt (\psi \colon \DD\drt\QQ, B)$ consist of a map \eqref{diag:koszulintertwiner} and a map of $\PP$-algebras $A\rt f^*B$.
\end{itemize}
Similarly, let $\cat{Alg}^{\mm{dual}, \dg}$ denote the category whose
\begin{itemize}\setlength{\itemsep}{3pt}
\item objects are tuples $(\phi \colon \CC\drt\mathscr{P}, \mf{g})$ consisting of a weakly Koszul twisting morphism, together with a $\CC^\vee$-algebra $\mf{g}$.
\item morphisms $(\phi \colon \CC\drt\mathscr{P}, \mf{g})\rt (\psi \colon \DD\drt\QQ, \mf{h})$ consist of a map \eqref{diag:koszulintertwiner} and a map of $\PP$-algebras $(g^\vee)^*(\mf{g})\rt\mf{h}$.
\end{itemize}
There are obvious projections
$$\begin{tikzcd}
\cat{Alg}^{\dg}\arrow[r] & \Koszul & \cat{Alg}^{\mm{dual}, \dg}\arrow[r] & \Koszul,
\end{tikzcd}$$
whose fibers over $\phi\colon\CC\drt\PP$ are given by the categories of cofibrant $\PP$-algebras and of $\CC^\vee$-algebras, respectively. We will say that a map in $\cat{Alg}^{\dg}$ is a \emph{fiberwise weak equivalence} if it is a quasi-isomorphism of algebras that covers the identity in the base category $\Koszul$.

Note that both projections are cocartesian fibrations: given a map \eqref{diag:koszulintertwiner} in the base category $\Koszul$, the induced functors between the fibers are given by $f_!$ and $(g^\vee)^*$. By  construction, these change-of-fiber functors preserve fiberwise weak equivalences, so that inverting the fiberwise weak equivalences yields cocartesian fibrations \cite[Proposition 2.1.4]{hinichlocalization}
$$\begin{tikzcd}
\cat{Alg}\arrow[r] & \Koszul & & \cat{Alg}^{\mm{dual}}\arrow[r] & \Koszul.
\end{tikzcd}$$
These are exactly the cocartesian fibrations classified by the functors $\Alg$ and $\Alg^{\mm{dual}}$ from \eqref{diag:algalgdual}. Since these functors take values in $\infty$-categories and left adjoint functors between them, the projections are cartesian fibrations as well \cite[Corollary 5.2.2.5]{HTT}.
\end{construction}
\begin{proof}[Proof (of Proposition \ref{prop:naturality})]
Consider the commuting triangle
$$\begin{tikzcd}
\Alg^{\dg}\arrow[rr, "\Bar^\vee"]\arrow[rd] & & \Alg^{\mm{dual}, \dg}\arrow[ld]\\
& \Koszul &
\end{tikzcd}$$
where the vertical functors are the projections and the top horizontal functor is given by
$$
\Big(\phi\colon \cat{C}\drt \mathscr{P}, A\Big)\longmapsto \Big(\phi\colon \cat{C}\drt \mathscr{P}, B_\phi(A)^\vee\Big).
$$
This functor sends fiberwise weak equivalences in $\Alg^{\dg}$ to fiberwise weak equivalences in $\Alg^{\mm{dual}, \dg}$ by Lemma \ref{lem:algbarpreswe}. Consequently, it descends to a functor between the $\infty$-categorical localizations, which we will denote by
\begin{equation}\label{diag:Dbetweenfibrations}\begin{tikzcd}
\Alg\arrow[rr, "\mf{D}"]\ar[rd] & & \Alg^{\mm{dual}}\ar[ld]\\
& \Koszul. &
\end{tikzcd}\end{equation}
When restricted to the fiber over a weakly Koszul twisting morphism $\phi$, this functor is given by $\mf{D}_\phi$ and admits a right adjoint $\mf{D}'_\phi$ by Lemma \ref{lem:adjoint}. Furthermore, the functor $\mf{D}\colon \cat{Alg}\rt \cat{Alg}^{\mm{dual}}$ preserves cocartesian edges. Indeed, unraveling the definitions, this is exactly the assertion of Lemma \ref{lem:naturality}. It follows from \cite[Proposition 7.3.2.6]{LurieHA} (and its dual) that the functor $\mf{D}$ has a right adjoint which commutes with the projections, preserves cartesian edges and is given fiberwise by $\mf{D}'_\phi$. Under straightening, this means precisely that $\mf{D}$ determines a natural transformation of the form \eqref{diag:naturalD}.
\end{proof}

\subsection{Naturality of the main theorem}
We will now use Corollary \ref{cor:naturality} to show that the equivalence between formal moduli problems and algebras of Theorem \ref{thm:mainthm} depends functorially on the operad:
\begin{proposition}\label{prop:naturalityfmp}
Let $\base$ be a bounded connective dg-category and let $\cat{Op}_{\base}^{+}$ denote the $\infty$-category of splendid connective $\base$-operads. There is a natural equivalence of functors
$$\begin{tikzcd}[column sep=3.1pc]
\cat{Op}_{\base}^+\arrow[r, bend left, "\Alg_{\mf{D}}"{above}]\arrow[r, phantom,  ""{name=s, pos=0.4}, ""{name=t, below, pos=0.4}]\arrow[r, bend right, "\FMP"{below}] & \cat{Pr^R}\arrow[Rightarrow, from=s, to=t, start anchor={[yshift=1.8ex]}, end anchor={[yshift=-1.8ex]}, "\MC"]
\end{tikzcd}$$
with values in the $\infty$-category of locally presentable $\infty$-categories and right adjoint functors. The value of this natural equivalence at a map $f\colon \mathscr{P}\rt \QQ$ is given by the commuting square of right adjoints
\begin{equation}\label{diag:naturalityright}
\begin{tikzcd}
\cat{Alg}_{\mf{D}(\PP)}\arrow[r, "\mm{MC}", "\simeq"{swap}]\arrow[d, "\mf{D}(f)^*"{swap}] & \cat{FMP}_{\PP}\arrow[d, "(f^*)^*"]\\
\cat{Alg}_{\mf{D}(\QQ)}\arrow[r, "\simeq"{swap}, "\mm{MC}"] & \cat{FMP}_{\QQ}.
\end{tikzcd}\end{equation}
where the right vertical functor restricts a formal moduli problem along the forgetful functor $f^*\colon \cat{Art}_{\mathscr{Q}}\rt \cat{Art}_{\mathscr{P}}$.
\end{proposition}
\begin{lemma}\label{lem:naturalitysmall}
Let $\base$ be a bounded connective dg-category and let $\cat{Op}_{\base}^{+}$ denote the $\infty$-category of splendid connective $\base$-operads. Then there is a natural transformation of functors
\begin{equation}\label{diag:naturalitysmall}
\begin{tikzcd}[column sep=2.8pc]
\big(\cat{Op}_{\base}^{+}\big)^{\op}\arrow[r, bend left, "\Art"{above}]\arrow[r, phantom,  ""{name=s, pos=0.4}, ""{name=t, below, pos=0.4}]\arrow[r, bend right, "\Alg^{\mm{dual}}"{below}] & \cat{Cat}_\infty\arrow[Rightarrow, from=s, to=t, start anchor={[yshift=1.8ex]}, end anchor={[yshift=-1.8ex]}, "\mf{D}"]
\end{tikzcd}\end{equation}
whose value on a map $f\colon \PP\rt \QQ$ is given by
$$\begin{tikzcd}
\Art_{\QQ}\arrow[d, "f^*"{swap}]\arrow[r, "\mf{D}"] & \cat{Alg}_{\mf{D}(\QQ)}^\op\arrow[d, "\mf{D}(f)_!"]\\
\Art_{\PP}\arrow[r, "\mf{D}"{swap}] & \cat{Alg}_{\mf{D}(\PP)}^\op.
\end{tikzcd}$$
\end{lemma}
\begin{proof}
Recall the commuting triangle \eqref{diag:Dbetweenfibrations}, where the vertical projections are cartesian and cocartesian fibrations and the top horizontal functor $\mf{D}$ sends $(\phi\colon \CC\drt\PP, A)$ to $(\phi\colon \CC\drt\PP, \mf{D}_\phi(A))$. Let us consider the following subcategories of the $\infty$-categories appearing in that triangle:
\begin{itemize}
\item Let $\Koszul^+\subseteq \Koszul$ denote the subcategory of universal twisting morphisms $\Bar\PP\drt \PP$ where $\PP$ is a splendid connective $\base$-operad, with maps between those given by tuples $f\colon \PP\rt \QQ$ and $\Bar(f)\colon \Bar\PP\rt \Bar\QQ$. By Remark \ref{rem:koszul=operads}, inverting the weak equivalences in $\Koszul^+$ yields the $\infty$-category $\cat{Op}_\base^+$.

\item Let $\Alg^{+, \mm{dual}}\eqqcolon \Alg^\mm{dual}\times_{\Koszul} \Koszul^+$ be the restriction of $\Alg^\mm{dual}$ to the category $\Koszul^+$.

\item Let $\Alg^{+, \mm{art}}\subseteq \Alg\times_{\Koszul} \Koszul^+$ be the full subcategory of tuples $(\phi\colon \Bar\PP\drt \PP, A)$ with $A$ an Artin $\PP$-algebra.
\end{itemize}
The functors appearing in \eqref{diag:Dbetweenfibrations} then restrict to
$$\begin{tikzcd}
\cat{Alg}^{+, \mm{art}}\arrow[rr, "\mf{D}"] \arrow[rd] & & \cat{Alg}^{+, \mm{dual}}\arrow[ld]\\
 & \Koszul^+.
\end{tikzcd}$$
Note that the projection $\cat{Alg}^{+, \mm{dual}}\rt \Koszul^+$ is (the restriction of) a cocartesian and cartesian fibration. Since the restriction of an Artin algebra along a map of operads $\mathscr{P}\rt \QQ$ is again Artin, the projection $\Alg^{+, \mm{art}}\rt \Koszul^+$ is a cartesian fibration as well. Recall that $\mf{D}$ sends a tuple $\big(\phi\colon \CC\drt\PP, A\big)$ to $\big(\phi\colon \CC\drt\PP, \mf{D}_\phi(A)\big)$. In particular, Proposition \ref{prop:functorialitysmall} shows that it preserves cartesian edges, so that we obtain a natural transformation
$$\begin{tikzcd}[column sep=2.8pc]
\Koszul^{+, \op}\arrow[r, bend left=25, "\Art"{above}]\arrow[r, phantom,  ""{name=s, pos=0.4}, ""{name=t, below, pos=0.4}]\arrow[r, bend right=25, "\Alg^{\mm{dual}}"{below}] & \cat{Cat}_\infty.\arrow[Rightarrow, from=s, to=t, start anchor={[yshift=1.8ex]}, end anchor={[yshift=-1.8ex]}, "\mf{D}"]
\end{tikzcd}$$
The domain of this natural transformation is given by $\PP\longmapsto \Art_{\PP}$ and the codomain is given by $\PP\longmapsto \Alg^{\op}_{\mf{D}(\PP)}$. Since these functors send quasi-isomorphisms to equivalences of $\infty$-categories, the natural transformation descends to the localization of $\Koszul^+$ at the quasi-isomorphisms, yielding the desired natural transformation \eqref{diag:naturalitysmall}.
\end{proof}
\begin{proof}
Consider the natural transformation \eqref{diag:naturalitysmall} from Lemma \ref{lem:naturalitysmall}. Taking the opposites of its values, one obtains a functor with values in $\infty$-categories sending a map $f\colon \PP\rt \QQ$ to
$$\begin{tikzcd}
\Art_{\QQ}^{\op}\arrow[d, "f^*"{swap}]\arrow[r, "\mf{D}"] & \cat{Alg}_{\mf{D}(\QQ)}\arrow[d, "\mf{D}(f)_!"]\\
\Art_{\PP}^{\op}\arrow[r, "\mf{D}"{swap}] &\cat{Alg}_{\mf{D}(\PP)}.
\end{tikzcd}$$
By the universal property of presheaf categories \cite[Theorem 5.1.5.6]{HTT}, one obtains a natural transformation of functors $\cat{Op}_\base^+\rt \cat{Pr^R}$ with values in presentable $\infty$-categories and right adjoint functors, whose value on $f$ is given by
$$\begin{tikzcd}
\cat{Alg}_{\mf{D}(\PP)}\arrow[r, hookrightarrow, "\mf{D}^*"]\arrow[d, "\mf{D}(f)^*"{swap}] & \Fun\big(\Art_{\PP}, \sS\big)\arrow[d, "(f^*)^*"]\\
\cat{Alg}_{\mf{D}(\QQ)}\arrow[r, hookrightarrow, "\mf{D}^*"{swap}] & \Fun\big(\Art_{\QQ}, \sS\big).
\end{tikzcd}$$
Here the functor $(f^*)^*$ restricts a (co)presheaf along $f^*$ and $\mf{D}^*$ sends a $\mf{D}(\PP)$-algebra $\mf{g}$ to the functor $A\longmapsto \Map_{\mf{D}(\PP)}\big(\mf{D}(A), \mf{g}\big)$. By part \ref{it:checkC} of Theorem \ref{thm:biduals}, the natural transformation $\mf{D}^*$ takes values in diagram of full subcategories $\mathscr{P}\longmapsto \cat{FMP}_{\mathscr{P}}$. The result then follows from the fact that
$$\begin{tikzcd}
\MC\colon \Alg_{\mf{D}(\PP)}\arrow[r] & \FMP_{\PP}\arrow[r, hookrightarrow] & \Fun\big(\Art_{\PP}, \sS\big)
\end{tikzcd}$$
agrees with $\mf{D}^*$ by definition.
\end{proof}

\subsection{FMPs from algebras over operads}\label{sec:fmpsoverfmps}
Since every Artin (augmented symmetric) $k$-operad $\mathscr{Q}$ is in particular concentrated in finitely many arities, it is splendid. Proposition \ref{prop:naturalityfmp} and passing to opposite categories therefore shows that there is a triangle
$$\begin{tikzcd}[row sep=2.2pc, column sep=4pc]
\ArtOp^{\op}\arrow[r, "\mf{D}"]\ar[rd, "\FMP"{swap}, ""{name=t}] & \cat{Op}\arrow[d, "\cat{Alg}"]\arrow[Rightarrow, to=t, shorten=1ex, "\MC"]\\
& \cat{Pr^L}
\end{tikzcd}$$
commuting up to the natural equivalence $\mm{MC}\colon \cat{Alg}_{\mf{D}(\mathscr{Q})}\rt \cat{FMP}_{\mathscr{Q}}$. Since $\cat{Pr^L}$ admits all colimits, this induces a commuting diagram
$$\begin{tikzcd}[row sep=2.2pc, column sep=4pc]
\Fun(\ArtOp, \sS)\arrow[r, "\mf{D}_!"]\ar[rd, "\FMP_!"{swap}, ""{name=t}] & \cat{Op}\arrow[d, "\cat{Alg}"]\arrow[Rightarrow, to=t, shorten=1ex, "\MC"]\\
& \cat{Pr^L}
\end{tikzcd}$$
where $\mf{D}_!$ and $\cat{FMP}_!$ are the unique colimit-preserving extensions of $\mf{D}$ and $\cat{FMP}$. The right adjoint to $\mf{D}_!$ sends an operad $\PP$ to the functor $\Map_{\cat{Op}}(\mf{D}(-), \PP)$. In other words, this right adjoint is precisely the functor $\MC$, composed with the inclusion $\cat{FMP}_{\cat{Op}}\hookrightarrow \Fun(\ArtOp, \sS)$. This implies that for a formal moduli problem $F$, 
$$
\mf{D}_!(F)\simeq \MC^{-1}(F)=T(F)
$$
is given by the tangent complex (by Theorem \ref{thm:axiomatic} the inverse of $\MC$ is the tangent complex). 
\begin{definition}\label{def:fmpoverfmp}
Let $X\colon \ArtOp\rt \sS$ be a functor. We define $\cat{FMP}_X$ to be the value of $\cat{FMP}_!$ on $X$. One can identify $\cat{FMP}_X$ with the limit of the diagram
$$\begin{tikzcd}
\big(\ArtOp^{\op}/X\big)^{\op}\arrow[r] & \cat{Cat}_\infty
\end{tikzcd}$$
sending
$$\begin{tikzcd}[row sep=1.6pc]
\big(\mathscr{Q}, x\in X(\mathscr{Q})\big)\arrow[d, "\substack{f\colon \mathscr{Q}\rt \mathscr{Q}'\\ f(x)\rto{\sim} x'}"{swap}]\arrow[r, mapsto] & \cat{FMP}_{\mathscr{Q}}\arrow[d, "F\mapsto F\circ f^*"]\\
\big(\mathscr{Q}', x'\in X(\mathscr{Q}')\big)\arrow[r, mapsto] & \cat{FMP}_{\mathscr{Q}'}.
\end{tikzcd}$$
\end{definition}
\begin{theorem}[Theorem \ref{thm:metaduality}]\label{thm:meta2}
For any formal moduli problem $X\colon \ArtOp\rt \sS$, there is an equivalence of $\infty$-categories
$$\begin{tikzcd}
\mm{MC}\colon \cat{Alg}_{T(X)}\arrow[r] & \cat{FMP}_X.
\end{tikzcd}$$
\end{theorem}
\begin{lemma}\label{lem:algsift}
The functor $\cat{Alg}\colon \cat{Op}\rt \cat{Pr^L}; \mathscr{P}\longmapsto \cat{Alg}_{\mathscr{P}}$ preserves sifted colimits.
\end{lemma}
\begin{proof}
The functor $\cat{Alg}$ is classified by a cartesian and cocartesian fibration $\cat{Alg}\rt \cat{Op}$; in fact, this is just the functor obtained by localizing the functor $\cat{Alg}^{\dg}\rt \cat{Op}^{\dg}$ of Construction \ref{con:algebrasandoperads}. Note that $\cat{Alg}$ is itself the category of algebras over a coloured operad (namely, the operad for operads with an algebra over them), and hence admits all limits and colimits.

We claim that the lemma follows from the following assertion: consider a cone diagram $F\colon\cat{K}^{\rhd}\rt \cat{Alg}$ such that
\begin{itemize}
\item the full subcategory $\cat{K}\subseteq \cat{K}^{\rhd}$ is a sifted $\infty$-category.
\item the composite $\cat{K}^{\rhd}\rt \cat{Alg}\rt \cat{Op}$ is a colimit diagram of operads.
\item for each arrow in the subcategory $\cat{K}\subseteq \cat{K}^{\rhd}$, its image in $\cat{Alg}$ is a cartesian arrow.
\end{itemize}
Then the diagram $F\colon \cat{K}^{\rhd}\rt \cat{Alg}$ is a colimit diagram if and only if for every arrow in $\cat{K}^{\rhd}$, its image in $\cat{Alg}$ is cartesian. 

Indeed, let $K^{\rhd}\rt \cat{Op}$ be a colimit diagram and consider the functor
$$\begin{tikzcd}
\Map^\flat_{\cat{K}^\rhd}\big((\cat{K}^{\rhd})^{\sharp}, \cat{Alg}^\natural\big)\arrow[r] & \Map^\flat_{\cat{K}}\big(\cat{K}^{\sharp}, \cat{Alg}^\natural\big)
\end{tikzcd}$$
which restricts a lift $\cat{K}^{\rhd}\rt \cat{Alg}$ with values in cartesian edges to the full subcategory $\cat{K}\subseteq \cat{K}^{\rhd}$. By the assertion, this functor is an equivalence with inverse given by left Kan extension (cf.\ \cite[Proposition 4.3.2.15]{HTT}) and the lemma then follows from \cite[Proposition 3.3.3.1]{HTT}.

To verify the assertion, note that the forgetful functor 
$$\begin{tikzcd}
\Phi=(\Phi^1, \Phi^2)\colon \cat{Alg}\arrow[r] & \cat{Op}\times\cat{Mod}_k; \hspace{4pt} \big(\mathscr{P}, A\in\cat{Alg}_{\mathscr{P}}\big)\arrow[r, mapsto] & \big(\mathscr{P}, A\big)
\end{tikzcd}$$
arises from restriction along a map of operads, and hence detects sifted colimits \cite[Corollary 3.2.3.2]{LurieHA}. Furthermore, an arrow in $\cat{Alg}$ is cartesian if and only if its image under $\Phi^2$ is essentially constant.

We now have a diagram $F\colon K^{\rhd}\rt \cat{Alg}$ such that $\Phi^1(F)$ is a colimit and diagram and $\Phi^2(F\big|\cat{K})$ is essentially constant. Then $F$ is a colimit diagram iff $\Phi^2(F)$ is a colimit diagram, which is equivalent to $\Phi^2(F)$ being essentially constant, i.e.\ $F$ sends every arrow to a cartesian arrow.
\end{proof}
\begin{proof}[Proof (of Theorem \ref{thm:meta2})]
By Lemma \ref{lem:algsift}, we have sifted colimit-preserving functors that send a diagram $X\colon \ArtOp\rt \sS$ to $\cat{FMP}_X$ and $\cat{Alg}_{\mf{D}_!(X)}$. Since $\mm{MC}$ defines a natural equivalence between them on corepresentables, the same is true for all $X\colon \ArtOp\rt \sS$ that can be written as sifted colimits of corepresentables. In particular, this holds when $X$ is an FMP \cite[Proposition 1.5.8]{DAGX}. The result then follows from the fact that $T(X)=\mf{D}_!(X)$ when $X$ is an FMP.
\end{proof}

\section{Maurer--Cartan equation}\label{sec:Dprime}

In the previous sections we have discussed how -- for a suitable augmented $\base$-operad $\PP$ -- every algebra $\mf{g}$ over the dual operad $\mf{D}(\PP)$ determines a formal moduli problem
$$\begin{tikzcd}
\MC_\mf{g}\colon \Art_{\mathscr{P}}\arrow[r] & \sS.
\end{tikzcd}$$
The formal moduli problem has been defined in $\infty$-categorical terms by the formula
$$
\MC_\mf{g}(A)=\Map_{\mf{D}(\PP)}\big(\mf{D}(A), \mf{g}\big).
$$
The purpose of this section is to give a more explicit chain-level description of this functor in terms of Maurer--Cartan elements of (nilpotent) $L_\infty$-algebras (see Theorem \ref{thm:MC functor}). In particular, in the classical case where $\PP=\Com$ and $\mf{g}$ is a Lie algebra, we recover the usual formula (see e.g.\ \cite{hinichformal})
$$
\MC_\mf{g}(A) = \MC\big(A\otimes \mf{g}\otimes \Omega_\bullet\big)
$$
describing the formal moduli problem classified by $\mf{g}$ in terms of simplicial sets of Maurer--Cartan elements (see Example \ref{ex:comyieldsMC}).

We will start by recalling some models for $\infty$-categories of algebras by simplicially enriched categories. Under certain finiteness conditions on the $\base$-operad $\PP$ (see Assumption \ref{ass:mc}), we can then present the $\infty$-functor $\mf{D}$ on Artin $\PP$-algebras by a simplicially enriched functor that sends a strictly Artin $\PP$-algebra to the \emph{cobar} construction of its linear dual. The results of Section \ref{sec:barcobaralgebras} then allow us to describe $\MC_\mf{g}$ in terms of Maurer--Cartan elements.

\subsection{Simplicial categories of algebras}
Recall that for any $\base$-operad $\PP$, the $\infty$-category of $\PP$-algebras is defined to be the $\infty$-category obtained from the model category of $\PP$-algebras by localizing at the quasi-isomorphisms. Such localizations can be modelled by simplicially enriched categories, using the simplicial localization of Dwyer and Kan, and often its mapping spaces can be computed using fibrant resolutions \cite{DwyerKanFunction}:
\begin{definition}\label{def:simpcatofalgs}
Given a $\base$-operad $\PP$, the \emph{naive simplicial category of $\PP$-algebras} $\scat{Alg}_{\PP}$ is the following simplicially enriched category:
\begin{itemize}
\item objects are $\PP$-algebras.
\item for two $\PP$-algebras $A$ and $B$, the simplicial set $\Map_{\PP}^{\Delta}(A, B)$ of maps between them has $n$-simplices given by maps of $\PP$-algebras
$$\begin{tikzcd}
A\arrow[r] & B\otimes\Omega_n
\end{tikzcd}$$
where $\Omega_n=\Omega[\Delta^n]$ denotes the cdga of differential forms on the $n$-simplex. Equivalently, these are maps of $\PP\otimes \Omega_n$-algebras $A\otimes\Omega_n\rt B\otimes\Omega_n$.
\end{itemize}
Furthermore, let $\scat{Alg}_{\PP}^\circ\subseteq \scat{Alg}_{\PP}$ denote the full simplicial subcategory on the cofibrant $\PP$-algebras.
\end{definition}
Let $\Alg_{\PP}^\dg$ denote the (ordinary) category of $\PP$-algebras and let $\Alg_{\PP}^{\dg, \circ}\subseteq \Alg_{\PP}^{\dg}$ denote the subcategory of cofibrant $\PP$-algebras. We then have a commuting square of simplicial categories
$$\begin{tikzcd}
\Alg_{\PP}^{\dg, \circ}\arrow[r, "i"]\arrow[d] & \Alg_{\PP}^{\dg}\arrow[d]\\
\scat{Alg}_{\PP}^\circ\arrow[r] & \scat{Alg}_\PP
\end{tikzcd}$$
where the vertical functors simply include the vertices of the mapping spaces.

After taking simplicial localizations at the quasi-isomorphisms, each of the above functors yields a weak equivalence of simplicial categories. Indeed, taking cofibrant replacements produces a functor $Q\colon \Alg_{\PP}^{\dg}\rt \Alg_{\PP}^{\dg, \circ}$ such that $Q\circ i$ and $i\circ Q$ are naturally quasi-isomorphic to the identity. It follows that $i$ induces a weak equivalence after simplicial localization at the quasi-isomorphisms, and similarly for the inclusion $\scat{Alg}_{\PP}^{\circ}\rt \scat{Alg}_{\PP}$. The right vertical functor induces an equivalence after localizations because for every $\PP$-algebra $A$, the simplicial presheaf
$$\begin{tikzcd}
\Map_{\PP}^{\Delta}(-, A)\colon \Alg_{\PP}^{\dg}\arrow[r] & \cat{sSet}
\end{tikzcd}$$
is representable by the simplicial diagram of $\PP$-algebras $A\otimes \Omega[\Delta^\bullet]$, all of which are quasi-isomorphic (see e.g. \cite{DwyerKanEquivalences} or \cite[Corollary 2.9]{nuiten2016localizing}). Now note that every quasi-isomorphism in $\scat{Alg}_{\PP}^\circ$ is already a homotopy equivalence \cite[Lemma 4.8.4]{hinich1997homological}. Consequently, $\scat{Alg}_{\PP}^\circ$ is weakly equivalent to its simplicial localization and we obtain the following:
\begin{lemma}\label{lem:simpcatofalgebras}
If $\PP$ is a $\base$-operad, then the $\infty$-category of $\PP$-algebras can be modelled by the simplicial category $\scat{Alg}_\PP^{\circ}$.
\end{lemma}

Recall that given a Koszul twisting morphism $\phi\colon\CC\drt\PP$, there is a more general notion of $\infty$-morphisms of $\PP$-algebras, given by maps between the respective bar constructions. This recovers the classical examples of $A_\infty$- or $L_\infty$-morphisms.
The $\infty$-categorical localization can also be described using $\infty$-morphisms:
\begin{definition}
Given a Koszul twisting morphism $\phi\colon \CC\drt\PP$ (Definition \ref{def:koszultwisting}), we define $\scat{Alg}_{\mathscr{P}}^{\infty}$ to be the following simplicially enriched category:
\begin{itemize}
\item the objects of $\scat{Alg}_{\PP}^\infty$ are $\PP$-algebras which are cofibrant as $\base$-modules.
\item for two such $\PP$-algebras $A$ and $B$, the simplicial set $\Map_{\PP}^{\infty}(A, B)$ of maps between them has $n$-simplices given by $\infty$-morphisms
$$\begin{tikzcd}
A\arrow[r, rightsquigarrow] & B\otimes \Omega_n.
\end{tikzcd}$$
Equivalently, an $n$-simplex is a map of $\CC\otimes\Omega_n$-coalgebras 
\begin{equation}\label{diag:coalgformofoomaorphism}\begin{tikzcd}
\Bar_\phi(A)\otimes \Omega_n\ar[r] & \Bar_\phi(B)\otimes \Omega_n.
\end{tikzcd}\end{equation}
\end{itemize}
\end{definition}
\begin{lemma}\label{lem:simpcatofalgebras2}
If $\phi\colon \CC\drt\PP$ is a Koszul morphism over $\base$, then the $\infty$-category of $\PP$-algebras can be modelled by the simplicial category $\scat{Alg}_\PP^{\infty}$.
\end{lemma}
\begin{proof}
Including the strict morphisms into the $\infty$-morphisms and sending an $\infty$-morphism $A\rightsquigarrow B$ to the strict morphism $\Omega_\phi\Bar_\phi(A)\rt \Omega_\phi\Bar_\phi(B)$ induces simplicially enriched functors
$$\begin{tikzcd}
j\colon \scat{Alg}_{\PP}^{\circ}\arrow[r] & \scat{Alg}_{\PP}^{\infty} & & \Omega_\phi\Bar_\phi\colon \scat{Alg}_{\PP}^{\infty}\arrow[r] & \scat{Alg}_{\PP}^\circ.
\end{tikzcd}$$
The natural homotopy equivalences $\Omega_\phi\Bar_\phi(A)\rt A$ (Lemma \ref{lem:algbarresol}) and $A\rightsquigarrow \Omega_\phi\Bar_\phi(A)$ show that $j$ and $\Omega$ define a homotopy equivalence of simplicial categories.
\end{proof}

\subsection{Simplicial categories of Artin algebras}
We will now specialize to the case where $\PP=\Omega\CC$ arises as the cobar construction of a $\base$-cooperad satisfying suitable finiteness hypotheses:
\begin{assumption}\label{ass:mc}
For the remainder of this section, let us fix a dg-category $\base$ and a $\base$-cooperad $\CC$ which is filtered-cofibrant as a left $\base$-module, and let $\PP=\Omega\CC$ denote its cobar construction. We will assume that $\base$ is concentrated in degrees $[0, N]$, for some $N$, and that $\CC$ satisfies the following conditions:
\begin{enumerate}
\item\label{it:mc1} for all colours $c_1, \dots, c_p\in S$, the left $\base$-module $\CC(c_1, \dots, c_p; -)$ is quasi-projective and finitely generated.
\item\label{it:mc2} each $\CC(p)$ is concentrated in degrees $\leq f(p)$, where $f(p)$ tends to $-\infty$ as the arity $p$ tends to $\infty$.
\end{enumerate}
\end{assumption}
\begin{example}
Let $\CC=\mm{coFree}(E, R)$ be a quadratic cooperad over a field $k$, where $E$ is finite dimensional and in cohomological degrees $\leq 1$. Then $\CC$ satisfies the conditions of Assumption \ref{ass:mc}. In particular, this applies to the quadratic dual cooperads of classical quadratic operads such as $\Com, \As, \Lie$ and $\Perm$, as well as $\mathscr{O}^\mm{sym}$ (Definition \ref{def:op of symops}).
\end{example}
Let us record the following immediate consequences of these assumptions:
\begin{remark}\label{rem:Casinvariant}
The $\base$-operad $\PP=\Omega\CC$ satisfies the conditions of Variant \ref{var:cobarC}. In particular, for degree reasons every Artin $\PP$-algebra $A$ is automatically nilpotent and the inclusion $\mf{D}^\mm{poly}_\phi(A)\subseteq \mf{D}_\phi(A)$ is the \emph{identity}. Theorem \ref{thm:axiomatic} then provides an equivalence of $\infty$-categories between formal moduli problems over $\PP$ and $\CC^\vee$-algebras.
\end{remark}
\begin{remark}\label{rem:Discobar}
There is a canonical map of $\base$-operads $\PP=\Omega\CC\rt \Bar(\CC^\vee)^\vee$. Because each $\CC(p)$ is finitely generated over $\base$, this map identifies $\Bar(\CC^\vee)^\vee$ with the completion $\PP^\wedge$ of $\PP$ at its augmentation ideal. 

Note that a strictly Artin (hence nilpotent) $\PP$-algebra $A$ has a canonical $\PP^\wedge$-algebra structure. Since such $A$ is perfect over $\base$, its linear dual $A^\vee$ has the canonical structure of a $\Bar(\CC^\vee)$-coalgebra. Unravelling the definitions, one then obtains a natural isomorphism of $\CC^\vee$-algebras
$$\begin{tikzcd}
\Omega_{\phi^\dagger}(A^\vee)\arrow[r, "\cong"] & \mf{D}_\phi^\mm{poly}(A)= \mf{D}_\phi(A)
\end{tikzcd}$$
where $\phi^\dagger\colon \Bar(\CC^\vee)\drt \CC^\vee$ is the canonical twisting morphism.
\end{remark}
\begin{definition}[Simplicial category of Artin $\mathscr{P}$-algebras]\label{def:simpcatofsmallalgs}
For $\PP=\Omega\CC$ as in Assumption \ref{ass:mc}, let us define
$$
\scat{Art}_\mathscr{P}^{\infty}\subseteq \scat{Alg}_{\PP}^{\infty}
$$
to be the full simplicial subcategory on the $\PP$-algebras that are strictly Artin (Definition \ref{def:verysmall}).
\end{definition}
\begin{lemma}\label{lem:simpsmall}
In the situation of Assumption \ref{ass:mc}, the $\infty$-category of Artin $\PP$-algebras can be presented by the simplicial category $\scat{Art}_\mathscr{P}^{\infty}$.
\end{lemma}
\begin{proof}
The simplicial category $\scat{Art}_\mathscr{P}^{\infty}$ presents a full subcategory of the $\infty$-category of $\PP$-algebras by Lemma \ref{lem:simpcatofalgebras2}. It presents the subcategory of Artin $\PP$-algebras by Lemma \ref{lem:verysmall}, which applies by Variant \ref{var:cobarC}.
\end{proof}
Assumption \ref{ass:mc} now allows us to give a very simple description of the functor $\mf{D}_\phi$ on the $\infty$-category of Artin $\PP$-algebras:
\begin{lemma}\label{lem:simpD}
In the situation of Assumption \ref{ass:mc}, there is a (strictly) fully faithful functor of simplicial categories
\begin{equation}\label{diag:simpD}\begin{tikzcd}
\mb{D}_\phi \colon \scat{Art}_\mathscr{P}^{\infty}\arrow[r] & \scat{Alg}^\circ_{\CC^\vee};\hspace{4pt} A\arrow[r, mapsto] & \mb{D}_\phi(A)\coloneqq\Omega_{\phi^\dagger}(A^\vee).
\end{tikzcd}\end{equation}
This simplicial functor presents the functor of $\infty$-categories $\mf{D}_\phi\colon \Art_{\PP}\rt \Alg_{\CC^\vee}$ from Section \ref{sec:weakkoszul}.
\end{lemma}
\begin{proof}
Let us start by defining the functor \eqref{diag:simpD} more precisely. By Remark \ref{rem:Discobar}, we have that $\mf{D}_\phi(A)=\mf{D}_\phi^\mm{poly}(A)\cong \Omega_{\phi^\dagger}(A^\vee)$ is cofibrant whenever $A$ is strictly Artin (and hence nilpotent, cf.\ Remark \ref{rem:Casinvariant}). Let us now define the functor $\mb{D}_\phi$ on simplicial sets of morphisms
$$\begin{tikzcd}
\mb{D}_\phi\colon \Map^\infty_{\PP}(A, B)\arrow[r] & \Map^\Delta_{\CC^\vee}\big(\mb{D}_\phi(B), \mb{D}_\phi(A)\big).
\end{tikzcd}$$ 
To this end, recall that an $n$-simplex in $\scat{Art}_\mathscr{P}^{\infty}$ is given by a map of $\CC\otimes\Omega_n$-coalgebras
\begin{equation}\label{diag:coalgformofoomorphisms}\begin{tikzcd}
\Bar_\phi(A)\otimes\Omega_n \arrow[r] & \Bar_\phi(B)\otimes \Omega_n.
\end{tikzcd}\end{equation}
Because $\CC$ satisfies the conditions \ref{it:mc1} and \ref{it:mc2} of Assumption \ref{ass:mc} and because $A$ is perfect over $\base$, we have that $\Bar_\phi(A)\otimes\Omega_n=\CC(A)\otimes\Omega_n$ is quasi-projective and finitely generated as a left $\base\otimes\Omega_n$-module. The $\base\otimes \Omega_n$-linear dual of $\Bar_\phi(A)\otimes\Omega_n$ is then given by
$$
\Hom_{\base\otimes \Omega_n}\big(\Bar_\phi(A), \base\otimes\Omega_n\big) \cong \mb{D}_\phi(A)\otimes \Omega_n.
$$
On (higher) morphisms, we can therefore simply define $\mb{D}_\phi$ to send \eqref{diag:coalgformofoomorphisms} to its $\base\otimes \Omega_n$-linear dual. The resulting map of simplicial sets is an isomorphism, with inverse taking the $\base^{\op}\otimes\Omega_n$-linear dual of a map of $\CC^\vee\otimes\Omega_n$-algebras. We therefore obtain the desired fully faithful functor \eqref{diag:simpD}.

To see that this enriched functor indeed presents the $\infty$-functor $\mf{D}_\phi$ defined in Section \ref{sec:weakkoszul}, consider the following commuting diagram:
$$\begin{tikzcd}
\scat{Art}_\mathscr{P}^{\infty}\arrow[d, "\mb{D}_\phi"{swap}]\arrow[r] & \scat{Alg}_{\PP}^\infty & \Alg_{\PP}^{\dg, \circ}\arrow[l, "\sim"{above}]\arrow[d, "\mf{D}_\phi"]\\
\scat{Alg}_{\CC^\vee}^{\circ, \op}\arrow[r, "\sim"] & \scat{Alg}_{\CC^\vee}^{\op} & \Alg_{\PP}^{\dg}.\arrow[l, "\sim"{above}]
\end{tikzcd}$$
By Lemma \ref{lem:simpsmall}, $\scat{Art}_\mathscr{P}^{\infty}\rt \scat{Alg}^{\infty}_{\PP}$ models the inclusion of the Artin $\PP$-algebras in the $\infty$-category of all $\PP$-algebras. The assertion then follows by noting that the marked arrows become equivalences after localizing at the quasi-isomorphisms. 
\end{proof}

\subsection{Formal moduli problems from the Maurer--Cartan equation}
We will now describe the equivalence provided by Theorem \ref{thm:axiomatic}
$$\begin{tikzcd}
\MC\colon \Alg_{\CC^\vee}\arrow[r, "\sim"] & \FMP_{\PP}
\end{tikzcd}$$
more concretely in terms of simplicial sets of Maurer--Cartan elements, at least for 1-reduced cooperads $\CC$ satisfying the finiteness hypotheses of Assumption \ref{ass:mc}. Let us start with the following observation:
\begin{lemma}\label{lem:tensorislie}
Fix a dg-category $\base$ and a \emph{1-reduced} $\base$-cooperad $\CC$ as in Assumption \ref{ass:mc}. For any twisting morphism $\phi\colon \CC\drt\PP$, there exists a functor
$$\begin{tikzcd}
\Alg^{\dg}_{\CC^\vee}\otimes\Alg^{\dg}_{\PP}\arrow[r] & \Alg_{L_\infty\{-1\}}^{\dg}; \hspace{4pt} (\mf{g}, A)\arrow[r, mapsto] & \mf{g}\otimes_\base A
\end{tikzcd}$$
to the category of shifted $L_\infty$-algebras. When $A$ or $\mf{g}$ is nilpotent, $\mf{g}\otimes_\base A$ is a nilpotent shifted $L_\infty$-algebra.
\end{lemma}
\begin{construction}[Hadamard tensor product of $\base$-operads]\label{con:hadamard}
Let $\base$ be a dg-category with a set of objects $S$. Given a $\base^{\op}$-operad $\PP$ and a $\base$-operad $\QQ$, we can construct a (monochromatic) operad $\PP\otimes_{\mm{H}} \QQ$ over the base field $k$, their (internal) \emph{Hadamard tensor product}\footnote{This construction differs from the (simpler) \emph{exterior} Hadamard tensor product \eqref{eq:exterior hadamard}.}, as follows. For two tuples of objects $\ul{c}=(c_1, \dots, c_p)$ and $\ul{d}=(d_1, \dots, d_p)$ in $\base$, consider the tensor product
$$
\PP(\ul{c})\otimes_{\base} \QQ(\ul{d}) \coloneqq \PP(\ul{c}; -)\otimes_{\base} \QQ(\ul{d}; -) = \Big(\bigoplus_{c_0} \PP(\ul{c}; c_0)\otimes \QQ(\ul{d}; c_0)\Big)\Big/\sim.
$$
Explicitly, the tensor product over $\base$ (see Section \ref{sec:conventions}) is computed as the quotient by relations 
$$
\big(\ul{c}\xrightarrow{\phi} c_0\xrightarrow{\ol{\lambda}} d_0\big)\otimes\big(\ul{d}\xrightarrow{\psi} d_0\big)\sim \big(\ul{c}\xrightarrow{\phi}c_0\big)\otimes\big(\ul{d}\xrightarrow{\psi} d_0\xrightarrow{\lambda} c_0\big)
$$
where $\phi$ and $\psi$ are operations in $\PP$ and $\QQ$, $\lambda$ is an arrow in $\base$ and $\ol{\lambda}$ denotes the corresponding arrow in $\base^{\op}$.
We then define
$$
\big(\PP\otimes_\mm{H} \QQ\big)(p)\coloneqq\Hom_{\base^{\otimes p}\otimes (\base^{\op})^{\otimes p}}\big(\base^{\otimes p}, \PP\otimes_\base \QQ\big)\subseteq \prod_{\ul{c}=(c_1, \dots, c_p)} \PP(\ul{c})\otimes_{\base} \QQ(\ul{c}).
$$
Explicitly, its elements are $S^{\times p}$-tuples of the form
\begin{equation}\label{eq:operationhadamard}
\Big(\sum_{c_0} \ul{c}\xrightarrow{\phi_{\ul{c}}} c_0\otimes \ul{c}\xrightarrow{\psi_{\ul{c}}} c_0\Big)_{\ul{c}\in S^{\times p}}
\end{equation}
such that for every tuple of maps $\lambda_i\colon d_i\rt c_i$ in $\base$
$$
\sum_{c_0}\big(\ul{c}\xrightarrow{\ol{\lambda}_1, \dots,\ol{\lambda}_p}\ul{d}\xrightarrow{\phi_{\ul{d}}} c_0\big)\otimes \ul{d}\xrightarrow{\psi_{\ul{d}}} c_0 =\sum_{c_0} \ul{c}\xrightarrow{\phi_{\ul{c}}} c_0\otimes \big(\ul{d}\xrightarrow{\lambda_1, \dots, \lambda_p}\ul{c}\xrightarrow{\psi_{\ul{c}}}c_0\big)
$$
in the complex $\PP(\ul{c})\otimes_\base \QQ(\ul{d})$. These equations guarantee that $\PP\otimes_{\mm{H}}\QQ$ carries a well-defined operad structure determined by
$$
\sum_{c_0}\big(\ul{c}\xrightarrow{\phi}c_0\otimes \ul{c}\xrightarrow{\psi} c_0\big)\circ_i \sum_{d_0}\big(\ul{d}\xrightarrow{\alpha}d_0\otimes \ul{d}\xrightarrow{\beta} d_0\big) = \sum_{c_0}\sum_{d_0=c_i}\xrightarrow{\phi\circ_i \alpha}\otimes \xrightarrow{\psi\circ_i \beta}.
$$
The operad $\PP\otimes_{\mm{H}} \QQ$ is constructed in order for the following to hold: if $A$ is a $\PP$-algebra and $B$ is a $\QQ$-algebra, then the chain complex $A\otimes_\base B$ is a $\PP\otimes_\mm{H}\QQ$-algebra. Indeed, given $p$ elements in $A\otimes_\base B$ of the form $\sum_{c_i} a_{c_i}\otimes b_{c_i}$ with $a_{c_i}\in A(c_i)$ and $b\in B(c_i)$, the operation \eqref{eq:operationhadamard} sends it to
$$
\sum_{c_0}\sum_{\ul{c}=(c_1, \dots, c_p)} \phi_{\ul{c}}(a_{c_1}, \dots, a_{c_p})\otimes \psi_{\ul{c}}(b_{c_1}, \dots, b_{c_p}).
$$
\end{construction}
\begin{proof}[Proof (of Lemma \ref{lem:tensorislie})]
Note that for any $\base$-cooperad $\CC$ and any $\base$-operad $\PP$, there is a natural inclusion of operads over the ground field $k$
$$\begin{tikzcd}
\CC^\vee\otimes_{\mm{H}}\PP\hspace{2pt}\arrow[r, hookrightarrow] & \Conv(\CC, \PP).
\end{tikzcd}$$
Here $\CC^\vee\otimes_{\mm{H}} \PP$ is the Hadamard tensor product (Construction \ref{con:hadamard}) and $\Conv(\CC, \PP)$ is the convolution operad (Remark \ref{rem:convolutionoperad}). Given an element in $\CC^\vee\otimes_{\mm{H}}\PP$ of the form \eqref{eq:operationhadamard}, with $\phi\in \CC^\vee$ and $\psi\in \PP$, the corresponding map $\CC(p)\rt \PP(p)$ is given by
$$
\CC\ni\big(\ul{c}\xrightarrow{\alpha} c\big)\hspace{3pt} \longmapsto \hspace{3pt} \sum_{c_0}\Big(\ul{c}\xrightarrow{\psi_{\ul{c}}} c_0\xrightarrow{\langle\phi_{\ul{c}}, \alpha\rangle} c\Big).
$$
The arrow $\big<\phi_{\ul{c}}, \alpha\big>$ is the natural value of $\phi_{\ul{c}}\in\CC^\vee$ on $\alpha\in\CC$, cf.\ Equation \eqref{eq:dual sequence}.

When $\CC$ is 1-reduced, the twisting morphism $\phi\colon \CC\drt\PP$ determines a map of operads $L_\infty\{-1\}\rt \Conv(\CC, \PP)$ (Remark \ref{rem:convolutionoperad}). When $\CC$ satisfies the finiteness conditions of Assumption \ref{ass:mc}, this maps factors as
$$\begin{tikzcd}
L_\infty\{-1\}\arrow[r] & \CC^\vee\otimes_\mm{H}\PP\hspace{2pt} \arrow[r, hookrightarrow] & \Conv(\CC, \PP).
\end{tikzcd}$$
Indeed, for every $\ul{c}=(c_1, \dots, c_p)$, Assumption \ref{ass:mc} allows us to pick a finite basis $e_{\ul{c},\alpha}\in\CC(\ul{c}; c_\alpha)$ for the left $\base$-module $\CC(\ul{c};-)$. Unravelling the definitions, one then sees that the generating $p$-ary operation of $L_\infty\{-1\}$ will be sent to the element 
$$
\big(\sum_\alpha e^\vee_{\ul{c},\alpha}\otimes \phi(e_{\ul{c}, \alpha})\big)_{\ul{c}=(c_1, \dots, c_p)} \in \CC^\vee\otimes_\mm{H}\PP.
$$
By Construction \ref{con:hadamard}, $\mf{g}\otimes_{\base} A$ is a $\CC^\vee\otimes_\mm{H}\PP$-algebra and hence an $L_\infty\{-1\}$-algebra by restriction. Furthermore, if $\mf{g}$ is nilpotent, then there are only finitely many composites of $e^\vee_{\ul{c}, \alpha}$ that act nontrivially on $\mf{g}$. Consequently, only finitely many composites of the generating operations in $L_{\infty}\{-1\}$ act nontrivially on $\mf{g}\otimes_\base A$. It follows that $\mf{g}\otimes_\base A$ is a nilpotent $L_\infty\{-1\}$-algebra, and similarly if $A$ is nilpotent.
\end{proof}
Using the shifted $L_\infty$-structure from Lemma \ref{lem:tensorislie}, we can now describe the formal moduli problem associated to a $\CC^\vee$-algebra more precisely as follows:
\begin{theorem}\label{thm:MC functor}
Consider a dg-category $\base$ and a 1-reduced $\base$-cooperad $\CC$ satisfying the conditions from Assumption \ref{ass:mc}, and denote $\PP=\Omega\CC$. For every $\CC^\vee$-algebra $\mf{g}$, there is a simplicially enriched functor
$$\begin{tikzcd}
\MC_\mf{g}\colon \scat{Art}_{\PP}^{\infty}\arrow[r] & \scat{sSet};\hspace{4pt} A\arrow[r, mapsto] & \MC\big(\mf{g}\otimes_\base A\otimes\Omega_\bullet\big)
\end{tikzcd}$$
where $\mf{g}\otimes_\base A$ carries the $L_\infty\{-1\}$-algebra structure from Lemma \ref{lem:tensorislie}. This determines a simplicially enriched functor
$$\begin{tikzcd}
\MC\colon \scat{Alg}_{\CC^\vee}^\Delta\arrow[r] & \Fun\big(\scat{Art}_\mathscr{P}^{\infty}, \scat{sSet}\big)
\end{tikzcd}$$
sending quasi-isomorphisms to pointwise homotopy equivalences. The associated functor between $\infty$-categories presents the fully faithful functor of Theorem \ref{thm:axiomatic}
\begin{equation}\label{diag:mcoofunctor}\begin{tikzcd}
\MC\colon \Alg_{\mf{D}(\PP)} \arrow[r, hookrightarrow] & \Fun\big(\Art_{\PP}, \sS\big),
\end{tikzcd}\end{equation}
whose essential image is the $\infty$-category of formal moduli problems over $\PP$.
\end{theorem}
\begin{example}\label{ex:comyieldsMC}
Let $\CC=\mm{Lie}^\vee\{1\}$ be the shifted coLie cooperad, so that $\Omega\CC=C_\infty$ is a resolution of the commutative operad. Suppose that $\mf{g}$ is a Lie algebra and that $A$ is a \emph{strict} unital Artin dg-algebra, i.e.\ an augmented unital cdga whose augmentation ideal $\mf{m}_A$ is (strictly) finite dimensional and nilpotent. Theorem \ref{thm:MC functor} then shows that
$$
\mm{MC}_\mf{g}(A)=\mm{MC}\big(\mf{m}_A\otimes \mf{g}\otimes\Omega_\bullet\big).
$$
In other words, the value on $A$ of the formal moduli problem associated to $\mf{g}$ by the equivalence of Lurie \cite[Theorem 2.0.2]{DAGX} coincides with the value of the deformation functor considered e.g.\ by Hinich \cite{hinichformal} and Pridham \cite{Pridham}.
In addition, Theorem \ref{thm:MC functor} shows that the \emph{full} FMP associated to $\mf{g}$ can be described similarly, by allowing $A$ to be a strictly Artin \emph{$C_\infty$-algebra}, in which case $\mf{m}_A\otimes\mf{g}$ is an $L_\infty$-algebra.
\end{example}
\begin{proof}
Consider the simplicially enriched functor $\scat{Alg}_{\CC^\vee}\rt \Fun(\scat{Art}_\mathscr{P}^{\infty}, \scat{sSet})$ sending $\mf{g}$ to the enriched functor
$$
A\longmapsto \Map^\Delta_{\CC^\vee}\big(\scat{D}_{\phi}(A), \mf{g}\big).
$$
By Lemma \ref{lem:simpD}, this enriched functor presents the functor of $\infty$-categories \eqref{diag:mcoofunctor} after inverting the weak equivalences. It therefore suffices to identify $\Map^\Delta_{\CC^\vee}\big(\scat{D}_{\phi}(A), \mf{g}\big)$ with a simplicial set of Maurer--Cartan elements. But now recall that
$$
\scat{D}_{\phi}(A)=\Omega_{\phi^\dagger}(A^\vee)
$$
is the cobar construction of the linear dual of $A$, which is a $B(\CC^\vee)$-coalgebra. By the universal property of the cobar construction (Proposition \ref{prop:Tw adjunction}), we then have that
$$
\Map^\Delta_{\CC^\vee}\big(\scat{D}_{\phi}(A), \mf{g}\big) \cong \MC\big(\Hom_{\base^{\op}}(A^\vee, \mf{g}\otimes\Omega_\bullet)\big)
$$
is given by the simplicial set of Maurer--Cartan elements in the convolution $L_\infty\{-1\}$-algebra $\Hom_{\base^{\op}}(A^\vee, \mf{g}\otimes\Omega_\bullet)$ (Remark \ref{rem:convolutionalgebra}). The result now follows from the fact that the maps
$$\begin{tikzcd}
A\otimes_\base\mf{g}\arrow[r] & A^{\vee\vee}\otimes_{\base} \mf{g}\arrow[r] & \Hom_{\base^{\op}}(A^\vee, \mf{g})
\end{tikzcd}$$
are isomorphisms of $L_\infty\{-1\}$-algebras, since $A$ is quasi-free, finitely generated over $\base$ (Remark \ref{rem:smallisperfect}).
\end{proof}
\begin{remark}\label{rem:maurer-cartan different resolution}
Suppose we are in the setting of Theorem \ref{thm:MC functor} and fix a strictly Artin $\PP$-algebra $A$. The above proof shows that the space
$$
\MC_\mf{g}(A)\simeq \Map_{\CC^\vee}(\mf{D}(A), \mf{g}) \simeq \Map_{\CC^\vee}\big(\Omega_{\phi^\dagger}(A^\vee), \mf{g}\big)
$$
can be presented by the simplicial set
$$
\MC\big(\Hom_{\base^{\op}}(A^\vee, \mf{g}_\bullet)\big)\cong \MC\big(A\otimes_{\base} \mf{g}_\bullet\big)
$$
for \emph{any choice} of fibrant simplicial resolution $\mf{g}_\bullet$ of the $\CC^\vee$-algebra $\mf{g}$. The additional feature of the particular choice $\mf{g}_\bullet=\mf{g}\otimes\Omega_\bullet$ is that one obtains a \emph{simplicially enriched} functor in $A$.
\end{remark}
\begin{remark}
Let $\CC$ be a $\base$-cooperad as in Assumption \ref{ass:mc} with contributions in arity $0$ or $1$. Let us denote by $\Omega_{nc}$ the cobar construction of non-coaugmented cooperads, i.e.\ $\Omega_{nc}(\CC)$ is by definition $\Omega(k\oplus \CC)$.
 A twisting morphism $\phi\colon\CC\drt \PP$ then determines an operad map 
$$\begin{tikzcd}
\Omega_{nc}(\cat{coCom}^u)\arrow[r] & \CC^\vee\otimes_{\mm{H}}\PP.
\end{tikzcd}$$
from the non-coaugmented cobar construction on the counital cocommutative cooperad. This operad is freely generated by symmetric operations $l_p$ of degree $1$, with $p\geq 0$; the operation $l_1$ differs from the differential. 

If $A$ is a strictly Artin $\Omega\CC$-algebra, then maps of $\CC^\vee$-algebras $\Omega(A^\vee)\rt \mf{g}$ correspond bijectively to \emph{Maurer--Cartan elements} of the nilpotent $\Omega_{nc}(\cat{coCom}^u)$-algebra $A\otimes_\base \mf{g}$, i.e.\ degree $0$ elements satisfying 
$$
dx+\sum_{p\geq 0}\frac{1}{p!}l_p(x, \dots, x)=0.
$$
One can then repeat the proof of Theorem \ref{thm:MC functor} to show the following: the formal moduli problem associated to a $\CC^\vee$-algebra $\mf{g}$ is represented by the simplicially enriched functor $A\longmapsto \mm{MC}\big(A\otimes_\base \mf{g}\otimes\Omega_\bullet\big)$ on strictly Artin $\Omega\CC$-algebras.
\end{remark}

\section{Relative Koszul duality}\label{sec:relativekoszul}
In this final section, we describe a somewhat simplified case of \emph{quadratic duality} in the setting of $\base$-operads. This was already mentioned in our discussion of operadic deformation problems (Section \ref{sec:operadicdeftheory}) and we will conclude this section by providing the leftover proofs of Theorem \ref{thm:main operad} and Proposition \ref{prop:adjunction Perm Osym} appearing there.

\subsection{Distributive laws and quadratic duality}
Recall that a $k$-operad $\mathscr P$ is said to be quadratic if it admits a presentation of the form $\mathscr P = \mm{Free}_{\cat{Op}}(V) / (R)$, where $V$ is a symmetric sequence of graded (but non-dg) $k$-vector spaces and 
$$
R\subseteq V \circ_{(1)} V = \mm{Free}^{(2)}_{\cat{Op}}(V)
$$
is contained in the subspace spanned by $V$-labeled trees with two vertices \cite[Section 7.1]{LodayVallette2012}. In this section, let us fix a dg-category $\base$ with set of objects $S$ and consider the following generalization of this:
\begin{definition}
A $\base$-operad $\PP$ is called \emph{quadratic} if it admits a quadratic presentation
$$
\PP=Q_{\base}(V, R)\coloneqq \mm{Free}_{\cat{Op}_{\base}}(V) / (R)
$$
where $V$ is a symmetric $\base$-bimodule which is free (not quasi-free) as a left $\base$-module and $R\subseteq V \circ_{(1)} V$.
In this case, its \emph{Koszul dual cooperad} $\mathscr P^\antishriek=Q_{\base}^\mm{co}(V[1], R[2])$ is the conilpotent quadratic $\base$-cooperad cogenerated by $V[1]$ with corelations $R[2]$. We denote its \emph{Koszul dual operad} (relative to $\base$) by $\mathscr P^! \coloneqq (\mathscr P^{\antishriek}\{-1\})^\vee$.
\end{definition}
There is a canonical $\base$-twisting morphism $\mathscr P^\antishriek \drt \mathscr P$ arising from the identity $V\to V[1]$ on (co)generators. Indeed, such a canonical twisting morphism $Q^\mm{co}_S(V[1], R[1])\drt Q_S(V, R)$ exists at the level of $S$-coloured operads \cite[Section 3]{van2003coloured}, and descends to the quotients obtained by taking tensor products relative to $\base$.

\begin{definition}
A quadratic operad $\PP=Q_\base(V, R)$ is said to be a \emph{weakly Koszul operad} (relative to $\base$) if the induced map $\mathscr P^\antishriek \rt \Bar \mathscr P$ is a quasi-isomorphism. In that case, the operad $\mathscr P^!$ is quasi-isomorphic to $\mf{D}_\base(\mathscr P)\{1\}$.
\end{definition}
\begin{observation}\label{obs:relative koszul case}
Suppose that $\base$ is concentrated in degrees $\leq 0$ and that $\PP=Q(V, R)$ is Koszul relative to $\base$ with $V$ concentrated in degrees $\geq 0$ and finitely many arities. Then $\PP$ is a splendid operad and $\FMP_{\PP}\simeq \Alg_{\PP^!}$ (as in Observation \ref{obs:koszul case}).
\end{observation}
A main source of quadratic $\base$-operads arises from distributive laws. If $V$ is an $S$-coloured symmetric sequence, let us say that a \emph{$\base$-law} on $V$ is a map of $S$-coloured symmetric sequences
$$\begin{tikzcd}
\Lambda\colon V\circ \base\arrow[r] & \base\circ V
\end{tikzcd}$$
such that the following diagrams commute
$$\begin{tikzcd}
	V\circ \base \circ \base\arrow[r, "\Lambda\circ 1"]\arrow[d, "V\circ \mu"{swap}] & \base\circ V\circ\base\arrow[r, "1\circ\Lambda"] & \base\circ\base\circ V\arrow[d, "\mu\circ V"] & V\circ k \arrow[r]\arrow[d, "\cong"{swap}] &   V \circ \base \arrow[d, "\Lambda"]\\
	V\circ \base\arrow[rr, "\Lambda"{swap}] & & \base\circ V & k\circ V \arrow[r]& \base \circ V.
\end{tikzcd}$$
Let $\Lambda\colon V\circ \base\rt \base\circ V$ be a $\base$-law on a symmetric sequence. Then $\base\circ V$ has the natural structure of a symmetric $\base$-bimodule (free as a left $\base$-module) via the maps
$$\begin{tikzcd}
\base\circ (\base\circ V)\arrow[r, "\mu\circ 1"] & \base\circ V & & (\base\circ V)\circ \base\arrow[r, "1\circ \Lambda"] & \base\circ \base\circ V\arrow[r, "\mu\circ 1"] & V.
\end{tikzcd}$$
Any symmetric $\base$-bimodule that is free as a left $\base$-module arises in this way. 
If $V=\mathscr{P}$ is an $S$-coloured operad (augmented, as always), then $\base\circ \mathscr{P}$ inherits a $\base$-operad structure via
$$\begin{tikzcd}
(\base\circ\mathscr{P}) \circ_\base (\base\circ\mathscr{P}) \cong \base\circ(\mathscr{P} \circ\mathscr{P}) \arrow[r] &  \base\circ\mathscr{P}
\end{tikzcd}$$
as long as the $\base$-law is a \emph{distributive law} in the sense of \cite[Section 8.6]{LodayVallette2012}, i.e.\ it is also \textit{right distributive}:
$$\begin{tikzcd}[column sep=1.9pc] %augmented case
	\mathscr{P}\circ\mathscr{P}\circ \base \circ \base\arrow[r, "\Lambda\circ 1"]\arrow[d, "\mu\circ 1"{swap}] & \mathscr{P}\circ \base\circ\mathscr{P}\arrow[r, "1\circ\Lambda"] & \base\circ\mathscr{P}\circ\mathscr{P}\arrow[d, "1\circ\mu"] & k\circ \base \arrow[r]\arrow[d, "\cong"{swap}] & \mathscr P \circ \base \arrow[r]\arrow[d, "\Lambda"] & k\circ \base\arrow[d, "\cong"]\\
	\mathscr{P}\circ\base\arrow[rr, "\Lambda"{swap}] & & \base\circ \mathscr{P} &  \base \circ k \arrow[r] & \base\circ \mathscr{P} \arrow[r] & \base \circ k.
\end{tikzcd}$$
Similarly, if $V=\mathscr{C}$ is an $S$-coloured cooperad, then
$$\begin{tikzcd}[column sep=2.3pc]
\base\circ \mathscr{C}\arrow[r, "\base\circ \Delta"] & \base\circ \mathscr{C}\circ\mathscr{C}\cong (\base\circ \mathscr{C})\circ_{\base}(\base\circ \mathscr{C})
\end{tikzcd}$$
endows $\base\circ \mathscr{C}$ with the natural structure of a $\base$-cooperad as long as the $\base$-law is \emph{codistributive}:
$$\begin{tikzcd}[column sep=1.9pc]
	\mathscr{C}\circ\base\arrow[rr, "\Lambda"]\arrow[d] & & \base\circ \mathscr{C}\arrow[d] & k\circ \base \arrow[r]\arrow[d, "\cong"{swap}] & \mathscr C \circ \base \arrow[r]\arrow[d, "\Lambda"] & k\circ \base\arrow[d, "\cong"]\\
	\mathscr{C}\circ\mathscr{C}\circ \base \circ \base\arrow[r, "\Lambda\circ 1"{swap}] & \mathscr{C}\circ \base\circ\mathscr{C}\arrow[r, "1\circ\Lambda"{swap}] & \base\circ\mathscr{C}\circ\mathscr{C} &  \base \circ k \arrow[r] & \base\circ \mathscr{C} \arrow[r] & \base \circ k.
\end{tikzcd}$$
\begin{example}
Let $\Lambda\colon V\circ \base\rt \base\circ V$ be a $\base$-law. This induces a distributive law on the free operad $\mm{Free}_{\cat{Op}_S}(V)$ and a codistributive law on the cofree cooperad $\mm{Cofree}_{\cat{Coop}_S}(V)$. 
\end{example}
\begin{definition}
Let $\phi\colon \CC\drt \mathscr{P}$ be a twisting morphism. We will say that a \emph{$\base$-law} on $\phi$ is the data of:
\begin{itemize}\setlength{\itemsep}{1pt}
\item a codistributive law $\Lambda_{\mathscr{C}}$ on $\mathscr{C}$.% (compatible with the coaugmentation).
\item a distributive law $\Lambda_{\mathscr{P}}$ on $\mathscr{P}$.% (compatible with the augmentation).
\end{itemize}
such that $\phi$ intertwines $\Lambda_{\mathscr{C}}$ and $\Lambda_{\mathscr{P}}$.
\end{definition}
\begin{lemma}
If $\phi\colon \mathscr{C}\drt \mathscr{P}$ is a twisting morphism and $\Lambda$ is a distributive law on $\phi$, then $\base\circ \phi\colon \base\circ\mathscr{C}\drt \base\circ \mathscr{P}$ is a $\base$-twisting morphism (see Construction \ref{constr:k-twisting morphisms}).
\end{lemma}
\begin{proof}
	This follows from checking that given maps $f,g\colon \mathscr C \rt \mathscr P$ such that $f$ intertwines $\Lambda_{\mathscr{C}}$ and $\Lambda_{\mathscr{P}}$, the equation $(\base \circ f)\star( \base \circ g) = \base \circ (f \star g)$ is satisfied.
\end{proof}

\begin{example}\label{ex:klawbar}
Let $\Lambda_{\mathscr{P}}$ be a distributive law on an augmented $k$-operad $\mathscr{P}$. This extends to a canonical $\base$-law on the universal twisting morphism $\pi\colon \Bar(\mathscr{P})\drt \mathscr{P}$ such that $\base\circ \pi$ is the universal twisting morphism of the $\base$-operad $\base\circ\mathscr{P}$.
In particular, if $\PP$ is splendid, then $\base\circ \PP$ is a splendid $\base$-operad.
\end{example}

The following proposition shows that under good conditions, distributive laws are compatible with Koszul duality.

\begin{proposition}\label{prop:klawquadratic}
Let $\mathscr{P}=\mm{Q}(V, R)$ be an $S$-coloured quadratic operad and consider a $\base$-law $\Lambda\colon V\circ \base \rt \base\circ V$ such that the induced distributive law on $\mm{Free}_{\cat{Op}_S}(V)$ preserves the quadratic relations $R$. In other words, $\Lambda$ induces a distributive law on $\mathscr{P}$. Then:
\begin{enumerate}
\item $\Lambda[1]\colon V[1]\circ \base \rt \base\circ V[1]$ induces a codistributive law on the quadratic dual $\mathscr{P}^\antishriek$.
\item $\Lambda$ and $\Lambda[1]$ together determine a $\base$-law on the twisting morphism $\mathscr{P}^\antishriek\drt \mathscr{P}$.
\item If $\mathscr{P}^\antishriek\drt \mathscr{P}$ is weakly Koszul, then the induced twisting morphism $\base\circ \mathscr{P}^\antishriek\drt \base\circ\mathscr{P}$ is weakly Koszul over $\base$.
\end{enumerate}
\end{proposition}
\begin{proof}
For (1), one can check that $\Lambda[1]$ preserves the quadratic relations. For (2), note that the maps 
$$
\mm{Q}^{\mm{co}}(V[1], R[2])\rt V[1]\rt V\rt \mm{Q}(V, R)
$$
are all compatible with the distributive law by construction. Finally, for (3), the map of cooperads $\phi'\colon \PP^{\antishriek}\rt \Bar(\mathscr{P})$ is compatible with the codistributive law on $\PP^\antishriek$ and that of Example \ref{ex:klawbar} (both ultimately arise from $\Lambda$). Consequently, there is an induced map of $\base$-coalgebras 
$$\begin{tikzcd}
\base\circ \phi'\colon \base\circ \mm{Q}^{\mm{co}}(V[1], R[2])\arrow[r] & \base\circ \Bar(\mathscr{P}).
\end{tikzcd}$$
Since the composition product (over $k$) preserves quasi-isomorphisms, the result follows.
\end{proof}

\subsection{\texorpdfstring{Koszul self-duality of $\mathscr{O}^\mm{sym}$ and proof of Theorem \ref{thm:main operad}}{Proof of Theorem \ref{thm:main operad}}}\label{sec:proof of thm:main operad}
Let us now spell out how the operad $\mathscr{O}^\mm{sym}$ for nonunital symmetric operads, discussed in Section \ref{sec:operadicdeftheory}, fits into the framework from the previous section. To this end, notice that the presentation of the $\mathbb Z_{\geq 0}$-coloured operad of nonsymmetric operads $\mathscr O^{\mm{ns}}$ given in Definition \ref{def:op of nsops} is quadratic and therefore fits the framework of (coloured) Koszul duality:

\begin{proposition}[{\cite[Theorem 4.3]{van2003coloured}}]\label{prop:vdL}
The quadratic $\mathbb{Z}_{\geq 0}$-coloured operad $\mathscr{O}^{\mm{ns}}$ is Koszul, and it is isomorphic to its Koszul dual operad ${\mathscr{O}^{\mm{ns}}}^!=\mathscr{O}^{\mm{ns}}$.
\end{proposition}

A priori a similar result is not expectable for the operad of symmetric operads since the relations for the symmetric group are not quadratic. Dehling and Vallette \cite{dehling2015symmetric} used curved Koszul duality theory to construct an appropriate cofibrant replacement functor over any ring. Crucially, they observed that the symmetric group data could be obtained as a distributive law.

\begin{proposition}[{\cite[Proposition 1.9]{dehling2015symmetric}}]\label{prop:vallettedehling}
There is a $k[\Sigma]$-law on the quadratic data generating $\mathscr{O}^{\mm{ns}}$, such that 
$$
k[\Sigma]\circ \mathscr{O}^{\mm{ns}}\cong \mathscr{O}^{\mm{sym}}.
$$
\end{proposition}
Following the previous section we can therefore interpret $\mathscr{O}^{\mm{sym}}$ as a quadratic $k[\Sigma]$-operad which is Koszul. Using the canonical equivalence $\mm{inv}\colon \Sigma\rt \Sigma^\op$, sending a permutation to its inverse, to identify $k[\Sigma]$-operads with $k[\Sigma^{\op}]$-operads, the Koszul dual of $\mathscr{O}^\mm{sym}$ can be viewed as a $k[\Sigma]$-operad as well.

\begin{corollary}[The operad of symmetric operads is Koszul]\label{cor:OpisKoszul}
The quadratic $k[\Sigma]$-operad $\mathscr{O}^{\mm{sym}}$ is self-dual in the sense that
$$
\mf{D}_{k[\Sigma]}({\mathscr{O}^{\mm{sym}}})\{1\} \simeq \big(\mathscr{O}^{\mm{sym}}\big){}^!\cong \mathscr{O}^{\mm{sym}}.
$$
\end{corollary}
\begin{proof}
Propositions \ref{prop:vdL} and \ref{prop:vallettedehling}, together with Proposition \ref{prop:klawquadratic} imply that $\mc{O}^\mm{sym}$ is a (weakly) Koszul operad with quadratic dual cooperad $k[\Sigma]\circ (\mathscr{O}^{\mm{ns}})^{\text{!`}}$. Since $k[\Sigma]\simeq k[\Sigma^{\op}]$ and using Proposition \ref{prop:vdL} once more, one finds that the Koszul dual of $\mathscr{O}^\mm{sym}$ is $k[\Sigma]\circ (\mathscr{O}^{\mm{ns}}){}^!\cong \mathscr{O}^\mm{sym}$. 
\end{proof}
\begin{proposition}\label{prop:bar of operad is operadic bar}
Consider the $\base$-twisting morphism $\big(\mathscr{O}^{\mm{sym}}\big)^{\antishriek}\rt \mathscr{O}^{\mm{sym}}$ relative to $k[\Sigma]$. The induced bar-cobar adjunction can be identified with the usual bar-cobar adjunction
$$\begin{tikzcd}
\Bar\colon \cat{Op}^{\mm{sym}}\arrow[r, yshift=0.8ex] & \cat{CoOp}^{\mm{sym}}\colon \Omega\arrow[l, yshift=-0.8ex]
\end{tikzcd}$$
between nonunital $k$-operads and $k$-cooperads.
\end{proposition}
\begin{proof}[Proof (sketch)]
Note that coalgebras for the $k[\Sigma]$-cooperad $\big(\mathscr{O}^{\mm{sym}}\big){}^{\antishriek}=k[\Sigma]\circ \big(\mathscr{O}^{\mm{ns}}\big){}^{\text{!`}}$ are symmetric sequences with a (conilpotent) non-counital cooperad structure. The bar construction then takes the cofree $\big(\mathscr{O}^{\mm{sym}}\big){}^{\antishriek}$-coalgebra (in symmetric sequences) with some differential. Unraveling the definitions, this is exactly the cofree cooperad with the bar differential.
\end{proof}

\subsection{\texorpdfstring{Relating operads over $k$ and $k[\Sigma]$ and proof of Proposition \ref{prop:adjunction Perm Osym}}{Proof of Proposition \ref{prop:adjunction Perm Osym}}}\label{sec:adjunctions for Perm and Osym}
In this final section, we will describe a functor $L\colon \cat{Op}_k\rt \cat{Op}_{k[\Sigma]}$ associating to each $k$-operad a $k[\Sigma]$-operad, and show that $L$ preserves Koszul operads (see Proposition \ref{prop:adjunction Perm Osym}). This was already used in Section \ref{sec:operadicdeftheory} to relate permutative algebras and nonunital operads using a map of $k[\Sigma]$-operads $\mathscr{O}^\mm{sym}\rt L(\Perm)$.
It will be convenient to compare operads over $k$ and over $k[\Sigma]$ in two steps, passing by the linear category $\base^{\mm{ns}}$ with non-negative integers as objects and morphisms given by multiples of the identities.

\begin{construction}\label{con:weights}
	Consider the following two adjoint pairs of functors:
	$$\begin{tikzcd}[column sep=2pc]
		\wt\colon \Mod^{\dg}_k\arrow[r, yshift=0.8ex] & \Mod^{\dg}_{\base^{\mm{ns}}}\cocolon \Tot\arrow[l, yshift=-0.8ex] & 
		\wt\colon \cat{BiMod}_k^{\Sigma, \dg}\arrow[r, yshift=0.8ex] & \cat{BiMod}_{\base^{\mm{ns}}}^{\Sigma, \dg}\cocolon \Tot.\arrow[l, yshift=-0.8ex]
	\end{tikzcd}$$
	Here the first adjunction between categories of modules simply sends a $k$-module $M$ to $\mm{wt}(M)(n)=M$ and a $\base^{\mm{ns}}$-module $N$ to $\Tot(N)=\prod_n N(n)$. The second adjunction, at the level of symmetric bimodules, has left adjoint
	$$
	\mm{wt}(\mathscr{M})(n_1, \dots, n_k; n_0)=\left\{\begin{array}{cl} \mathscr{M}(k) & \text{if } n_1+\dots+n_k=n_0+k-1\\
		0 & \text{otherwise}\end{array}\right.
	$$
	and right adjoint 
	$$
	\Tot(\mathscr{N})(k)=\prod_{n_1, \dots, n_k}  \mathscr{N}(n_1, \dots, n_k; n_1+\dots+n_k+1-k).
	$$
	Note that viewing a module as a symmetric bimodule concentrated in arity $0$ does \emph{not} identify the two version of $\wt$ and $\Tot$. All of these functors preserve quasi-isomorphisms.
	
	The explicit description of the relative composition product shows that $\wt$ is a (strong) monoidal functor which is (strongly) compatible with the action of symmetric bimodules on left modules, i.e.
	$$
	\wt(k)=\base^{\mm{ns}}, \qquad \wt(\mathscr{M}\circ_k \mathscr{N})\cong \wt(\mathscr{M})\circ_{\base^{\mm{ns}}}\wt(\mathscr{N}), \qquad \wt(\mathscr{M}\circ_k N)\cong \wt(\mathscr{M})\circ_{\base^{\mm{ns}}}\wt(N)
	$$
	for symmetric sequences $\mathscr{M}$ and $\mathscr{N}$ and a $k$-module $N$. It follows that $\Tot$ is lax monoidal, i.e.\ it preserves operads and algebras over them.
\end{construction}
\begin{remark}\label{rem:graded alg}
	Let $\QQ$ be a $k$-operad. Unraveling the definitions, $\wt(\QQ)$ is a $\base^{\mm{ns}}$-operad whose algebras can be identified with $\mathbb{Z}_{\geq 0}$-graded $\QQ$-algebras, where each $n$-ary operation in $\QQ$ has weight $1-n$. In this case, the adjoint pair $\wt\colon \Alg_{\QQ}\leftrightarrows \Alg_{\wt(\QQ)}\cocolon\Tot$ sends a $\QQ$-algebra $A$ to the $\mathbb{Z}_{\geq 0}$-graded algebra $\wt(A)(p)=A$, and $\Tot(B)=\prod_n B(n)$.
\end{remark}
\begin{construction}\label{con:trivial symmetry}
	Let $\pi\colon k[\Sigma]\rt \base^{\mm{ns}}$ be the evident $k$-linear functor sending $n$ to $n$ and all permutations to the identity. This induces restriction and (co)induction functors at the level of modules, and similarly at the level of symmetric bimodules
	$$\begin{tikzcd}[column sep=2pc]
		\pi_!\colon \cat{BiMod}_{k[\Sigma]}^{\Sigma, \dg}\arrow[r, yshift=0.8ex] & \cat{BiMod}_{\base^{\mm{ns}}}^{\Sigma, \dg}\cocolon \pi^*\arrow[l, yshift=-0.8ex] & \pi^*\colon \cat{BiMod}_{\base^{\mm{ns}}}^{\Sigma, \dg}\arrow[r, yshift=0.8ex] & \cat{BiMod}_{k[\Sigma]}^{\Sigma, \dg}\cocolon \pi_*.\arrow[l, yshift=-0.8ex]
	\end{tikzcd}$$
	Explicitly, $\pi^*\mc{M}(n_1, \dots, n_k; n_0)$ carries the trivial $\Sigma_{n_1}\times\dots \times\Sigma_{n_0}$-action and $\pi_!$ and $\pi_*$ take the coinvariants and invariants with respect to this action.
	
	As is the case for any dg-functor, $\pi_!$ is strong monoidal for the relative composition product (as follows from the explicit description in Section \ref{sec:operads over dg-cat}). Consequently, the fully faithful functor $\pi^*$ is lax monoidal. In fact, the structure maps for the lax monoidal structure on $\pi^*$ have the following additional property: the lax monoidality and lax unitality maps induce isomorphisms
	\begin{equation}\label{diag:almost monoidal}\begin{tikzcd}[row sep=0.1pc]
			\pi^*\mathscr{M}\circ_{k[\Sigma]} \pi^*\mathscr{N}\arrow[r, "\cong"] & \pi^*(\mathscr{M}\circ_{\base^{\mm{ns}}}\mathscr{N}) \\
			\pi^*\mathscr{M}\circ_{k[\Sigma]} \big(k[\Sigma]\oplus \pi^*\mathscr{N}\big)\arrow[r, "\cong"] & \pi^*\mathscr{M}\circ_{k[\Sigma]} \big(\pi^*\base^{\mm{ns}}\oplus \pi^*\mathscr{N}\big).
	\end{tikzcd}\end{equation}
	Indeed, this follows from the explicit description of the relative tensor product and the fact that trivial representations are closed under tensor products.
	
	Now let us write $\cat{X}^{\dg}_{k[\Sigma]}$ for the category of  retract diagrams $k[\Sigma]\rt \cat{M}\rt k[\Sigma]$ of symmetric $k[\Sigma]$-bimodules (and likewise for $\base^{\mm{ns}}$). The pointwise composition product endows this category with a monoidal structure, such that associative algebras are precisely (augmented) $k[\Sigma]$-operads. The adjoint pair $(\pi^*, \pi_*)$ then induces an adjoint pair
	$$\begin{tikzcd}
		\ol{\pi}^*\colon \cat{X}^{\dg}_{\base^{\mm{ns}}}\arrow[r, yshift=0.8ex] & \cat{X}_{k[\Sigma]}^{\dg}\colon \ol{\pi}_*.\arrow[l, yshift=-0.8ex]
	\end{tikzcd}$$
	Explicitly, $\ol{\pi}^*(\base^{\mm{ns}}\oplus \ol{\mathscr{M}})=k[\Sigma]\oplus \pi^*\ol{\mathscr{M}}$ and similarly for $\ol{\pi}_*$. The two natural isomorphisms \eqref{diag:almost monoidal} now imply that $\ol{\pi}^*$ is \emph{strong} monoidal. It follows that there are adjoint pairs
	$$\begin{tikzcd}
		\ol{\pi}^*\colon \cat{Op}_{\base^{\mm{ns}}}\arrow[r, yshift=0.8ex] & \cat{Op}_{k[\Sigma]}\colon \ol{\pi}_*\arrow[l, yshift=-0.8ex] & \pi^*\colon \Alg_{\PP} \arrow[r, yshift=0.8ex] & \cat{Op}_{\ol{\pi}^*\PP}\colon \pi_*\arrow[l, yshift=-0.8ex]
	\end{tikzcd}$$
	between $\base^{\mm{ns}}$-operads and $k[\Sigma]$-operads, and between algebras over a $\base^{\mm{ns}}$-operad $\PP$ and algebras over $\ol{\pi}_*\PP$.
\end{construction}
\begin{remark}\label{rem:algebra with trivial sym}
	Given a $\base^{\mm{ns}}$-operad $\QQ$, an algebra over $\ol{\pi}^*\QQ$ is simply given by a symmetric sequence $\cat{A}$, together with operations $\cat{A}(n_1)\otimes \dots \otimes \cat{A}(n_k)\rt \cat{A}(n_0)$ for each element of $\QQ(n_1, \dots, n_k; n_0)$ that are \emph{invariant} under pre- and postcomposition with some $\sigma\in \Sigma_{n_i}$. This induces a $\PP$-algebra structure on the invariants of the $\Sigma_{n_1}\times \dots \times \Sigma_{n_0}$-action.
\end{remark}
Combining Construction \ref{con:weights} and Construction \ref{con:trivial symmetry} yields a (left adjoint) functor
$$
L=\ol{\pi}^*\circ \wt\colon \cat{Op}_k\rt \cat{Op}_{k[\Sigma]}.
$$
Using this, we prove Proposition \ref{prop:adjunction Perm Osym}:
\begin{proof}[Proof of Proposition \ref{prop:adjunction Perm Osym}]
Since both $\ol{\pi}^*$ and $\wt$ are strong monoidal for the composition product, they preserve quadratic presentations. Furthermore, the explicit formulas in Construction \ref{con:weights} and \ref{con:trivial symmetry} show that they commute with linear duality and preserve all quasi-isomorphisms. It follows that for any Koszul binary quadratic operad $\QQ$, the $k[\Sigma]$-operad $L(\QQ)$ is Koszul and Koszul dual  to $L(\QQ^!)$ (relative to $k[\Sigma]$).

Let us now compare algebras over a $k$-operad $\QQ$ and over $L(\QQ)$. Remarks \ref{rem:graded alg} and \ref{rem:algebra with trivial sym} show that an $L(\QQ)$-algebra is given by a symmetric sequence $\cat{A}$, together with an operation $q\colon \cat{A}(n_1)\otimes \dots\otimes \cat{A}(n_k)\rt \cat{A}(n_1+\dots+n_k-k+1)$ for each $q\in \QQ(k)$ which is invariant under pre- and postcomposition with the symmetric group actions. 
Combining Constructions \ref{con:weights} and \ref{con:trivial symmetry} then shows that for any $k$-operad $\QQ$, there is an adjoint pair
\[\begin{tikzcd}
	L\colon \cat{Alg}_{\mathscr{Q}}^{\dg}\arrow[r, yshift=0.8ex] & \cat{Alg}_{L(\mathscr{Q})}^{\dg}\colon R\arrow[l, yshift=-0.8ex]
\end{tikzcd}\]
as announced in Proposition \ref{prop:adjunction Perm Osym}.
\end{proof}

\appendix

\section{Operadic toolkit}\label{sec:operads}

In this section we introduce the operadic homotopical algebra required for our purposes throughout the text, notably in Section \ref{sec:relativekoszul}. The results from this section hold over a general dg-category $\base$ but many of them are standard when the base $\base$ is the point and similar statements can be found in \cite{LodayVallette2012} or in \cite{lefevre2003infini}. 

We recall that we work over a fixed field $k$, and we denote by $S$ the set of objects of the dg-category $\base$. When we use the term operad (resp.\ cooperad) we always mean \emph{unital augmented operad} (resp.\ counital coaugmented cooperad) unless otherwise explicitly written.

\subsection{Operads over a dg-category}\label{sec:operads over dg-cat}
A  \emph{symmetric $\base$-bimodule} in $S$-coloured symmetric sequences is a family of chain complexes over $k$
$$
V(\ul{c}; c_0)\coloneqq V(c_1, \dots, c_p; c_0), \quad c_i \in S
$$
together with maps
$$
	\base(c_0, d_0)\otimes V(\ul{c}; c_0) \otimes \bigotimes_{i} \base(d_{\phi^{-1}(i)}, c_i)\longrightarrow \base(d_1, \dots, d_p; d_0), \quad d_i \in S$$
for every $\phi\colon \{1, \dots, p\}\rt \{1, \dots, p\}$ which satisfy natural associativity conditions.

\begin{definition}\label{def:symmetricbimodules}
 We denote by
$
\cat{BiMod}^{\Sigma, \dg}_{\base}
$
the category of \emph{symmetric $\base$-bimodules}.
\end{definition}
Note that a symmetric $\base$-bimodule $M$ has an arity $p$ part $M(p)$, which is a $\base$-$\base^{\otimes p}$-bimodule with a $\Sigma_p$-action that is compatible with the right $\base^{\otimes p}$-structure. In arity $1$ this is simply a $\base$-bimodule in the usual sense, i.e.\ a functor $\base\otimes \base^{\op}\rt \cat{Ch}_k$, while a symmetric $\base$-bimodule in arity $0$ is a left $\base$-module $\base\rt \cat{Ch}_k$.

The category $\cat{BiMod}_{\base}^{\Sigma, \dg}$ has a (nonsymmetric) monoidal structure given by the relative composition product $\circ_{\base}$. An element in $M\circ_{\base} N$ can be identified with a tree of height $2$ with root vertex labeled by $\phi\in M(c_1, \dots, c_p; c_0)$ and all other vertices labeled by $\psi_1, \dots, \psi_p\in N$, with $\psi_i$ having an output of color $c_i$, subject to the relation that edges are $\base $ equivariant. In other words, for all $a\in \base (c_i,c'_i)$, 
$$
\text{labeling }(\phi\circ_i a), \psi_1, \dots, \psi_p \sim \text{ labeling }\phi, \psi_1, \dots, a\psi_i, \dots, \psi_q.
$$
\begin{proposition}\label{prop:modelstructures}
The following categories carry model structures, in which the fibrations are the surjections and the weak equivalences are the quasi-isomorphisms:
\begin{enumerate}
\item\label{it:operads} The category of \emph{$\base$-operads}, defined to be the category of augmented unital associative algebras in symmetric $\base$-bimodules
$$
\cat{Op}_{\base}^{\dg}\coloneqq \cat{Alg}^\mm{aug}\big(\cat{BiMod}_{\base}^{\Sigma, \dg}\big).
$$
\item\label{it:modules} For any associative algebra $\mathscr{P}$ in symmetric $\base$-bimodules, the categories of left and right $\mathscr P$-modules
$$
\cat{LMod}_{\mathscr{P}}^{\dg}\coloneqq\cat{LMod}_{\mathscr{P}}\big(\cat{BiMod}_{\base}^{\Sigma, \dg}\big)\qquad\text{and}\qquad \cat{RMod}_{\mathscr{P}}^{\dg}\coloneqq\cat{RMod}_{\mathscr{P}}\big(\cat{BiMod}_{\base}^{\Sigma, \dg}\big).
$$
In particular, the category $\cat{BiMod}_{\base}^{\Sigma, \dg}$ itself.
\item\label{it:algebras} For any $\base$-operad $\mathscr P$, the category of \emph{$\mathscr P$}-algebras, defined to be the category of left $\mathscr P$-modules that are concentrated in arity $0$
$$
\Alg_{\PP}^{\dg}\coloneqq \cat{LMod}_{\mathscr P}\big(\cat{LMod}^{\dg}_\base\big).
$$
\end{enumerate}
\end{proposition}
\begin{proof}
	The proposition follows essentially from the fact that over a field of characteristic zero, algebras over a coloured operad have a canonical model structure \cite{HinichRectification}. For example, \ref{it:operads} the category of augmented unital $\base$-operads can be identified with the category of nonunital operads in symmetric $\base$-bimodules; these are algebras over an operad with set of colours given by $\coprod_{n\geq 0} S^{\times n-1}$ (cf.\ Definition \ref{def:op of symops}). Something similar holds for \ref{it:modules} left and right modules, and for \ref{it:algebras} it suffices to observe that an algebra over a $\base$-operad is simply an algebra over its underlying $S$-coloured operad.
\end{proof}
\begin{example}[Free algebras]\label{ex:free}
Let $\PP$ be an $\base$-operad and $V$ a left $\base$-module. Then the free $\PP$-algebra on $V$ is given by the usual formula
$$
\PP(V)\coloneqq \PP\circ_\base V = \bigoplus_p \PP(p)\otimes_{\Sigma_p\ltimes \base^{\otimes p}} V^{\otimes p}.
$$
\end{example}
A symmetric $\base$-bimodule that is cofibrant for the model structure from Proposition \ref{prop:modelstructures}\ref{it:modules} is (in particular) given in each arity $p$ by a quasi-projective $\base$-$\base^{\otimes p}$-bimodule. For many practical purposes, it will suffice to impose a slightly weaker cofibrancy condition, concerning only the left $\base$-module structure:
\begin{definition}\label{def:cofibrantleftmodule}
A symmetric $\base$-bimodule $M$ is \emph{cofibrant as a left $\base$-module} if for each tuple of objects $c_i\in S$, the left $\base$-module $M(c_1, \dots, c_p; -)$ is cofibrant. This holds in particular if $M$ is cofibrant in the model structure of Proposition \ref{prop:modelstructures}. If $M$ and $N$ are cofibrant as left $\base$-modules, then $M\circ_\base N$ is as well.
\end{definition}
\begin{proposition}\label{prop:cofibrant algebras are cofibrant modules}
Let $\PP$ be a $\base$-operad which is cofibrant as a left $\base$-module. Then the forgetful functor $\Alg^{\dg}_{\PP}\rt \cat{LMod}^{\dg}_{\base}$ preserves cofibrant objects, i.e.\ every cofibrant $\PP$-algebra is also cofibrant as a left $\base$-module.
\end{proposition}
\begin{proof}
This follows from a variation of the argument from \cite[Appendix 5]{bergermoerdijk2003}. Let us consider the following two conditions on a $\PP$-algebra $A$:
\begin{enumerate}
\item for each cofibrant $\base$-module $W$, the $\PP$-algebra coproduct $A\sqcup \PP(W)$ is cofibrant as a $\base$-module.
\item $A$ is cofibrant as a $\base$-module and for each cofibration of $\base$-modules $V\rightarrowtail W$ and a map $V\to A$, the map $A\to A\sqcup_{\PP(V)}\PP(W)$ into the pushout of $\PP$-algebras is a cofibration of $\base$-modules.
\end{enumerate}
Clearly (2) implies (1), and the converse implication holds as well. To see this, we claim that there exists an increasing filtration on $A\sqcup_{\PP(V)}\PP(W)$ whose associated graded is $A\sqcup \PP(W/V)$. Assuming this, condition (1) implies that the cokernel of $i\colon A\hookrightarrow A\sqcup_{\PP(V)}\PP(W)$ admits an increasing filtration whose associated graded is the cokernel of the summand inclusion $A\hookrightarrow A\sqcup \PP(W/V)$. The latter is cofibrant, so that the cokernel of $i$ is cofibrant as well and $i$ is a cofibration of $\base$-modules (cf.\ Remark \ref{rem:filtrations}).

For the desired filtration, we can filter the $\base$-module $W$ by $F_0(W)=V$ and $F_1(W)=W$, put $A$ and $V$ in weight $0$ and compute the pushout $A\sqcup_{\PP(V)}\PP(W)$ in the category of $\PP$-algebras in filtered $\base$-modules. Since forgetting the filtration defines a symmetric monoidal left adjoint functor from filtered $\base$-modules to $\base$-modules, this provides a filtration on $A\sqcup_{\PP(V)}\PP(W)$. Likewise, taking the associated graded is symmetric monoidal, so that the associated graded can be identified with $A\sqcup_{\PP(V)} \PP(\mm{gr}(W))$. Since $\mm{gr}(W)=V\oplus (W/V)$ (with $W/V$ of weight $1$), this coincides with the graded $\PP$-algebra $A\sqcup \PP(W/V)$. 

Now let us say that a $\PP$-algebra $A$ is \emph{adequate} if it satisfies the equivalent conditions (1) and (2). Our goal will be to show that all cofibrant $\PP$-algebras are adequate. To see this, note that if $A$ is adequate and $V'\rt W'$ is a cofibration of $\base$-modules, then $A\sqcup_{\PP(V')}\PP(W')$ is adequate (one easily verifies (1) using condition (2) for $A$). The class of adequate $\PP$-algebras is therefore closed under iterated pushouts along generating cofibrations and under retracts. To conclude that it contains all cofibrant $\PP$-algebras, it remains to verify that the \emph{initial} $\PP$-algebra $\PP(0)$ is adequate, i.e.\ that for any cofibrant left $\base$-module $V$, the free $\PP$-algebra $\PP(V)$ is cofibrant as a left $\base$-module. This follows directly from the formula in Example \ref{ex:free} and the fact that $\PP$ was cofibrant as a left $\base$-module.
\end{proof}

Given a $\base$-operad $\mathscr P$ together with a right module $M$ and a left module $N$, we denote by $M\circ_\mathscr P N$ the coequalizer of $M\circ_\base \mathscr P \circ_\base N \rightrightarrows M\circ_\base N$.

\begin{lemma}\label{lem:circ preserves qi}
Let $\mathscr{P}\in\cat{Op}^{\dg}_{\base}$ and suppose that $M\in \cat{RMod}^{\dg}_{\mathscr{P}}$ and $N\in\cat{LMod}^{\dg}_{\PP}$ are cofibrant. Then the two functors
$$\begin{tikzcd}
M\circ_{\mathscr{P}}(-)\colon\cat{LMod}^{\dg}_{\cat{P}}\arrow[r] & \cat{BiMod}_{\base}^{\Sigma, \dg} & (-)\circ_{\mathscr{P}}N\colon\cat{RMod}^{\dg}_{\cat{P}}\arrow[r] & \cat{BiMod}_{\base}^{\Sigma, \dg}
\end{tikzcd}$$
both preserve quasi-isomorphisms.
\end{lemma}
\begin{proof}
We will only deal with the first functor, the other is similar. Consider the simplicial resolution of $M$ as a right $\PP$-module $M\circ\PP\circ\PP \rightrightarrows M\circ \PP\rightarrow M$, where $\circ$ is the composition product for $S$-coloured symmetric sequences of chain complexes. Since $M$ is cofibrant, it is quasi-free as a right $\PP$-module; in particular, without differentials this augmented simplicial object has extra degeneracies. Taking the relative composition product over $\PP$ with a quasi-isomorphism $X\rt Y$ yields a map of augmented simplicial objects
$$\begin{tikzcd}
M\circ_{\PP} X\arrow[d] & M\circ X\arrow[d, "\sim"] \arrow[l] & M\circ \PP\circ X\arrow[d, "\sim"]\arrow[l, yshift=-0.5ex]\arrow[l, yshift=0.5ex] & \dots\arrow[l, yshift=-1ex]\arrow[l, yshift=1ex]\arrow[l]\\
M\circ_{\PP} Y & M\circ Y \ar[l] & M\circ \PP\circ X\arrow[l, yshift=-0.5ex]\arrow[l, yshift=0.5ex] & \dots\arrow[l, yshift=-1ex]\arrow[l, yshift=1ex]\arrow[l]
\end{tikzcd}$$
Since the composition product $\circ$ preserves quasi-isomorphisms, all marked vertical maps are quasi-isomorphisms. Without differentials, the rows are augmented simplicial objects with (natural) extra degeneracies, so that the above diagram provides a simplicial resolution of the map $M\circ_\PP X\rt M\circ_\PP Y$ and the result follows.
\end{proof}
\begin{remark}\label{rem:derivedcomposition}
Lemma \ref{lem:circ preserves qi} implies that the composition product has a left derived functor, which we will denote by
$$
M\circ^h_{\PP} N
$$
and which can be computed by taking a cofibrant resolution of either $M$ or $N$. A quasi-isomorphism $\PP\rt \QQ$ induces a quasi-isomorphism $M\circ^h_{\PP} N\rt M\circ^h_{\QQ} N$ for any $M\in\cat{RMod}^{\dg}_\QQ$ and $N\in\cat{LMod}^{\dg}_\QQ$.
\end{remark}
\begin{corollary}\label{cor:quasiiso gives equivalence}
Given a map $f\colon \PP\rt \QQ$ in $\cat{Op}^{\dg}_{\base}$, there are Quillen adjunctions
$$\begin{tikzcd}
f_!\colon \Alg^{\dg}_{\PP}\arrow[r, yshift=0.8ex] & \Alg^{\dg}_{\QQ}\arrow[l, yshift=-0.8ex]\colon f^* & f_!\colon \cat{LMod}^{\dg}_{\PP}\arrow[r, yshift=0.8ex] & \cat{LMod}^{\dg}_{\QQ}\arrow[l, yshift=-0.8ex]\colon f^*
\end{tikzcd}$$
given by restriction and induction. When $f$ is a quasi-isomorphism, these are Quillen equivalences.
\end{corollary}
\begin{proof}
The functor $f^*$ clearly preserves (and detects) fibrations and quasi-isomorphisms. When $f$ is a quasi-isomorphism, $(f_!, f^*)$ is a Quillen equivalence because the counit $f_!f^*(M)=\QQ\circ_{\PP} M\rt M$ is a quasi-isomorphism for all cofibrant $M$ by Lemma \ref{lem:circ preserves qi}.
\end{proof}

\subsubsection*{Dualizing}
Given two dg-categories $\base_1$ and $\base_2$, one can take the \emph{exterior Hadamard tensor product}
\begin{equation}\label{eq:exterior hadamard}
\begin{tikzcd}
\cat{BiMod}_{\base_1}^{\Sigma, \dg}\times \cat{BiMod}_{\base_2}^{\Sigma, \dg}\arrow[r] & \cat{BiMod}_{\base_1\otimes \base_2}^{\Sigma, \dg}; \hspace{4pt} (M_1, M_2)\arrow[r, mapsto] & M_1\otimes M_2
\end{tikzcd}\end{equation}
where for any $c_i\in \base_1$ and $d_i\in \base_2$, 
$$
\big(M_1\otimes M_2\big)\big((c_1, d_1), \dots, (c_p, d_p); (c_0, d_0)\big)=M_1(c_1, \dots, c_p; c_0)\otimes M_2(d_1, \dots, d_p; d_0).
$$
From the description of the composition product, one sees that it is compatible with the exterior Hadamard tensor product in the sense that there is a natural morphism
\begin{equation}\label{eq:hadarmardovercomposition}
\begin{tikzcd}
\big(M_1\circ_{\base_1} N_1\big)\otimes \big(M_2\circ_{\base_2} N_2\big)\arrow[r] & \big(M_1\otimes M_2\big)\circ_{\base_1\otimes \base_2} \big(N_1\otimes N_2\big).
\end{tikzcd}\end{equation}
An element in the domain can be represented by a tensor product of two trees of height two, with vertices labeled by $M_1$ and $N_1$, resp.\ by $M_2$ and $N_2$. Such a tensor product is sent to zero if the two trees are different and if the trees are the same, one labels its vertices by the corresponding elements in $M_1\otimes M_2$ and $N_1\otimes N_2$.

The exterior Hadamard tensor product preserves colimits in both of its variables. It follows that there are functors
$$\begin{tikzcd}[row sep=0pc]
\ul{\Hom}_{\base_1}(-, -)\colon \Big(\cat{BiMod}_{\base_1}^{\Sigma, \dg}\Big)^{\op}\times \cat{BiMod}_{\base_1\otimes \base_2}^{\Sigma, \dg}\arrow[r] & \cat{BiMod}_{\base_2}^{\Sigma, \dg}\\
\ul{\Hom}_{\base_2}(-, -)\colon \Big(\cat{BiMod}_{\base_2}^{\Sigma, \dg}\Big)^{\op}\times \cat{BiMod}_{\base_1\otimes \base_2}^{\Sigma, \dg}\arrow[r] & \cat{BiMod}_{\base_1}^{\Sigma, \dg}
\end{tikzcd}$$
such that for $M_1\in \cat{BiMod}_{\base_1}^{\Sigma, \dg}, M_2\in \cat{BiMod}_{\base_2}^{\Sigma, \dg}$ and $N\in \cat{BiMod}_{\base_1\otimes \base_2}^{\Sigma, \dg}$ there are natural bijections
$$
\Hom\big(M_1, \ul{\Hom}_{\base_2}(M_2, N)\big)\cong \Hom\big(M_1\otimes M_2, N\big)\cong \Hom\big(M_2, \ul{\Hom}_{\base_1}(M_1, N)\big).
$$
We will be interested in applying this to the case where $\base_1=\base$ and $\base_2=\base^\op$ is its opposite.
\begin{definition}\label{def:dual symmetric sequence}
Let $\mm{End}(\base)$ denote the endomorphism operad of $\base$, considered as a left $\base\otimes \base^{\op}$-module. More precisely, $\mm{End}(\base)$ has set of colours $S\times S$ and $p$-ary morphisms $\big((c_1, d_1), \dots, (c_p, d_p)\big)\rt (c_0, d_0)$ given by $k$-linear maps
$$\begin{tikzcd}
\base(c_1, d_1)\otimes \dots \base(c_p, d_p)\arrow[r] & \base(c_0, d_0).
\end{tikzcd}$$
This is a (non-augmented) $\base\otimes \base^{\op}$-operad. We define the \emph{dual} of a symmetric $\base$-bimodule $M$ to be the symmetric $\base^{\op}$-bimodule
$$
M^\vee\coloneqq \ul{\Hom}_{\base}\big(M, \mm{End}(\base)\big).
$$
\end{definition}
Unravelling the definition, one sees that $M^\vee$ is given in arity $p$ by the dual $M^\vee(p)=\Hom_{\base}(M(p), \base)$ with respect to the \emph{left} $\base$-module structure on $M(p)$. The right $\base$-action on $\base$ and the right $\base^{\otimes p}$-action on $M(p)$ endow $M^\vee$ with the structure of a symmetric $\base^{\op}$-module. Explicitly, we have
\begin{equation}\label{eq:dual sequence}
M^\vee(c_1, \dots, c_p; c_0)=\Hom_\base\big(M(c_1, \dots, c_p; -), \base(c_0; -)\big).
\end{equation}
Note that taking duals is only homotopically well-behaved on symmetric $\base$-bimodules that are cofibrant as left $\base$-modules (Definition \ref{def:cofibrantleftmodule}).

\subsubsection*{Cooperads over a dg-category}
\begin{definition}\label{def:cooperad}
	A \emph{$\base$-cooperad} $\CC$ is a coaugmented counital coalgebra in the category $\cat{BiMod}_{\base}^{\Sigma, \dg}$. We will say that $\CC$ is \emph{filtered-cofibrant} as a left $\base$-module if it admits an exhaustive filtration
	$$
	\base=F_0\CC\subseteq F_1\CC\subseteq F_2\CC\subseteq \dots
	$$
	such that $\Delta(F_r\CC)\subseteq \bigoplus_{p+q=r} F_p\CC\circ F_q\CC$ and each $F_r\CC$ is cofibrant as a left $\base$-module. The first condition implies that $\CC$ is conilpotent and the second is equivalent to the associated graded $\mm{gr}(\CC)$ being cofibrant as a left $\base$-module.
\end{definition}
\begin{remark}
Recall that one can always endow a $\base$-cooperad with its coradical filtration, where $F_r\CC=\base\oplus \ker(\ol{\Delta}^r)$. If $\CC(0)=0$ and $\CC(1)=\base$, then $\CC$ is filtered-cofibrant as a left $\base$-module if and only if it is cofibrant as a left $\base$-module, using the filtration by arity.
\end{remark}

\begin{proposition}\label{prop:dual of cooperad is operad}
Let $\CC$ be a $\base$-cooperad and let $\CC^\vee$ be its (left) $\base$-linear dual. Then $\CC^\vee$ has the natural structure of an operad. If $C$ is a $\CC$-coalgebra, then the dual $C^\vee$ has a natural $\CC^\vee$-algebra structure.
\end{proposition}
\begin{proof}
It suffices to verify that the functor $(-)^\vee$ is lax monoidal, in the sense that there is a natural map $M^\vee\circ_{\base^{\op}} N^\vee\rt (M\circ_\base N)^\vee$. This map is the adjoint of
$$\begin{tikzcd}[row sep=0pc, column sep=1.5pc]
\big(M^\vee\circ_{\base^{\op}} N^\vee\big)\otimes \big(M\circ_\base N\big)\arrow[r, "\eqref{eq:hadarmardovercomposition}"] & \big(M^\vee\otimes M\big)\circ_{\base^{\op}\otimes\base} \big(N^\vee\otimes N\big)\\
\phantom{\big(M^\vee\circ_{\base^{\op}} N^\vee\big)\otimes \big(M\circ_\base N\big)} \arrow[r] & \mm{End}(\base)\circ_{\base^{\op}\otimes\base}\mm{End}(\base)\arrow[r] & \mm{End}(\base),
\end{tikzcd}$$
where the second map arises from the evaluation map $\ul{\Hom}_\base(M, \mm{End}(\base))\otimes M\rt \mm{End}(\base)$ and the last map uses that $\mm{End}(\base)$ is a (non-augmented) $\base^{\op}\otimes\base$-operad.
\end{proof}

\subsection{All we need about bar-cobar for operads}\label{sec:barcobar}

From now on, all $\base$-objects (bimodules, operads) that we consider are assumed to be as in Assumption \ref{ass:cofibrancy}: they are cofibrant as left $\base$-modules, and filtered-cofibrant in the case of cooperads.

\begin{definition}[Bar-cobar constructions, see  \cite{ginzburg1994koszul} and Section 6.5 of \cite{LodayVallette2012}]
		Given a $\base$-operad $\mathscr P$, its \textit{bar construction} $\Bar \mathscr P$ is the $\base$-cooperad constructed as the cofree conilpotent  $\base$-cooperad on the augmentation ideal $\ol{\mathscr{P}}[1]$, i.e.\
	$$
	\Bar\mathscr{P} =T^c_{\base}\big(\ol{\mathscr{P}}[1]\big) = \base\oplus \ol{\mathscr{P}}[1]\oplus \ol{\mathscr{P}}[1]\circ_{\base}\ol{\mathscr{P}}[1]\oplus\dots
	$$
	with an additional bar differential given by contraction of trees along inner edges.
	
	Similarly, given a conilpotent $\base$-cooperad $\mathscr C$, its cobar construction $\Omega\mathscr{C}$ is the free graded $\base$-operad on the coaugmentation coideal $\ol{\mathscr{C}}[-1]$, i.e.\
	$$
	\Omega\mathscr{C} =T_{\base}\big(\ol{\mathscr{C}}[-1]\big) = \base\oplus \ol{\mathscr{C}}[-1]\oplus \ol{\mathscr{C}}[-1]\circ_{\base}\ol{\mathscr{C}}[-1]\oplus\dots
	$$
	with an additional cobar differential given by decomposing trees along inner edges.
\end{definition}

\begin{construction}[$\base$-twisting morphisms]\label{constr:k-twisting morphisms}
Let $M$ and $N$ be symmetric $\base$-bimodules. Their infinitesimal composition product $M\circ_{(1)} N$ is the subobject of $M\circ_{\base} N$ given by trees with 2 vertices, with root vertex labeled by $M$ and the other vertex labeled by $N$. There is a natural retraction
$$
M\circ_{(1)} N\longrightarrow M\circ_{\base} N\longrightarrow M\circ_{(1)} N
$$
where the projection quotients out trees labeled by $M$ and $N$ with more than two vertices.

Let $\mathscr{P}$ be a $\base$-operad and $\mathscr{C}$ a conilpotent $\base$-cooperad. A \emph{twisting morphism} $\phi\colon \CC\drt \PP$ is a map of symmetric $\base$-bimodules of cohomological degree $1$, which vanishes both after composing with the augmentation and coaugmentation map, such that:

\begin{equation}\label{eq:MC on def complexes}
d\phi + \phi\star \phi = 0
\end{equation}
where $\phi\star \phi$ is the composite
$$
\mathscr{C}\longrightarrow \mathscr{C}\circ_{\base} \mathscr{C}\longrightarrow \mathscr{C}\circ_1 \mathscr{C} \xrightarrow{\ \phi\circ_1 \phi\ } \mathscr{P}\circ_1 \mathscr{P}\longrightarrow\mathscr{P}\circ_{\base} \mathscr{P}\longrightarrow \mathscr{P},
$$
and $d$ denotes the commutator of differentials in $\Hom_{\cat{BiMod}_{\base}^{\Sigma, \dg}}(\mathscr C, \mathscr P)$. 
We denote by $\Tw(\mathscr{C}, \mathscr{P})\subset \Hom_{\cat{BiMod}_{\base}^{\Sigma, \dg}}(\mathscr C, \mathscr P)$ the set of twisting morphisms.
\end{construction}
\begin{remark}\label{rem:convolutionoperad}
Similar to \cite[Proposition 6.4.3]{LodayVallette2012}, one checks that the sequence of complexes
$$
\Conv(\CC, \PP)(p)\coloneqq \Hom_{\base\otimes (\base^{\op})^{\otimes p}}\big(\CC(p), \PP(p)\big)
$$
has the structure of an (ordinary) operad in chain complexes, called the \emph{convolution operad}. As in \cite[Proposition 6.4.5]{LodayVallette2012}, it follows that $\big(\Hom_{\cat{BiMod}_{\base}^{\Sigma, \dg}}(\mathscr C, \mathscr P),\star,d\big)$ is a pre-Lie algebra and the twisting morphisms are its Maurer--Cartan elements. If $\CC$ or $\PP$ is 1-reduced, i.e.\ zero in arity $0$ and $\base$ in arity $1$, then \cite[Section 7]{Wierstra} shows that such Maurer--Cartan elements correspond bijectively to maps of operads
$$\begin{tikzcd}
L_\infty\{-1\}\arrow[r] & \Conv(\CC, \PP)
\end{tikzcd}$$
from the operadic suspension of the $L_\infty$-operad: the value on the generating $p$-ary operation $l_p$ of $L_\infty\{-1\}$ is given by $\phi_p\colon \CC(p)\rt \PP(p)$.
\end{remark}
\begin{proposition}\label{prop:Tw adjunction}
Let $\mathscr{C}$ be a conilpotent $\base$-cooperad and $\mathscr{P}$ a $\base$-operad. Then there are natural bijections
$$
\Hom_{\cat{CoOp}^{\dg}_{\base}}\big(\mathscr{C}, \Bar\mathscr{P}\big)\cong \Tw(\mathscr{C}, \mathscr{P})\cong \Hom_{\cat{Op}^{\dg}_{\base}}\big(\Omega\mathscr{C}, \mathscr{P}\big).
$$
\end{proposition}
\begin{proof}
Maps of bimodules $\varphi \colon \mathscr C \to \mathscr D$ which vanish both when composed with the augmentation and coaugmentation map are in one-to-one correspondence with maps of augmented operads from the free operad generated by $\ol{\mathscr C}$ to $\mathscr D$. One can check that the compatibility with the differentials is given exactly by equation \eqref{eq:MC on def complexes}.
A dual argument on the category of conilpotent cooperads shows that $\Hom_{\cat{CoOp}^{\dg}_{\base}}\big(\mathscr{C}, \Bar\mathscr{P}\big)\cong \Tw(\mathscr{C}, \mathscr{P})$, see \cite[Theorem 6.5.7]{LodayVallette2012} for the case $\base=k$.
\end{proof}

\begin{lemma}\label{lem:operadicbarpreservesqisos}
Let $\mathscr{P}\to \mathscr{Q}$ be a quasi-isomorphism between two $\base$-operads which are cofibrant as left $\base$-modules. Then the map $\Bar\mathscr{P}\rt \Bar\mathscr{Q}$ is a quasi-isomorphism of $\base$-cooperads, which are filtered-cofibrant as  left $\base$-modules.
\end{lemma}
\begin{proof}
Endow both bar constructions with the (exhaustive) filtration by word length in $\ol{\mathscr{P}}$ and $\ol{\mathscr{Q}}$. The map on the associated graded is just the map $T^c(\ol{\mathscr{P}}[1])\rt T^c(\ol{\mathscr{Q}}[1])$. 
When $\PP$ and $\QQ$ are cofibrant as left $\base$-modules, these associated gradeds are cofibrant as left $\base$-modules, so that $\Bar\PP$ and $\Bar\QQ$ are filtered-cofibrant.
Using Lemma \ref{lem:circ preserves qi}, we conclude that the map at the level of the associated graded is a quasi-isomorphism.
\end{proof}
\begin{proposition}\label{prop:counit bar cobar}
Let $\mathscr{P}$ be a $\base$-operad which is cofibrant as a left $\base$-module. Then the counit of the bar-cobar adjunction $\Omega \Bar\mathscr{P}\rt \mathscr{P}$ is a quasi-isomorphism.
\end{proposition}
\begin{proof}
Ignoring degrees, elements of $\Omega \Bar\mathscr{P}$ can be seen as trees whose vertices are themselves (``inner'') trees whose vertices are labeled by $\mathscr P$. Filtering by the number of inner edges (bar word length) and using the cofibrancy of $\mathscr P$ as a left $\base$-module we recover at the level of the associated graded only the piece of the differential corresponding to the one from $\mathscr P$ and a second one making an inner edges into an outer edge. 

One checks that the associated graded retracts into $\mathscr P$ by constructing a homotopy that makes an outer edge into an inner edge.
\end{proof}
\begin{definition}[Twisted composition products]\label{def:twisted comp product} Given a twisting morphism $\phi\colon \CC\drt\PP$, the \emph{twisted composition product} $\CC\circ_\phi \PP$ \cite[Section 6.4.11]{LodayVallette2012} is the symmetric $\base$-bimodule $\CC\circ_\base \PP$, but with differential twisted by the map
$$\begin{tikzcd}[column sep=1.1pc]
\CC\circ_\base\PP\arrow[rr, "\Delta_{(1)}\circ 1"] & & \big(\CC\circ_{(1)}\CC\big)\circ_\base \PP\hspace{2pt}\arrow[r, hookrightarrow] & \CC\circ_\base\CC\circ_\base \PP \arrow[rr, "1\circ\phi\circ 1"] & & \CC\circ_\base\PP\circ_\base\PP\arrow[rr, "1\circ\mu"] & & \CC\circ_\base\PP.
\end{tikzcd}$$
Similarly, the twisted composition product $\PP\circ_\phi \CC$ has differential twisted by
$$\begin{tikzcd}[column sep=1.1pc]
\PP\circ_\base \CC\arrow[rr, "1\circ \Delta"] & & \PP\circ_\base\CC\circ_\base\CC\arrow[rr, "1\circ\phi\circ 1"] & & \PP\circ_\base\PP\circ_\base \CC\arrow[r, two heads] & \big(\PP\circ_{(1)} \PP\big)\circ_\base\CC\arrow[rr, "\mu_{(1)}\circ 1"] & & \PP\circ_\base\CC.
\end{tikzcd}$$
\end{definition}
\begin{example}\label{ex:twisted comp universal}
For the universal twisting morphism $\pi\colon\Bar\PP\drt \PP$, elements of $\Bar\PP\circ_\pi \PP$ can be identified with trees whose vertices are labeled by elements in $\ol{\PP}[1]$, or by elements of $\ol{\PP}$ for (some of the) leaf vertices. The differential then has three parts: (a) applying the differential of $\ol{\PP}$ to vertices, (b) contracting inner edges between $\ol{\PP}[1]$-labeled trees and (c) replacing an $\ol{\PP}[1]$-labeled vertex with only $\PP$-labeled vertices above it by a $\ol{\PP}$-labeled vertex and contracting (at the same time) all inner edges above it.

Similarly, $\Omega\CC\circ_\iota \CC$ consists of trees with vertices labeled by $\ol{\CC}[-1]$, or by $\ol{\CC}$ for (some of the) leaf vertices, with differential having three terms: (a) applying the differential of $\ol{\CC}$, (b) partially decomposing along inner edges between $\ol{\CC}[-1]$-labeled vertices and \hypertarget{it:twistedcompdifferential}{(c)} decomposing a $\ol{\CC}$-labeled leaf vertex into height 2 trees with root vertex labeled by $\ol{\CC}[-1]$.
\end{example}
\begin{lemma}\label{lem:twistedcomp}
Let $\phi\colon \CC\drt\PP$ be a twisting morphism, where $\CC$ and $\PP$ are filtered-cofibrant, resp.\ cofibrant as left $\base$-modules.
\begin{enumerate}
\item\label{it:twisted comp is flat} Let $M\rt N$ be a quasi-isomorphism between left $\PP$-modules that are cofibrant as left $\base$-modules. Then $(\CC\circ_\phi \PP)\circ_{\PP}M\rt (\CC\circ_\phi \PP)\circ_{\PP}N$ is a quasi-isomorphism between filtered-cofibrant left $\base$-modules.

\item\label{it:twisted comp is flat 2} Let $M\rt N$ be a quasi-isomorphism between right $\PP$-modules. Then $M\circ_{\PP}(\PP\circ_\phi \CC)\rt N\circ_{\PP}(\PP\circ_\phi \CC)$ is a quasi-isomorphism.

\item\label{it:twisted comp is acyclic} The maps $\Bar\PP\circ_\pi \PP\rt \base$ and $\Omega\CC\circ_\iota \CC\rt \base$ are quasi-isomorphisms.
\end{enumerate}
\end{lemma}
\begin{proof}
For \ref{it:twisted comp is flat}, filter $\CC\circ_\phi \PP$ using the filtration on $\CC$. The associated graded is $\mm{gr}(\CC)\circ_\base \PP$. The map $(\CC\circ_\phi \PP)\circ_{\PP}M\rt (\CC\circ_\phi \PP)\circ_{\PP}N$ preserves the induced filtrations and is given on the associated graded by $\mm{gr}(\CC)\circ_\base M\rt\mm{gr}(\CC)\circ_\base N$. This is a quasi-isomorphism by Lemma \ref{lem:circ preserves qi}. The same argument applies to \ref{it:twisted comp is flat 2}.

For \ref{it:twisted comp is acyclic}, filter $\Bar\PP\circ_\pi \PP$ by the number of inner edges. On the associated graded, one can then construct a contracting homotopy replacing a $\ol{\PP}$-labeled leaf vertex by a $\ol{\PP}[1]$-labeled leaf vertex. 

Similarly, the filtration on $\CC$ induces a total filtration on $\Omega\CC\circ_\psi \CC$. The associated graded consists of trees with vertices labeled by the associated graded $\mm{gr}(\ol{\CC})[-1]$, or $\mm{gr}(\ol{\CC})$ for (some) leaf vertices. Since the cocomposition vanishes on $\mm{gr}(\ol{\CC})$, the differential has two remaining contributions: (a) the differential on $\mm{gr}(\ol{\CC})$ and its shift and \hyperlink{it:twistedcompdifferential}{(c)} sending a $\mm{gr}(\ol{\CC})$-labeled leaf vertex to the corresponding $\mm{gr}(\ol{\CC})[-1]$-labeled vertex. This has a contracting homotopy by replacing $\mm{gr}(\ol{\CC})[-1]$-labeled leaf vertices by $\mm{gr}(\ol{\CC})$-labeled leaf vertices.
\end{proof}
\begin{corollary}\label{cor:bar of operad as derived comp}
Let $\PP$ be a $\base$-operad which is cofibrant as a left $\base$-module. Then $\Bar\PP\simeq \base\circ^h_\PP\base$. 
\end{corollary}

\subsection{All we need about bar-cobar for algebras}\label{sec:barcobaralgebras}

Let $\mathscr C$ be a $\base$-cooperad and $\mathscr P$ a $\base$-operad, which are filtered-cofibrant, resp.\ cofibrant, as left $\base$-modules.

\begin{definition}\label{def:koszultwisting}
A twisting morphism $\phi\colon \mathscr{C}\drt \mathscr{P}$ is said to be \emph{Koszul} if $\phi$ induces a quasi-isomorphism $\Omega\mathscr{C}\rt \mathscr{P}$.

We will say that it is \emph{weakly Koszul} if instead the map $\mathscr{C}\rt \Bar\mathscr{P}$ is a quasi-isomorphism. Since the bar construction preserves quasi-isomorphisms (Lemma \ref{lem:operadicbarpreservesqisos}), Koszul morphisms induce a quasi-isomorphism $\Bar \Omega \mathscr C\xrightarrow{\sim} \Bar \mathscr P$ and are therefore weakly Koszul.
\end{definition}
\begin{definition}
Let $\phi\colon \mathscr{C}\drt \mathscr{P}$ be a twisting morphism, $C$ a $\mathscr{C}$-coalgebra (in left $\base$-modules) and $A$ a $\mathscr{P}$-algebra (in left $\base$-modules). A twisting morphism $f\colon C\rt A$ over $\phi$ is a left $\base$-linear map of degree $0$ satisfying
$$
d f+ \phi\circ f =0
$$
where $\phi\circ f: C\rt A$ is given by 
$$
C\rt \mathscr{C}\circ_{\base} C\rto{\phi\circ f} \mathscr{P}\circ_{\base} A \rt A.
$$
We denote by $\Tw_\phi(C, A)$ the set of twisting morphisms over $\phi$.
\end{definition}
\begin{remark}\label{rem:convolutionalgebra}
If $C$ is a conilpotent $\CC$-coalgebra and $A$ is a $\PP$-algebra, then one can check that the complex $\Hom_{\base}(C, A)$ has the structure of an algebra over the convolution operad $\Conv(\CC, \PP)$ of Remark \ref{rem:convolutionoperad} \cite[Proposition 7.1]{Wierstra}. 

If $\CC$ or $\PP$ is 1-reduced, then a twisting morphism $\phi$ determines a map $L_\infty\{-1\}\rt \Conv(\CC, \PP)$, so that $\Hom_{\base}(C, A)$ has a shifted $L_\infty$-structure. As in loc.\ cit.\ the value $l_p(f_1, \dots, f_p)$ of the generating $p$-ary operation $l_p$ in $L_\infty\{-1\}$ is given by
$$\begin{tikzcd}[column sep=2.6pc]
C\arrow[r] & \CC(p)\otimes_{\base^{\otimes p}} C^{\otimes p}\arrow[rrr, "\sum_\sigma \phi(p)\otimes f_{\sigma(1)}\otimes\dots \otimes f_{\sigma(p)}"] & & & \PP(p)\otimes_{\base^{\otimes p}} A^{\otimes p}\arrow[r] & A,
\end{tikzcd}$$ 
where the sum runs over $\sigma\in \Sigma_p$. The twisting morphisms $f\colon C\rt A$ are exactly the degree $1$ elements of this $L_\infty$-algebra satisfying the Maurer--Cartan equation $\sum_n \frac{1}{n!}l_n(f, \dots, f)=0$ \cite[Theorem 7.1]{Wierstra}. Note that the infinite sum becomes finite when evaluated at some $c\in C$, because $C$ is a conilpotent $\CC$-coalgebra.
\end{remark}

\begin{definition}[Bar-cobar construction for algebras]\label{def:bar-cobar algebras}
Given a twisting morphism $\varphi \colon \mathscr C \to \mathscr P$ and $C$ a $\mathscr{C}$-coalgebra, we define the cobar construction $\Omega_\phi C$, to be the free $\mathscr{P}$-algebra on $C$, $\mathscr P \circ_\base C$, with differential given on generators by $d(c) = d_C(c) + \delta(c)$ with $\delta\colon C\rt \mathscr{C}\circ_\base C\rt \mathscr{P}\circ_\base C$.

 Similarly, given $A$ a $\mathscr{P}$-algebra, its bar construction $\Bar_\phi A$ is the cofree $\mathscr{C}$-coalgebra on $A$, $\mathscr C \circ_\base A$, with differential given by $d_A+\delta$ with $\delta$ onto generators given by $\mathscr{C}\circ_\base A\rt \mathscr{P}\circ_\base A\rt A$.
\end{definition}
\begin{remark}\label{rem:bar of algebra as derived comp}
One can also identify using twisted composition products (Definition \ref{def:twisted comp product}) as $\Omega_\phi C\cong \big(\PP\circ_\phi \CC\big)\circ^{\CC} C$ and $\Bar_\phi=\big(\CC\circ_\phi\PP\big)\circ_\PP A$.  In particular, if $\pi\colon \Bar\PP\rt \PP$ is the universal twisting morphism, then Lemma \ref{lem:twistedcomp} shows that for every $\PP$-algebra which is cofibrant as a left $\base$-module,
$$
\Bar_\pi A\cong \big(\Bar\PP\circ_\pi \PP\big)\circ_\PP A \simeq \base\circ_\PP^h A.
$$
\end{remark}

\begin{proposition}\label{prop:cobardjointtobar}
There are natural bijections
$$
\Hom_{\cat{Alg}^{\dg}_{\mathscr{P}}}\big(\Omega_\phi C, A\big) \cong \Tw_\phi(C, A)\cong \Hom_{\cat{CoAlg}^{\dg}_{\mathscr{C}}}\big(C, \Bar_\phi A\big).
$$
\end{proposition}

\begin{proof}
	The proof is similar to Proposition \ref{prop:Tw adjunction}, see also \cite[Proposition 11.3.1]{LodayVallette2012}.
\end{proof}

\begin{lemma}\label{lem:algbarpreswe}
Let $\phi\colon \mathscr{C}\drt \mathscr{P}$ be a twisting morphism and $A$ a $\mathscr P$-algebra. Then:
\begin{enumerate}
\item $\Bar_\phi$ preserves quasi-isomorphisms between $\mathscr{P}$-algebras that are cofibrant as left $\base$-modules.
\item\label{it:barofcofibrant} If $A$ is cofibrant as a left $\base$-module and $\cat{C}$ is filtered-cofibrant as a left $\base$-module, then $\Bar_\phi(A)$ is filtered-cofibrant as a left $\base$-module.
\item In the setting of \ref{it:barofcofibrant}, $\Omega_\phi \Bar_\phi(A)$ is a cofibrant $\PP$-algebra.
\end{enumerate} 
\end{lemma}
\begin{proof}
The first two points follow from Lemma \ref{lem:twistedcomp} and Remark \ref{rem:bar of algebra as derived comp}. In particular, the proof of Lemma \ref{lem:twistedcomp} shows that $\Bar_\phi(A) = (\mathscr C \circ_\base A, d_A + d_\Bar)$ carries a filtration induced from the filtration on $\CC$.

For the third point, note that $\Omega_\phi\Bar_\phi(A)$ inherits a filtration by subalgebras from the filtration on $\Bar_\phi(A)$. Since $\mm{gr}(\Bar_\phi(A))$ is a trivial coalgebra, $\mm{gr}(\Omega_\phi\Bar_\phi(A))$ is the free $\PP$-algebra on $\mm{gr}(\CC)\circ_\base A$. Since $\mm{gr}(\CC)\circ_\base A$ is cofibrant as a graded left $\base$-module, an inductive argument shows that $\Omega_\phi\Bar_\phi(A)$ is cofibrant (see also \cite[Proposition 2.8]{vallette2014homotopy}).
\end{proof}

\begin{lemma}\label{lem:algbarresol}
	\leavevmode
	\begin{enumerate}
\item\label{it:algebra bar-cobar universal} Let $\CC$ be filtered-cofibrant as a left $\base$-module and let $\iota\colon \CC\rt \Omega\CC$ be the universal twisting morphism. Then the counit $\Omega_\iota \Bar_\iota A\rt A$ is a quasi-isomorphism for all $A\in\Alg_{\Omega\CC}$ which are cofibrant as left $\base$-modules.

\item\label{it:algebra bar-cobar koszul} Let $\phi \colon \mathscr C \to \mathscr P$ be a Koszul twisting morphism. Then $\Omega_\phi \Bar_\phi B\rt B$ is a quasi-isomorphism for all $B\in\Alg_{\PP}$ which are cofibrant as left $\base$-modules.
\end{enumerate}

\end{lemma}
\begin{proof}
For \ref{it:algebra bar-cobar universal}, note that $\Omega_\iota \Bar_\iota A$ consists of trees with vertices labeled by $\ol{\CC}[-1]$ or by $\ol{\CC}$ for (some of the) leaf vertices, and with leaves labeled by $A$. The differential has a contribution from the differential on $\Omega\CC\circ_\iota \CC$ (Example \ref{ex:twisted comp universal}) and a contribution by letting $\CC$-labeled leaf vertices act on their leaves. Filtering $\Omega_\iota \Bar_\iota A$ by the number of leaves, the associated graded is $\big(\Omega\CC\circ_\iota \CC\big)\circ_\base A$. The result then follows from Lemma \ref{lem:twistedcomp}.

For \ref{it:algebra bar-cobar koszul}, let $f\colon \Omega\CC\rt\PP$ be the induced map and notice that $\Omega_\phi \Bar_\phi B = f_!\big(\Omega_\iota \Bar_\iota (f^* B)\big)$. 
The result then follows from part \ref{it:algebra bar-cobar universal}, $(f_!, f^*)$ being a Quillen equivalence (Corollary \ref{cor:quasiiso gives equivalence}) and $\Omega_\iota \Bar_\iota (f^* B)$ being cofibrant (Lemma \ref{lem:algbarpreswe}).
\end{proof}

\subsection{Free resolutions of operads}\label{sec:freeoperads}

The remainder of this section is devoted to a proof of the following result, relating the homotopy-invariant condition appearing in Theorem \ref{thm:mainthm} to a more concrete condition in terms of quasi-free resolutions:
\begin{proposition}\label{prop:generatorsforoperad}
Let $\mathscr{P}$ be a connective 0-reduced $\base$-operad. Then the following are equivalent:
\begin{enumerate}
\item the symmetric sequence $\mathscr{P}^{\leq 1}\circ^h_{\mathscr{P}} \mathscr{P}^{\leq 1}$ is eventually highly connective.\footnote{Here we use the natural left and right actions of $\mathscr P$ on its quotient $\mathscr P^{\leq 1}$.}
\item $\mathscr{P}$ is quasi-isomorphic to a quasi-free, non-positively graded $\base$-operad with higher arity generators in increasingly negative degrees. More precisely, for every $n\in\mathbb{Z}$, there exists a $p(n)\in\mathbb{N}$ such that all generators of arity $\geq p(n)$ are in cohomological degrees $< n$.
\end{enumerate}
\end{proposition}
\begin{remark}
Recall that every cofibrant $\base$-operad is the retract of an operad which is quasi-freely generated by a ($S$-coloured) symmetric sequence of graded vector spaces (one can take for instance its cobar-bar construction). Conversely, if $\mathscr{P}$ is quasi-freely generated by a symmetric sequence of graded vector spaces in nonpositive degree, then $\mathscr{P}$ is cofibrant.\footnote{More generally, a triangulated quasi-free operad is cofibrant, see \cite[Proposition B.6.10]{LodayVallette2012}.}
\end{remark}

For the remaining of the section, all (co)operads are 0-reduced (trivial in arity zero). We will make use of the following Quillen adjunction between the categories of 0-reduced (augmented) $\base$-operads
$$\begin{tikzcd}
\mm{sk}_p\colon \cat{Op}_{\base}^{\mm{nu},\dg}\arrow[r, yshift=0.8ex] & \cat{Op}_{\base}^{\mm{nu},\dg}\arrow[l, yshift=-0.8ex]\colon (-)^{\leq p}
\end{tikzcd}$$

The right adjoint is the ``truncation to arity at most $p$'' functor that quotients an operad $\mathscr P$ by the operadic ideal $\bigoplus_{k>p} \mathscr P(k)$. 
Its left adjoint is the ``$p$-skeleton'' functor that associates to $\mathscr Q$ the operad $\mm{sk}_p(\mathscr Q)$ which is given in arities $\leq p$ by $\mathscr{Q}$, and which is freely generated by this data. 

\begin{remark}\label{rem:skeletonofalmostquasiiso}
A map $f\colon \PP\rt \QQ$ between cofibrant operads induces a quasi-isomorphism $\mm{sk}_p\PP\rt \mm{sk}_p\QQ$ as soon as it induces a quasi-isomorphism in arity $\leq p$. Indeed, factor $f$ as an acyclic cofibration $\PP\rt \PP'$ followed by a fibration $f\colon \PP'\rt \QQ$ and use Lemma \ref{lem:addgenerators} to resolve $f'$ by a map which is an isomorphism in arities $\leq p$. The result then follows from the fact that $\mm{sk}_p$ is a left Quillen functor and which only depends on arity $(\leq p)$-parts.
\end{remark}
\begin{lemma}\label{lem:skeletarelbar}
Let $\mathscr{P}$ be a cofibrant 0-reduced $\base$-operad, let $p\geq 1$ and consider the cofiber sequences
$$\begin{tikzcd}[row sep=0.4pc]
\mm{sk}_p(\mathscr{P})\arrow[r] & \mathscr{P}\arrow[r] & X\\
\mathscr{P}^{\leq 1}\circ^h_{\mm{sk}_p(\mathscr{P})} \mathscr{P}^{\leq 1}\arrow[r] & \mathscr{P}^{\leq 1}\circ^h_{\mathscr{P}} \mathscr{P}^{\leq 1} \arrow[r] & Y.
\end{tikzcd}$$
There is a natural map $X\rt Y[-1]$, which is an equivalence in arity $p+1$. Furthermore, the map $\mathscr{P}^{\leq 1}\circ^h_{\mathscr{P}} \mathscr{P}^{\leq 1} \rt Y$ is an equivalence in arity $p+1$ as well.
\end{lemma}
\begin{proof}
Let $\mathscr{P}^{\geq 2}$ denote the kernel of the quotient $\mathscr{P}\rt \mathscr{P}^{\leq 1}$, so that there is a cofiber sequence
\begin{equation}\label{diag:cof1}\begin{tikzcd}
\mathscr{P}^{\geq 2}\circ_{\mathscr{P}}^h \mathscr{P}^{\leq 1}\arrow[r] & \mathscr{P}^{\leq 1}\arrow[r] &\mathscr{P}^{\leq 1}\circ^h_{\mathscr{P}}\mathscr{P}^{\leq 1}.
\end{tikzcd}\end{equation}
Using the same cofiber sequence for $\mm{sk}_p(\mathscr{P})$ and unraveling the definitions, one sees that there is a natural cofiber sequence
$$\begin{tikzcd}
\mm{sk}_p(\mathscr{P})^{\geq 2}\circ_{\mm{sk}_p(\mathscr{P})}^h \mathscr{P}^{\leq 1}\arrow[r] & \mathscr{P}^{\geq 2}\circ_{\mathscr{P}}^h\mathscr{P}^{\leq 1}\arrow[r] & Y[-1].
\end{tikzcd}$$
There is a natural map $\mathscr{P}\rt \mathscr{P}^{\geq 2}\rt \mathscr{P}^{\geq 2}\circ^h_{\mathscr{P}}\mathscr{P}^{\leq 1}$ (the first one quotients out the arity $1$ part), and similarly for $\mm{sk}_p(\mathscr{P})$. The desired map $X\rt Y[-1]$ is the induced map on cofibers.

Now suppose that $\mathscr{P}=\big(\mm{Free}_{\cat{ Op}_\base}(V), d\big)$ is a cofibrant $\base$-operad, quasi-freely generated by a symmetric $\base$-bimodule $V$. Then $\mathscr{P}^{\geq 2}$ is a cofibrant right $\mathscr{P}$-module, given by $\mathscr{P}(1)\circ V^{\geq 2}\circ \mathscr{P}$ (with some differential), where $V^{\geq 2}$ is the arity $\geq 2$ piece of $V$. It follows that
\begin{equation}\label{eq:indec}
\mathscr{P}^{\geq 2}\circ_{\mathscr{P}} \mathscr{P}^{\leq 1}\simeq \mathscr{P}^{\leq 1}\circ V^{\geq 2}\circ \mathscr{P}^{\leq 1}
\end{equation}
with some differential. A similar equivalence holds for $\mm{sk}_p(\mathscr{P})$, which is a cofibrant suboperad of $\PP$ freely generated by $V^{\leq p}$, the arity $\leq p$ piece of $V$. One then deduces that
$$
Y[-1]\simeq \mathscr{P}(1)\circ V^{\geq p+1}\circ \mathscr{P}(1).
$$
In particular, it agrees with $\mathscr{P}^{\geq 2}\circ^h_{\mathscr{P}} \mathscr{P}(1)$ in arity $p+1$. Note that the above symmetric sequence consists exactly of the $(p+1)$-ary operations of $\mathscr{P}$, modulo those that are compositions of $(\leq p)$-ary operations. This is exactly the $(p+1)$-ary part of the cofiber $X$.
\end{proof}
\begin{corollary}\label{cor:equivoncot}
Let $f\colon \mathscr{P}\rt \mathscr{Q}$ be a map of connective 0-reduced $\base$-operads such that the map of symmetric $\base$-bimodules 
$$
\mathscr{P}(1)\circ^h_{\mathscr{P}} \mathscr{P}(1)\rt \mathscr{Q}(1)\circ^h_{\mathscr{Q}} \mathscr{Q}(1)
$$
is a quasi-isomorphism. Then $f$ is a quasi-isomorphism.
\end{corollary}
\begin{proof}
We may assume that $\PP$ and $\QQ$ are cofibrant. In that case, the map $f$ induces a quasi-isomorphism in arity $\leq p$ if and only if the induced maps on $p$-skeleta $\mm{sk}_p\PP\rt \mm{sk}_p\QQ$ is a quasi-isomorphism (Remark \ref{rem:skeletonofalmostquasiiso}). We check this for all $p$ by induction. For $p=1$, note that $\mathscr{P}(1)\circ^h_{\mathscr{P}} \mathscr{P}(1)$ is given in arity 1 by $\PP(1)$; this follows from the cofiber sequence \ref{diag:cof1} and equation \ref{eq:indec}. 

Next, notice that the arity $(p+1)$-part of $\PP$ is quasi-isomorphic to the arity $(p+1)$-part of the cofiber $Y$ from Lemma \ref{lem:skeletarelbar}. If $f$ induces a quasi-isomorphism on $p$-skeleta, then this cofiber is quasi-isomorphic to the corresponding cofiber for $\QQ$. It then follows that $f$ also induces a quasi-isomorphism on $(p+1)$-skeleta.
\end{proof}

\begin{proof}[Proof of Proposition \ref{prop:generatorsforoperad}]
(2) $\Rightarrow$ (1): follows from the cofiber sequence \eqref{diag:cof1} and the identification \eqref{eq:indec}.

(1) $\Rightarrow$ (2): we can assume that $\PP$ is cofibrant to begin with. It then suffices to show that $\mathscr{P}$ admits a free resolution with all generators in degrees $\leq 0$ and with the following property at each arity $p\geq 2$:
\begin{itemize}[leftmargin=*]
\item[] If $H^*\big(\mathscr{P}^{\leq 1}\circ_{\mathscr{P}}\mathscr{P}^{\leq 1}\big)(p)$ is concentrated in degrees $\leq n(p)$, then the generators of arity $p$ are concentrated in degrees $\leq n(p)+1$. 
\end{itemize} 
We construct this resolution by induction on skeleta, using that
$$
\mathscr{P}^{\leq 1}\circ^h_{\mm{sk}_p(\mathscr{P})} \mathscr{P}^{\leq 1}\simeq \Big(\mathscr{P}^{\leq 1}\circ^h_{\mathscr{P}}\mathscr{P}^{\leq 1}\Big)^{\leq p}.
$$
For the $1$-skeleton $\mathscr{P}^{\leq 1}=\mm{sk}_1(\mathscr{P})$, there is no condition. Suppose we have found the desired presentation for $\mm{sk}_{p-1}(\mathscr{P})$. It follows from Lemma \ref{lem:skeletarelbar} that in arity $p$, the cohomology of the cofiber $\mm{sk}_p\mathscr{P}/\mm{sk}_{p-1}(\mathscr{P})$ is concentrated in degrees $\leq n(p)+1$ (and also in degrees $\leq 0$). This means that $\mm{sk}_p\mathscr{P}$ can be obtained from $\mm{sk}_{p-1}(\mathscr{P})$ by adding arity $p$ generators of degree $\leq n(p)+1$ (as well as generators of higher arity).
\end{proof}
In Section \ref{sec:cohsmall}, we will need a slight refinement of Proposition \ref{prop:generatorsforoperad} which provides a quasi-free resolution of the entire tower $\PP\rt \dots\rt \PP^{\leq n}\rt \dots$.
\begin{lemma}\label{lem:addgenerators}
Let $f\colon \mathscr{P}\rt \mathscr{Q}$ be a fibration of connective operads such that $f$ induces a trivial fibration in arities $\leq p$. For any cofibrant resolution $\tilde{\mathscr{Q}}\rto{\sim} \mathscr{Q}$, there exists a cofibrant resolution $\tilde{\mathscr P}$ of $\mathscr{P}$ which fits into a diagram
$$\begin{tikzcd}
\tilde{\mathscr{P}}\arrow[d, "\sim"{swap}]\arrow[r, "\tilde{f}"] & \tilde{\mathscr{Q}}\arrow[d, "\sim"]\\
\mathscr{P}\arrow[r, "f"{swap}] & \mathscr{Q}
\end{tikzcd}$$
such that $\tilde{f}$ induces an isomorphism in arities $\leq p$.
\end{lemma}
\begin{proof}
	Since $\tilde{\mathscr Q}$ is cofibrant, there exists a lift
	
	$$\begin{tikzcd}
		& (\mathscr{P}\times_{\mathscr{Q}} \tilde{\mathscr{Q}})^{\leq p}\arrow[d, two heads, "\sim"]\\
		\tilde{\mathscr{Q}}\arrow[r]\arrow[ru, dotted] & (\tilde{\mathscr{Q}})^{\leq p}
	\end{tikzcd}$$
	and therefore, by adjuction we have a lift
$$\begin{tikzcd}
 & \mathscr{P}\times_{\mathscr{Q}} \tilde{\mathscr{Q}}\arrow[d]\\
\mm{sk}_p\tilde{\mathscr{Q}}\arrow[r]\arrow[ru, dotted, "g"{above}] & \tilde{\mathscr{Q}}.
\end{tikzcd}$$
 We can now factor the map $g$ into a cofibration followed by a weak equivalence $\mm{sk}_p\tilde{\mathscr{Q}}\hookrightarrow \tilde{\mathscr{P}}\stackrel{\sim}{\twoheadrightarrow} \mathscr P \times_{\mathscr Q} \tilde{\mathscr Q}$.

 Since all operads involved are connective, this can be done inductively by `adding cells to kill a cycle'. As $g$ is already a weak equivalence in arity $\leq p$, it suffices to add cells in arity $\geq p+1$, which does not change the arity $\leq p$ part. In particular, the composite map $\tilde{\mathscr{P}}\rt \tilde{\mathscr{Q}}$ induces a weak equivalence in arities $\leq p$.
\end{proof}
\begin{proposition}\label{prop:genfortruncations}
Let $\mathscr{P}$ be a connective $\base$-operad and consider the tower of $\base$-operads
$$\begin{tikzcd}
\mathscr{P}\arrow[r] & \dots \arrow[r] & \mathscr{P}^{\leq p}\arrow[r] & \mathscr{P}^{\leq p-1}\arrow[r] & \dots \arrow[r] & \mathscr{P}^{\leq 1}. 
\end{tikzcd}$$
Then there exists a resolution of this tower by a tower of quasi-free, non-positively graded $\base$-operads ${\mathscr{Q}}\rt \dots \rt {\mathscr{Q}}^{(p)}\rt \dots$ with the following properties:
\begin{enumerate}[label=(\alph*)]
\item\label{it:resola} Each ${\mathscr{Q}}^{(p)}\rt {\mathscr{Q}}^{(p-1)}$ induces an isomorphism in arity $\leq p-1$.
\item\label{it:resolb} Each ${\mathscr{Q}}^{(p)}$ has higher arity generators in increasingly negative degrees (in the sense of Proposition \ref{prop:generatorsforoperad}).
\item\label{it:resolc} ${\mathscr{Q}}$ is the limit of the tower.
\item\label{it:resold} For each $p\geq 2$, the generators of $\mathscr{Q}^{(p)}$ in arity $p$ are concentrated in degrees $\leq n(p)+1$, where $n(p)$ is such that
$$
H^*\big(\mathscr{P}^{\leq 1}\circ_{\mathscr{P}} \mathscr{P}^{\leq 1}\big)(p) = 0 \qquad\qquad *> n(p).
$$
\end{enumerate}
In particular, ${\mathscr{Q}}$ is a graded-free resolution of $\mathscr{P}$ with higher arity generators in increasingly negative degrees.
\end{proposition}
\begin{proof}
We can assume from the start that $\PP$ is already cofibrant, and then construct such a tower of free resolutions inductively, as in Lemma \ref{lem:addgenerators}. In each inductive step, it suffices to add generators of arity $\geq p$ to $\mm{sk}_{p-1}\big({\mathscr{Q}}^{(p-1)}\big)$. In particular, we can always arrange for condition \ref{it:resola}.

To see what kind of generators have to be added in arity $p$, note that there is a quasi-isomorphism 
$$\begin{tikzcd}
\mm{sk}_{p-1}\big({\mathscr{Q}}^{(p-1)}\big)\arrow[r] & \mm{sk}_{p-1}(\mathscr{P}^{\leq p})\cong \mm{sk}_{p-1}(\PP)
\end{tikzcd}$$
since both are quasi-isomorphic in arities $\leq p-1$ (Remark \ref{rem:skeletonofalmostquasiiso}). One deduces that the cofiber of $\mm{sk}_{p-1}\big({\mathscr{Q}}^{(p-1)}\big)\rt \mathscr{P}^{\leq p}$ is given in arity $p$ by the arity $p$ part of $ \mathscr{P}^{\leq 1}\circ^h_{\mathscr{P}^{\leq p}}\mathscr{P}^{\leq 1}$. Since
$$\begin{tikzcd}
\mathscr{P}^{\leq 1}\circ_{\mathscr{P}} \mathscr{P}^{\leq 1}\arrow[r] & \mathscr{P}^{\leq 1}\circ_{\mathscr{P}^{\leq p}} \mathscr{P}^{\leq 1}
\end{tikzcd}$$
is an equivalence in arity $\leq p$, is follows that we only have to add arity $p$ generators in degrees $\leq n(p)+1$ (together with generators of higher arity). This makes sure that we can arrange for condition \ref{it:resold}.

For the remaining generators that we have to add, note that $\mathscr{P}^{\leq p}$ satisfies the equivalent conditions of Proposition \ref{prop:generatorsforoperad}, since it is zero in arities $\geq p+1$. This implies that it suffices to add higher arity generators in increasingly negative degrees, so that we can arrange for \ref{it:resolc}.

Finally, define ${\mathscr{Q}}$ to be the limit of the tower. Since the tower becomes stationary in every fixed arity, it follows that ${\mathscr{Q}}$ is graded-free. Furthermore, one sees that the arity $p$ generators of ${\mathscr{Q}}$ are concentrated in degrees $\leq n(p)+1$, so they sit in increasingly negative degrees by the assumption on $\mathscr{P}$.
\end{proof}

{\small
\bibliographystyle{alpha}
\bibliography{biblio}
}

\end{document}